\documentclass[a4paper,11pt]{article}
\usepackage{latexsym}
\usepackage{cite}
\usepackage{algorithm}
\usepackage{amstext}
\usepackage{amsthm}
\usepackage{amssymb}
\usepackage{amsmath}
\usepackage{amscd}
\usepackage{color}
\usepackage{graphicx}
\usepackage[latin1]{inputenc}
\usepackage[english]{babel}
\usepackage{tabularx}
\usepackage{verbatim}
\usepackage{geometry}
\usepackage{enumerate}
\usepackage{xcolor}
\usepackage{subcaption}
\usepackage{hyperref}

\graphicspath{{Results/Phantom/}{Results/Micro/}{Results/MRI/}{Results/}{Results/pb1_FISTA/}{Results/pb1_GFISTA/}}

\theoremstyle{plain}
\newtheorem{definition}{Definition}[section]
\newtheorem{Theorem}{Theorem}[section]
\newtheorem{corollary}{Corollary}[section]
\newtheorem{assumptions}{Assumption}

\newcommand{\R}{\mathbb{R}}
\newcommand{\N}{\mathbb{N}}

\newtheorem{Lemma}{Lemma}[section]

\newtheorem{Remark}{Remark}[section]

  {\end{list}}

\makeatletter

  \gdef\listctr{list\romannumeral\the\@listdepth}\expandafter
  
\makeatother

\newenvironment{AlgorithmSteps}[1][1]{%
  \begin{list}{\csname label\listctr\endcsname}{%
      \usecounter{\listctr}
      
      \settowidth{\labelwidth}{\textsc{Step\ #1.}}%
      \setlength{\leftmargin}{\labelwidth}\addtolength{\leftmargin}{\labelsep}}}%
  {\end{list}}

\def\x{ {\ve x}}

\def\y{ {\ve y}}

\def\xk{ \x^{(k)}}

\def\xkk{ \x^{(k+1)}}

\def\ykk{ \y^{(k+1)}}

\def\dom{\mbox{dom}}
\def\R{\mathbb R}
\def\H{\mathcal H}

\def\prox{{\mbox{prox}}}

\def\x{ {x}}\def\y{ {y}}
\def\xkk{ {\x^{(k+1)}}} \def\ykk{ {\y^{(k+1)}}}
\def\xk{ \x^{(k)} }

\def\prox{\mathrm{prox}} 
\newcommand{\argmin}{\operatornamewithlimits{argmin}}
\def\dom{\mbox{dom}}

\title{Scaled, inexact and adaptive generalized FISTA \\for strongly convex optimization}

\date{}

\author{S. Rebegoldi (Dipartimento di Ingegneria Industriale, UniFi, Italy),\\
L. Calatroni (I3S, CNRS, UCA,  Inria, France).}

\begin{document}

\maketitle

\begin{abstract}
We consider a variable metric and inexact version of the FISTA-type algorithm considered in \cite{Chambolle-Pock-2016,Calatroni-Chambolle-2019} for the minimization of the sum of two (possibly strongly) convex functions.  The proposed algorithm is combined with an adaptive (non-monotone) backtracking strategy, which allows for the adjustment of the algorithmic step-size along the iterations in order to improve the convergence speed. 
We prove a linear convergence result for the function values, which depends on both the strong convexity moduli of the two  functions and the upper and lower bounds on the spectrum of the variable metric operators. We validate the proposed algorithm, named Scaled Adaptive GEneralized FISTA (SAGE-FISTA), on exemplar image denoising and deblurring problems where edge-preserving Total Variation (TV)  regularization is combined with Kullback-Leibler-type fidelity terms, as it is common in applications where signal-dependent Poisson noise is assumed in the data. 
\end{abstract}



\section{Introduction}
In recent years, inertial Forward--Backward (FB)  methods have become standard tools for solving composite convex optimization problems of the form
\begin{equation}\label{eq:main_probl}
\underset{x\in\mathcal{H}}{\min}\ F(x)\equiv f(x)+g(x),
\end{equation}
where $\mathcal{H}$ is an Hilbert space, $f,g:\mathcal{H}\rightarrow \mathbb{R}\cup \{\infty\}$ are convex functions and $f$ is continuously differentiable on a subset $Y\subseteq \mathcal{H}$. In its standard form, the FB algorithm \cite{Combettes-Wajs-2005,Combettes-Pesquet-2009} is defined as the sequential application of a (forward) gradient step on the smooth part $f$ followed by a (backward) proximal step on the non-smooth term $g$, thus resulting in the following iterative scheme defined for  all $k\geq 0$
\begin{equation}\label{eq:main_FB}
x^{(k+1)}=\operatorname{prox}_{\tau_k g}(x^{(k)}-\tau_k\nabla f(x^{(k)}))=\underset{z\in\mathcal{H}}{\operatorname{argmin}} \ g(z)+\frac{1}{2\tau_k}\|z-x^{(k)}+\tau_k\nabla f(x^{(k)})\|^2,
\end{equation}
where $\tau_k>0$ is the algorithmic step-size and $\operatorname{prox}_{\tau_k g}$ denotes the proximal operator associated to the convex function $\tau_k g$ \cite{Moreau1965}. In order to improve the convergence rate of \eqref{eq:main_FB} to a minimum point $x^*\in\argmin F$, one can consider the following inertial variant of the FB scheme \cite{Beck-Teboulle-2009b,Nesterov-1983}:
\begin{equation}\label{eq:main_inertial}
x^{(k+1)}=\operatorname{prox}_{\tau_{k+1} g}(y^{(k+1)}-\tau_{k+1}\nabla f(y^{(k+1)})), \quad \text{where } y^{(k+1)}=x^{(k)}+\beta_{k+1}(x^{(k)}-x^{(k-1)}). 
\end{equation}
Here, $y^{(k+1)}$ denotes the inertial point obtained by extrapolation of the sequence $\left\{x^{(k)}\right\}_{k\in\mathbb{N}}$ and $\beta_{k+1}$ is the inertial parameter. The idea of introducing inertia in \eqref{eq:main_FB} traces back to the pioneering work of Nesterov \cite{Nesterov-1983}, where the author proposes an inertial gradient method suited for the case $g\equiv 0$ with optimal convergence rate $\mathcal{O}(1/k^2)$, that is one order higher than the standard $\mathcal{O}(1/k)$ rate available for the method \eqref{eq:main_FB}. This approach has been extended by Nesterov in \cite{Nesterov-2004} to projected gradient descent algorithms and, later, to a general non-smooth convex function $g$ in \cite{Beck-Teboulle-2009b}, under the popular name of ``Fast  Iterative  Soft-Thresholding  Algorithm'' (FISTA). 
In spite of its optimal convergence rate, the practical acceleration of FISTA is often limited by the selection of the algorithmic step-sizes $\tau_k$, which, typically, are computed either as the inverse of a (pessimistic) estimate of the Lipschitz constant $L_f$ of the gradient of $f$, or through a monotone Armijo-type backtracking procedure when such an estimate is not available \cite[Section 4]{Beck-Teboulle-2009b}. Both choices usually generate a sequence of admissible step-sizes which, in case they are too small, may lead to a severe slowdown in the numerical performance of the algorithm. Possible remedies to this issue consist in either adopting an adaptive (non-monotone) backtracking procedure allowing for the local increasing of the step-size \cite{Nesterov-2013,Scheinberg-2014}, or computing the proximal--gradient operator in \eqref{eq:main_inertial} with respect to the norm induced by a suitably-chosen variable linear operator, which depends on some second-order information of the function $f$ at the current inertial point. Such variable metric strategy has been popularized in recent years and thoroughly studied in a series of works, such as, e.g., \cite{Bonettini-Loris-Porta-Prato-2015,Bonettini-Porta-Ruggiero-2016,Bonettini2018a,Combettes-Vu-2014}.    

Furthermore, if one or both functions $f,g$ in \eqref{eq:main_probl} are strongly convex, then it is possible to incorporate the strong convexity moduli inside the definition of the inertial parameters $\beta_{k+1}$ in order to obtain linear convergence rates. In the case of smooth optimization problems, this feature is well-known. In \cite[Section 2.2.1]{Nesterov-2004}, under the assumption that $f$ is $\mu_f-$strongly convex, $\nabla f$ is $L_f-$Lipschitz continuous and $g\equiv 0$, Nesterov considered the scheme \eqref{eq:main_inertial} with $\tau_{k+1}\equiv 1/L_f$ and $\beta_{k+1}\equiv (\sqrt{L_f}-\sqrt{\mu_f})/(\sqrt{L_f}+\sqrt{\mu_f})$, proving that \cite[Theorem 2.2.3]{Nesterov-2004}
$$
f(x^{(k)})-f(x^*)=\mathcal{O}\left(\left(1-\sqrt{\frac{\mu_f}{L_f}}\right)^k\right).
$$
 A first extension of Nesterov's algorithm to a general non-smooth $g$ is studied in \cite{Schmidt2011}, where strong convexity is assumed only on the smooth component $f$ . In \cite[Algorithm 5]{Chambolle-Pock-2016}, the authors consider a variant of this algorithm for problems where either one or both terms $f,g$ are strongly convex. The same algorithm is denominated Generalised FISTA (GFISTA) in \cite{Calatroni-Chambolle-2019} and equipped with an adaptive backtracking strategy allowing for local increasing of the step-size, similarly to the one proposed in \cite{Scheinberg-2014}. In \cite{Florea2019,Florea2020} similar algorithms are considered and thoroughly analyzed by means of generalized estimate sequences arguments. However, the convergence analysis carried out in \cite{Chambolle-Pock-2016,Calatroni-Chambolle-2019,Florea2019,Florea2020} does not treat explicitly the case where the proximal operator of $g$ is computed inexactly due to the lack of an exact analytical formula. This happens frequently in signal and image processing problems, especially in the presence of sparsity-inducing priors \cite{Bach-2012}. The practical performance of GFISTA may be in fact severely affected by inexact proximal evaluations, since inertial FB variants are usually more sensitive to computational errors than the standard FB scheme, see, e.g.~ \cite{Schmidt2011,Villa-etal-2013}. In this context, the adoption of a variable metric strategy in \eqref{eq:main_inertial} is expected to compensate for the proximal errors, thus favoring the recovery of the expected linear rate of GFISTA, as observed also for standard FISTA \cite{Bonettini2018a}.

\paragraph{Contribution and structure of the paper.}
In this paper, we consider an extension of the GFISTA algorithm \cite{Chambolle-Pock-2016,Calatroni-Chambolle-2019} suited for the minimization of composite (possibly strongly) convex functionals, which we name SAGE-FISTA (Scaled Adaptive GEneralized FISTA). The proposed algorithm computes the FB iterates inexactly with increasing accuracy and with respect to a sequence of suitably-chosen variable metric operators. In order to limit the computational slowdowns due to the use of possibly small step-sizes generated by Armijo-type monotone backtracking strategies, we further consider an adaptive backtracking procedure allowing for local adjustment of such parameter \cite{Scheinberg-2014,Calatroni-Chambolle-2019}. Under the assumption that one or both functions $f,g$ in \eqref{eq:main_probl} are strongly convex, we prove the linear convergence of the sequence of function values generated by SAGE-FISTA. Furthermore, we assess the numerical performance of SAGE-FISTA on some image restoration problems, showing that the combination of the scaling technique with the adaptive backtracking strategy significantly reduces the computational times and avoids misspecifications of the algorithmic step-size.

The paper is organized as follows. In Section \ref{sec:preliminaries} we set the notation and report some auxiliary technical results. In Section \ref{sec:2} we describe SAGE-FISTA in detail and provide a rigorous study of its convergence properties, showing, in particular, that linear convergence is achieved in strongly convex scenarios. Furthermore we report some useful results for the practical definition of both the inexactness parameter sequence and the variable metric operators. Finally, in Section \ref{sec:numerical}, we apply SAGE-FISTA to exemplar image restoration problems such as Poisson image denoising and deblurring. 

\section{Problem setting and preliminaries}  \label{sec:preliminaries}

In this section, we fix the notation, formulate the general optimization problem we aim to solve and prove some important preliminary results for the convergence analysis that follows.

\subsection{Notations}\label{subsec:0}
For a given Hilbert space $\H$, we denote by $\|\cdot\|$ the norm induced by the inner product $\langle \cdot, \cdot\rangle$ defined on $\H$. For any function $f:\mathcal{H}\rightarrow \R\cup\{\infty\}$, we denote the domain of $f$ by the set $\dom(f):=\{\x\in\H: \ f(\x)<\infty\}$. If $f:\mathcal{H}\rightarrow \mathbb{R}$ is a continuously differentiable function, then by $\mathbb{D}_f(x,y)=f(x)-f(y)-\langle \nabla f(y),x-y\rangle$ we denote the Bregman distance of $f$ between the two points $x, y\in\mathcal{H}$.\\
Let $\mathcal{S}(\H)$ be the set of linear, bounded and self-adjoint operators from $\H$ to $\H$, and $\mathcal{I}\in \mathcal{S}(\mathcal{H})$ the identity operator on $\mathcal{H}$. The Loewner partial ordering relation on $\mathcal{S}(\H)$ is defined as follows:
$$
\forall D_1, D_2 \in \mathcal{S}(\H) \quad D_1\preceq D_2 \Leftrightarrow \langle D_1 x, x \rangle \leq \langle D_2 x, x \rangle  \ \forall x \in \H \, .
$$
For any $\eta_{inf},\eta_{sup}$ with $0<\eta_{inf}\leq \eta_{sup}$, we introduce the following sets
\begin{align*}
\mathcal{D}_{\eta_{inf}}:=\{D\in \mathcal{S}(\H): \ \eta_{inf} \mathcal{I} \preceq D \},\quad
\mathcal{D}_{\eta_{inf}}^{\eta_{sup}}:=\{D\in \mathcal{S}(\H): \ \eta_{inf} \mathcal{I} \preceq D \preceq \eta_{sup} \mathcal{I} \}.
\end{align*}
By definition, it follows that $\mathcal{D}_{\eta_{inf}}^{\eta_{sup}}\subseteq \mathcal{D}_{\eta_{inf}}$. Furthermore, following \cite[Theorem 4.6.11]{Debnath-etal-1990}, we have that if $D\in\mathcal{D}_{\eta_{inf}}$, then $D$ is invertible and, if also $D\in\mathcal{D}_{\eta_{inf}}^{\eta_{sup}}$, then we have $D^{-1}\in\mathcal{D}_{1/\eta_{sup}}^{1/\eta_{inf}}$. If $D\in \mathcal{D}_{\eta_{inf}}$, then
\begin{equation}\label{eq:inner_productD}
(x,y):=\langle Dx, y \rangle,
\end{equation}
defines an inner product on $\mathcal{H}$, and $\|x\|_D=\sqrt{\langle Dx, x\rangle}$ denotes the $D-$norm induced by \eqref{eq:inner_productD}. Hence, if $D\in \mathcal{D}_{\eta_{inf}}^{\eta_{sup}}$, the following inequality holds
\begin{equation}\label{ine_norm}
\eta_{inf} \| x\|^2 \leq  \|x\|_D^2 \leq \eta_{sup} \| x \|^2, \quad \forall \ x\in\H.
\end{equation}
Given $Y\subseteq \mathcal{H}$ nonempty, closed convex set and $D\in \mathcal{D}_{\eta_{inf}}$, we define the projection operator onto $Y$ in the metric induced by $D$ as
$$
P_{Y,D}(x)=\underset{z\in Y}{\argmin} \ \|z-x\|^2_D, \quad \forall \ x\in\mathcal{H}.
$$

\subsection{Problem formulation and preliminary results} \label{subsec:2}

We are interested in solving the optimization problem
\begin{equation}\label{minf}
{\min_{\x\in\H}} \ F(\x) \equiv f(\x) + g(\x) \, ,
\end{equation}
where
\begin{itemize}
\item $f:\H\rightarrow \R$ is $\mu_f-$strongly convex ($\mu_f\geq 0$) and continuously differentiable with $L_f$~-Lipschitz continuous gradient on a nonempty, closed convex set $Y$ such that $\operatorname{dom}(g)\subseteq Y\subseteq \operatorname{dom}(f)$;
\item $g:\H\rightarrow \R\cup\{\infty\}$ is proper, $\mu_g-$strongly convex ($\mu_g\geq 0$), and lower semicontinuous.
\end{itemize}
We now prove a technical descent inequality that is crucial for the convergence analysis carried out in Section \ref{sec:2}. We start noting that, under the assumption that $f$ is $\mu_f-$strongly convex, then $f$ 
is also strongly convex with respect to the norm induced by any operator $D\in \mathcal{D}_{\eta_{inf}}^{\eta_{sup}}$.
\begin{Lemma}\label{lemmastronglyD}
Let $f:\mathcal{H}\rightarrow \R\cup\{\infty\}$ be a $\mu_f$-strongly convex function on a nonempty convex set $Y\subseteq \operatorname{dom}(f)$. For all $D\in \mathcal{D}_{\eta_{inf}}^{\eta_{sup}}$ and $\mu_{f,D}\leq \frac{\mu_f}{\eta_{sup}}$, we have that $f$ is $\mu_{f,D}-$strongly convex with respect to the $D-$norm, that is  
\begin{equation}\label{strongly-convexD-bis}
f(y)\geq  f(x)+\langle w, y-x\rangle+\frac{\mu_{f,D}}{2}\|y-x\|_D^2, \quad \forall \ y,x\in Y, \ w\in\partial f(x).
\end{equation}
\end{Lemma}
\begin{proof} In view of the assumption on $\mu_{f,D}$, the $\mu_f$-strong convexity of $f$ and \eqref{ine_norm}, we have
\begin{align*}
f(y)& \geq f(x)+\langle w, y-x\rangle +\frac{\mu_f}{2\eta_{sup}}\eta_{sup}\|y-x\|^2\\
&\geq f(x)+\langle w, y-x\rangle +\frac{\mu_{f,D}}{2}\|y-x\|_D^2, \quad \forall \ y,x\in Y, \ w\in\partial f(x).
\end{align*}
\end{proof}
Furthermore we report the statement of a scaled version of the classical descent lemma whose proof can be found in \cite{Bonettini-Porta-Ruggiero-2016}.
\begin{Lemma}\label{lemmadescent}
\cite[Lemma 6]{Bonettini-Porta-Ruggiero-2016} Let $f:\mathcal{H}\rightarrow \R\cup\{\infty\}$ be a continuously differentiable function with $L_f$-Lipschitz continuous gradient on $Y\subseteq \operatorname{dom}(f)$ and $D\in \mathcal{D}_{\eta_{inf}}$. If $0<\tau\leq \frac{\eta_{inf}}{L_f}$, we have
\begin{equation}\label{descent}
f(y)\leq f(x)+<\nabla f(x), y-x> +\frac{1}{2\tau} \|x-y\|_D^2 \quad \forall \ x,y\in Y.
\end{equation}
\end{Lemma}

We now introduce the notion of proximal--gradient point, which can be computed by alternating a (forward) gradient step performed w.r.t.~ the differentiable function $f$ with a (backward) proximal step on the non-smooth term $g$. To this aim, for any given $\bar{x}\in Y$, $\tau>0$ and $D\in \mathcal{D}_{\eta_{inf}}^{\eta_{sup}}$, we define the function $h_{\tau,D}(\cdot;\bar{x}):\H\rightarrow \R\cup\{\infty\}$ as
\begin{equation}\label{eq:function_h}
h_{\tau,D}(z;\bar{x}):=f(\bar{x})+ \langle\nabla f(\bar{x}),z-\bar{x}\rangle+\frac{1}{2\tau}\|z-\bar{x}\|_D^2+g(z), \quad \forall \ z\in\H.
\end{equation}
Since $g$ is strongly convex w.r.t. the $D$-norm with modulus $\mu_{g,D}\leq \mu_g/\eta_{sup}$ (see Lemma \ref{lemmastronglyD}), it is easy to show that $h_{\tau,D}(\cdot;\bar{x})$ is also strongly convex w.r.t. the $D$-norm with modulus $\frac{1}{\tau}+\mu_{g,D}$. Consequently, $h_{\tau,D}(\cdot;\bar{x})$ has a unique minimizer.

\begin{definition}
Given $\tau>0$ and $D\in \mathcal{D}_{\eta_{inf}}^{\eta_{sup}}$, we define the proximal operator associated to $g$ in the metric induced by $D$ as
$$
\operatorname{prox}_{g}^D(x)=\underset{z\in\mathcal{H}}{\operatorname{argmin}} \ g(z)+\frac{1}{2}\|z-x\|_D^2, \quad \forall \ x\in\mathcal{H}.
$$
For any $\bar{x}\in Y$, the proximal--gradient point $\hat{x}$ with parameters $\tau$ and $D$ is thus given by
\begin{equation}\label{eq:prox-grad}
\hat{x}:=\prox_{\tau g}^D(\bar{x}-\tau D^{-1}\nabla f(\bar{x}))=\underset{z\in\H}{\operatorname{argmin}} \ h_{\tau,D}(z;\bar{x}).
\end{equation}
\end{definition}

For our purposes, it will be convenient to consider a suitable approximation of the point $\hat{x}$ in \eqref{eq:prox-grad}  computed by means of a fixed positive tolerance parameter.
\begin{definition}  \label{def:eps_approx}
Given $\bar{x}\in Y$, $\tau>0$, $D\in \mathcal{D}_{\eta_{inf}}^{\eta_{sup}}$ and $\epsilon\geq 0$, a point $\tilde{x}\in\dom(g)$ is called an $\epsilon-$approximation of the proximal--gradient point $\hat{x}$ if the following condition holds
\begin{equation}\label{eq:eps_approx}
h_{\tau,D}(\tilde{x};\bar{x})-h_{\tau,D}(\hat{x};\bar{x})\leq \epsilon.
\end{equation}
In this case, we use the notation $\tilde{x}\approx_\epsilon \hat{x}$.
\end{definition}
We are now ready to state and prove the promised generalized descent inequality, adapted to the scaled and inexact framework introduced above.

\begin{Lemma}\label{lem:technical}
Given a point $\bar{x}\in Y$, $\tau>0$, $D\in \mathcal{D}_{\eta_{inf}}^{\eta_{sup}}$, $\epsilon\geq 0$ and a point $\tilde{x}\approx_\epsilon \hat{x}$, the following inequality holds
\begin{align}\label{eq:ine_fund}
F(\tilde{x})& +(1+\tau\mu_{g,D})\frac{\|x-\tilde{x}\|_D^2}{2\tau}+\left(\frac{\|\tilde{x}-\bar{x}\|_D^2}{2\tau}-\mathbb{D}_f(\tilde{x},\bar{x})\right) \nonumber\\ 
& \leq F(x)+(1-\tau\mu_{f,D})\frac{\|x-\bar{x}\|_D^2}{2\tau}
+\epsilon+\frac{\sqrt{2\epsilon\tau(1+\tau\mu_{g,D})}}{\tau}\|x-\tilde{x}\|_D,\quad \forall \ x\in \mathcal{H},
\end{align}
where $\mu_{g,D}\leq \frac{\mu_g}{\eta_{sup}}$ and $\mu_{f,D}\leq \frac{\mu_f}{\eta_{sup}}$.
\end{Lemma}

\begin{proof}
Let $\mu_{h,D}=\frac{1}{\tau}+\mu_{g,D}$ be the strong convexity modulus of $h_{\tau,D}(\cdot;\bar{x})$ with respect to the $D-$norm. Combining the strong convexity of $h_{\tau,D}(\cdot;\bar{x})$ with the optimality of the proximal--gradient point $\hat{x}$, we obtain
\begin{equation}\label{eq:starting_point}
h_{\tau,D}(x;\bar{x})\geq h_{\tau,D}(\hat{x};\bar{x})+\frac{\mu_{h,D}}{2}\|x-\hat{x}\|_D^2, \quad \forall x\in\H
\end{equation}
or, equivalently,
\begin{equation*}\label{eq:tech1}
\|x-\hat{x}\|_{D}^2\leq \frac{2}{\mu_{h,D}}(h_{\tau,D}(x;\bar{x})-h_{\tau,D}(\hat{x};\bar{x})), \quad \forall x\in\H.
\end{equation*}
Choosing $x=\tilde{x}$ above and combining it with \eqref{eq:eps_approx} leads to
\begin{equation}\label{eq:tech2}
\|\tilde{x}-\hat{x}\|_D^2\leq \frac{2}{\mu_{h,D}}\epsilon.
\end{equation}
Starting again from \eqref{eq:starting_point}, we can thus write the following chain of inequalities
\begin{align*}
h_{\tau,D}(x;\bar{x})&\geq h_{\tau,D}(\hat{x};\bar{x})+\frac{\mu_{h,D}}{2}\|x-\hat{x}\|_D^2\\
&=h_{\tau,D}(\tilde{x};\bar{x})+h_{\tau,D}(\hat{x};\bar{x})-h_{\tau,D}(\tilde{x};\bar{x})+\frac{\mu_{h,D}}{2}\|x-\hat{x}\|_D^2\\
&\geq h_{\tau,D}(\tilde{x};\bar{x})-\epsilon+\frac{\mu_{h,D}}{2}\|x-\tilde{x}\|_D^2+\frac{\mu_{h,D}}{2}\|\tilde{x}-\hat{x}\|_D^2+\mu_{h,D}\langle D(x-\tilde{x}), \tilde{x}-\hat{x}\rangle\\
&\geq h_{\tau,D}(\tilde{x};\bar{x})+\frac{\mu_{h,D}}{2}\|x-\tilde{x}\|_D^2-\epsilon-\mu_{h,D}\sqrt{\frac{2\epsilon}{\mu_{h,D}}}\|x-\tilde{x}\|_D, \quad \forall x\in\mathcal{H}
\end{align*}
where the third inequality follows from the definition of $\epsilon-$approximation \eqref{eq:eps_approx} and the three-point-equality $\|a-c\|_D^2=\|a-b\|_D^2+\|b-c\|_D^2+2\langle D(a-b),b-c\rangle $ and the fourth one from the Cauchy-Schwarz inequality, combined with relation \eqref{eq:tech2} and the non-negativity of the quantity $\|\tilde{x}-\hat{x}\|_D^2$. Hence, we deduce
\begin{equation}\label{eq:key}
h_{\tau,D}(x;\bar{x})\geq h_{\tau,D}(\tilde{x};\bar{x})+\frac{\mu_{h,D}}{2}\|\tilde{x}-x\|^2-\epsilon-\mu_{h,D}\sqrt{\frac{2\epsilon}{\mu_{h,D}}}\|x-\tilde{x}\|_D, \quad \forall x\in\H.
\end{equation}
Now, for all $x\in\H$, we have
\begin{align*}
& F(x)+(1-\tau\mu_{f,D})\frac{\|x-\bar{x}\|_D^2}{2\tau}\geq g(x)+f(\bar{x})+\langle \nabla f(\bar{x}),x-\bar{x}\rangle+\frac{\|x-\bar{x}\|_D^2}{2\tau}=h_{\tau,D}(x;\bar{x})\\
&\overset{\eqref{eq:key}}{\geq } h_{\tau,D}(\tilde{x};\bar{x})+\frac{\mu_{h,D}}{2}\|\tilde{x}-x\|^2-\epsilon-\mu_{h,D}\sqrt{\frac{2\epsilon}{\mu_{h,D}}}\|x-\tilde{x}\|_D\\
&= g(\tilde{x})+f(\bar{x})+\langle \nabla f(\bar{x}),\tilde{x}-\bar{x}\rangle+\frac{\|\tilde{x}-\bar{x}\|_D^2}{2\tau}+\frac{\tau\mu_{g,D}+1}{2\tau}\|\tilde{x}-x\|_D^2-\epsilon-\mu_{h,D}\sqrt{\frac{2\epsilon}{\mu_{h,D}}}\|x-\tilde{x}\|_D\\
&=F(\tilde{x})+\frac{\|\tilde{x}-\bar{x}\|_D^2}{2\tau}-\mathbb{D}_f(\tilde{x},\bar{x})+\frac{\tau\mu_{g,D}+1}{2\tau}\|\tilde{x}-x\|_D^2-\epsilon-\sqrt{2\epsilon\left(\frac{1}{\tau}+\mu_{g,D}\right)}\|x-\tilde{x}\|_D,
\end{align*}
which concludes the proof.
\end{proof}

\section{SAGE-FISTA: description and convergence analysis}\label{sec:2}

In the following, we describe a generalized version of the popular ``Fast Iterative Soft-Thresholding Algorithm'' (FISTA) proposed in \cite{Beck-Teboulle-2009b}, which is suited for solving the (possibly strongly) convex problem \eqref{minf} by means of an appropriate scaled and inexact inertial forward-backward splitting.



We name our algorithm SAGE-FISTA (Scaled Adaptive GEneralized FISTA); it is obtained by incorporating in the iterative scheme \eqref{eq:main_inertial} the following features:
\begin{itemize}
    \item the use of a variable non-Euclidean metric in the computation of the proximal--gradient point, which is induced along the iterations by the elements of a sequence of linear, bounded and self-adjoint positive operators $\{D_k\}_{k\in\mathbb{N}}$; such operators are typically chosen in order to capture some second order information of the differentiable part $f$ at the current iterate $x^{(k)}$ (see e.g. \cite{Bonettini-etal-2009,Bonettini2018a,Bonettini2019,Lanteri-etal-2001});  
    \item the approximate computation of the proximal-gradient point according to the notion of $\epsilon$~-approximation given in Definition \ref{def:eps_approx};
    \item a non-monotone backtracking strategy similar to the one considered in \cite{Calatroni-Chambolle-2019}, which allows for possible increasing of the step-size $\tau_{k+1}$ at each iteration; differently from monotone Armijo-type approaches, this strategy is  particularly helpful when the initial $\tau_0$ is chosen to be extremely small, which corresponds to a (pessimistic) large estimate $L_0$ of the Lipschitz constant value $L_f$;
    \item a novel selection rule for the inertial parameter $\beta_{k+1}$, which depends on both the strong convexity moduli $\mu_f,\mu_g$ and the bounds on the spectrum of the operators $D_{k+1}$. 
\end{itemize}
The proposed SAGE-FISTA is fully reported in Algorithm \ref{algo:GFISTA}.

\begin{algorithm}[H]
\caption{Scaled Adaptive GEneralized FISTA (SAGE-FISTA)}\label{algo:GFISTA}
{\bf Parameters:} $\rho\in (0,1)$, $\delta\in(0,1]$, $\{\eta_{sup}^k\}_{k\in\mathbb{N}}$ s.t. $0< \eta_{inf}\leq \eta_{sup}^k\leq \eta_{sup}$, $\{\mu_{f,k}\}_{k\in\mathbb{N}}$, $\{\mu_{g,k}\}_{k\in\mathbb{N}}$, $\{\mu_{k}\}_{k\in\mathbb{N}}$ s.t. $\mu_{f,k}= \frac{\mu_f}{\eta_{sup}^k}$, $\mu_{g,k}= \frac{\mu_g}{\eta_{sup}^k}$, $\mu_{k}=\mu_{f,k}+\mu_{g,k}$. \\
{\bf Initialization:} $\tau_0\in\mathbb{R}$ s.t. $0<\tau_0\mu_{f,0}<1$, $q_0=\frac{\tau_0\mu_{0}}{1+\tau_0\mu_{g,0}}$, $x^{(-1)}\in\mathcal{H}$, $x^{(0)}=x^{(-1)}$, and $t_0\in\R$ s.t. $1\leq t_0\leq \frac{1}{\sqrt{q_0}}$.\\

\noindent \textsc{For $k=0,1,2,\ldots$}
\begin{itemize}
\item[] Choose $D_{k+1}\in\mathcal{D}_{\eta_{inf}}^{\eta^k_{sup}}$ and set
$$
\tau_{k+1}^0=\frac{\tau_k}{\delta}.
$$

\noindent \textsc{For $i=0,1,\ldots$ repeat}
\begin{AlgorithmSteps}[6]
\item[1] Set $\tau_{k+1} = \rho^i \tau_{k+1}^0$;
\item[2] Set $q_{k+1}=\frac{\tau_{k+1}\mu_{k+1}}{1+\tau_{k+1}\mu_{g,k+1}}$;
\item[3] Set $t_{k+1} = \frac{1-q_kt_k^2+\sqrt{(1-q_kt_k^2)^2+4\frac{q_k}{q_{k+1}}t_k^2}}{2}$.
\item[4]\label{step2} Set $\beta_{k+1}=\left(\frac{t_k-1}{t_{k+1}}\right)\frac{1+\tau_{k+1}\mu_{g,k+1}-t_{k+1}\tau_{k+1}\mu_{k+1}}{1-\tau_{k+1}\mu_{f,k+1}}$.
\item[5] Set $y^{(k+1)}=P_{Y,D_{k+1}}(\xk+\beta_{k+1}(\xk-x^{(k-1)}))$.
\item[6] Set $\hat{x}^{(k+1)}=\prox_{\tau_{k+1} g}^{D_{k+1}}(y^{(k+1)}-\tau_{k+1}D_{k+1}^{-1}\nabla f(y^{(k+1)}))$, choose $\epsilon_{k+1}\geq 0$ and compute $\xkk\in\dom(g)$ such that
\begin{equation*}\label{eq:iterate}
\xkk\approx_{\epsilon_{k+1}} \hat{x}^{(k+1)}.
\end{equation*}
\end{AlgorithmSteps}
\noindent \textsc{Until $\mathbb{D}_f(x^{(k+1)},y^{(k+1)})< \frac{\| x^{(k+1)}-y^{(k+1)}\|_{D_{k+1}}^2}{2\tau_{k+1}}$}.
\end{itemize}
\noindent \textsc{End}
\end{algorithm}

At each iteration $k\geq 0$, the preliminary step of SAGE-FISTA is the choice of both the linear, bounded and self-adjoint operator $D_{k+1}\in\mathcal{D}_{\eta_{inf}}^{\eta^k_{sup}}$ and the tentative step-size $\tau_{k+1}^0=\tau_k/\delta$. If $\delta<1$, then an adaptive increasing of $\tau_k$ is allowed (see \cite{Calatroni-Chambolle-2019,Florea2020} for analogous strategies applied to GFISTA and \cite{Scheinberg-2014} for FISTA), while if $\delta=1$, a classical Armijo-type backtracking is performed (as in \cite{Beck-Teboulle-2009b} for FISTA). Then, an inner backtracking procedure computing the next iterate $x^{(k+1)}$ follows. At each inner backtracking iteration $i\geq 0$, the iterates $q_{k+1}$, $t_{k+1}$ are first updated ({\sc STEPS 2--3}); such quantities depend not only on the scaled quantities $\mu_{f,k+1}=\mu_f/\eta_{sup}^k$, $\mu_{g,k+1}=\mu_g/\eta_{sup}^k$ and $\mu_{k+1}=\mu_{f,k+1}+\mu_{g,k+1}$, but also on the tentative step-size $\tau_{k+1}$. Then, the inertial parameter $\beta_{k+1}$ and the projected inertial point $y^{(k+1)}$ are computed ({\sc STEPS 4--5}). Finally, an approximated point $x^{(k+1)}$ such that $x^{(k+1)}\approx_{\epsilon_{k+1}} \prox_{\tau_{k+1} g}^{D_{k+1}}(y^{(k+1)}-\tau_{k+1}D_{k+1}^{-1}\nabla f(y^{(k+1)}))$ is computed. For each backtracking iteration, a check on the condition 
\begin{equation}\label{eq:backtracking}
\mathbb{D}_{f}(x^{(k+1)},y^{(k+1)})< \frac{\|x^{(k+1)}-y^{(k+1)}\|_{D_{k+1}}^2}{2\tau_{k+1}}
\end{equation}
is performed ({\sc STEP 6}). If such condition is not satisfied, then the step-size is reduced by a factor $\rho$ ({\sc STEP 1}) and  {\sc STEPS 2--6} are performed until the condition is satisfied. Note that this procedure ends in a finite number of steps, thanks to Lemma \ref{lemmadescent}. 

When $\eta_{inf}^k=\eta_{sup}^k=1$, $D_k\equiv \mathcal{I}$ and $Y=\mathcal{H}$, Algorithm \ref{algo:GFISTA} reduces to the Generalised FISTA (GFISTA) algorithm studied in \cite{Calatroni-Chambolle-2019}. The presence of the strong convexity moduli $\mu_f,\mu_g$ in the update rule of the inertial parameter $\beta_{k+1}$ allows to prove a linear convergence rate for the function values also when an adaptive backtracking strategy is adopted, see \cite[Theorem 4.6]{Calatroni-Chambolle-2019}.

When $D_k\neq \mathcal{I}$, Algorithm \ref{algo:GFISTA} can be considered as a scaled version of GFISTA where the inexact computation of the proximal-gradient point is explicitly taken into account in \textsc{STEP 6}. Note that the introduction of the positive operator $D_k$ further imposes a modification of the parameters $q_{k+1}$ and $t_{k+1}$ in \textsc{STEP 2} and \textsc{STEP 3}, respectively. However, linear convergence rates still hold, as we will rigorously show in Section \ref{sec:2}.

Finally, note that when $\mu=\mu_f+\mu_g=0$ and $\delta=1$, {\sc STEP 3} reduces to
$$
t_{k+1}=\frac{1+\sqrt{1+4\frac{\tau_k}{\tau_{k+1}}t_k^2}}{2},
$$
and Algorithm \ref{algo:GFISTA} becomes the inexact scaled forward-backward extrapolation method equipped with Armijo-type backtracking proposed in \cite{Bonettini2018a}. In this case, standard $\mathcal{O}(1/k^2)$ convergence rates can be proved, coherently with the result obtained in \cite[Theorem 3.1]{Bonettini2018a}.

\begin{Remark}
We observe that
$$
0\leq q_{k}< 1, \quad t_{k}\geq 1, \quad \forall \ k\geq 1.
$$
Indeed $q_k$ is nonnegative due to the nonnegativity of $\mu_k$, $\mu_g$, $\tau_k$, and combining the $\mu_{f,k}-$strong convexity of $f$ with respect to the $D_{k}-$norm with \eqref{eq:backtracking}, we have
\begin{equation*}
\frac{\mu_{f,k}}{2}\|x^{(k)}-y^{(k)}\|_{D_{k}}^2\leq \mathbb{D}_f(x^{(k)},y^{(k)})< \frac{\| x^{(k)}-y^{(k)}\|_{D_{k}}^2}{2\tau_{k}},
\end{equation*}
which implies $\mu_{f,k}\tau_k< 1$ and thus $q_k< 1$. Furthermore, proceeding as in \cite[Lemma 4.4]{Calatroni-Chambolle-2019}, we have 
\begin{align*}  
t_{k} 
\geq \frac{1-q_{k-1}t^2_{k-1} + \sqrt{\left(1-q_{k-1}t^2_{k-1} \right)^2 + 4 q_{k-1}t^2_{k-1}}} {2} 
 =   \frac{1-q_{k-1}t^2_{k-1} + \sqrt{\left(1+q_{k-1}t^2_{k-1}\right)^2}} {2}=1. \notag
\end{align*}
\end{Remark}
\begin{Remark}
When $\mu=\mu_f+\mu_g>0$, the update rule for the sequence $\left\{ t_k \right\}_k$ is obtained by imposing the following condition at each iteration:
\begin{equation}  \label{eq:choice_tk}
\tau'_{k+1}t_{k+1}(t_{k+1}-1) = \frac{\mu_k}{\mu_{k+1}}\omega_{k+1}\tau'_k t^2_k, \quad \forall \ k\geq 0
\end{equation}
where we have set:
\begin{equation} \label{eq:tau_prime_omega_}
\tau_k':= \frac{\tau_k}{1+\tau_k\mu_{g,k}}, \quad \omega_{k}:=1-t_{k}q_{k} = \frac{1+\tau_k\mu_{g,k}-\tau_k t_{k}\mu_k}{1+\tau_k\mu_{g,k}}.
\end{equation}
Indeed, note that $t_{k+1}$ is defined as the solution of the following second-degree equation
$$
t^2-(1-q_kt_k^2)t-\frac{q_k}{q_{k+1}}t_k^2=0.
$$
By now observing that $q_k/q_{k+1}=(\mu_k\tau_k')/(\mu_{k+1}\tau_{k+1}')$, we can rewrite the equation above as
$$
t^2-\left(1-\frac{\mu_k\tau_k't_k^2}{\mu_{k+1}\tau_{k+1}'}q_{k+1}\right)t-\frac{\mu_k\tau_k't_k^2}{\mu_{k+1}\tau_{k+1}'}=0.
$$
Then, choosing $t=t_{k+1}$, multiplying by $\tau_{k+1}'$ and rearranging terms leads to \eqref{eq:choice_tk}.
When $\mu=0$, equality \eqref{eq:choice_tk} holds by neglecting the quantity $\mu_k/\mu_{k+1}$. 
\end{Remark}

We conclude this section with the following technical result holding for the sequences $\{q_kt_k^2\}_{k\in\mathbb{N}}$ and $\{q_kt_k\}_{k\in\mathbb{N}}$.

\begin{Lemma}\label{rem:2}
If $0\leq t_0\leq \frac{1}{\sqrt{q_0}}$ (as assumed in Algorithm \ref{algo:GFISTA}), then it follows that
\begin{equation}\label{eq:qt}
q_kt_{k}^2\leq 1 \quad\text{ and }\quad \ q_k t_{k}< 1, \quad \forall \ k\geq 0.
\end{equation}
\end{Lemma}
\begin{proof}
If $\mu=\mu_f+\mu_g=0$, the thesis follows trivially. Then, we assume $\mu>0$ and proceed by induction as in \cite[Lemma 4.5]{Calatroni-Chambolle-2019}. For $k=0$, the assumption on  $t_0$ reads $q_0 t_0^2\leq 1$, which entails $q_0 t_0\leq \sqrt{q_0}< 1$, so the thesis holds. Then, let us assume that \eqref{eq:qt} holds for some $k\geq 0$. Recalling that $t_{k+1}$ is the solution of \eqref{eq:choice_tk}, we have:
\begin{equation*}
q_{k+1} t^2_{k+1} = q_{k+1} t_{k+1}  + q_{k+1}\omega_{k+1}\frac{\mu_k\tau'_k}{\mu_{k+1}\tau'_{k+1}}t^2_k = q_{k+1} t_{k+1} + \omega_{k+1}q_kt^2_k = 1+ \omega_{k+1}(q_kt^2_k -1 ) \leq 1,
\end{equation*}
which follows by simply applying the induction assumption. Thus, we get $q_{k+1}t_{k+1}^2\leq 1$. Notice that the same chain of equalities also implies:
\[
q_{k+1} t^2_{k+1}  = q_{k+1} t_{k+1} + (1-q_{k+1}t_{k+1})t^2_k q_{k}.
\]
By contradiction, if $q_{k+1}t_{k+1}\geq 1$, then the previous inequality would imply
$$
q_{k+1}t_{k+1}^2\leq  q_{k+1}t_{k+1} \quad \Leftrightarrow \quad q_{k+1}t_{k+1}\leq q_{k+1}< 1,
$$
which is absurd. Thus, we deduce $q_{k+1}t_{k+1}< 1$, as required.
\end{proof}

\subsection{Convergence rates}
We now  show that SAGE-FISTA enjoys a convergence rate in function values that is either $R-$linear or proportional to $1/k^2$, depending on whether the strong convexity modulus $\mu=\mu_f+\mu_g$ is positive or null. A couple of technical results holding for nonnegative sequences are needed.

\begin{Lemma}\label{key_lemma}
\cite[Lemma 1]{Schmidt2011} Let $\{p_k\}_{k\in\N}$, $\{q_k\}_{k\in\N}$, $\{\lambda_k\}_{k\in\N}$ be sequences of real nonnegative numbers, with $\{q_k\}_{k\in\N}$ being a monotone nondecreasing sequence, such that the following recursive property is satisfied:
\begin{equation}\label{eq:tech3}
p_k^2\leq q_k+\sum_{i=1}^{k}\lambda_ip_i, \quad \forall \ k\geq 1.
\end{equation}
Then we have
\begin{equation}\label{eq:key_ineq}
p_k\leq \frac{1}{2}\sum_{i=1}^{k}\lambda_i+\left(q_k+\left(\frac{1}{2}\sum_{i=1}^{k}\lambda_i\right)^2\right)^{\frac{1}{2}}, \quad \forall \ k\geq 1.
\end{equation}
\end{Lemma}

\begin{Lemma}\label{key_lemma_2}\cite{Polyak-1987}
Let $\{p_k\}_{k\in\mathbb{N}}$, $\{\zeta_k\}_{k\in\mathbb{N}}$, $\{\xi_k\}_{k\in\mathbb{N}}$ be sequences of real nonnegative numbers such that $\sum_{k=0}^{\infty}\zeta_k <\infty$, $\sum_{k=0}^{\infty}\xi_k <\infty$ and the following inequality is satisfied:
$$
p_{k+1}\leq (1+\zeta_k)p_k+\xi_k, \quad \forall \ k\geq 0.
$$
Then $\{p_k\}_{k\in\mathbb{N}}$ is a convergent sequence.
\end{Lemma}

Furthermore, we make the following assumption on the sequence of operators $\{D_k\}_{k\in\mathbb{N}}$ and the sequence of upper bounding values $\{\eta_{sup}^k\}_{k\in\mathbb{N}}$.

\begin{assumptions}\label{ass:1}
There exists a sequence of real nonnegative numbers $\{\gamma_k\}_{k\in\mathbb{N}}$ s.t. $\sum_{k=0}^\infty\gamma_k<\infty$ and, for all $k\geq 0$, the following conditions hold  
\begin{eqnarray}
D_{k+1}&\preceq&(1+\gamma_{k+1})D_k,\label{eq:Dk_cond}\\
\frac{\eta_{sup}^{k+1}}{\eta_{sup}^k}&\leq & 1+\gamma_{k+1}.\label{eq:etak_cond}
\end{eqnarray}
\end{assumptions}

\begin{Remark}\label{eq:rem3}
Condition \eqref{eq:Dk_cond} has been often used to prove the convergence of variable metric FB methods in the convex setting, see for instance \cite{Combettes-Vu-2014,Bonettini-Loris-Porta-Prato-2015,Bonettini-Porta-Ruggiero-2016,Bonettini2018a}.
It is easy to see that \eqref{eq:Dk_cond} implies \eqref{eq:etak_cond} for the particular choice $\eta_{sup}^k=\|D_k^{\frac{1}{2}}\|^2$, with $D_k^{\frac{1}{2}}$ being the square root operator applied to $D_k$; when $\|\cdot\|=\|\cdot\|_2$, this amounts to choosing $\eta_{sup}^k=\|D_k\|_2$, i.e. the spectral radius of $D_k$. However, in general, we need to impose condition \eqref{eq:etak_cond} explicitly.
Let us mention two particular examples of sequences $\{D_k\}_{k\in\mathbb{N}}$ and $\{\eta_{sup}^k\}_{k\in\mathbb{N}}$ satisfying Assumption \ref{ass:1}, which we will use in the following numerical experiments.
\begin{itemize}
\item Suppose $D_k\equiv D$ and $\eta_{sup}^k\equiv \eta>0$, with $D\in \mathcal{D}^{\eta}_{\eta_{inf}}$. Then Assumption \ref{ass:1} holds with $\gamma_k\equiv 0$.
\item Suppose $D_k\in\mathcal{D}_{\eta_{inf}^k}^{\eta_{sup}^k}$ for all $k\geq 0$, and the upper and lower bounds converge to the same positive value at a sufficiently fast rate, that is
\begin{equation}\label{eq:eta_def}
\eta_{inf}^k=\eta-\nu_{inf}^k, \quad \eta_{sup}^k=\eta+\nu_{sup}^k,
\end{equation}
where $0\leq\nu_{inf}^k< \eta$, $\nu_{sup}^k\geq 0$, and $\sum_{k=0}^{\infty}\nu_{inf}^k<\infty$, $\sum_{k=0}^{\infty}\nu_{sup}^k<\infty$. In this case, recalling \eqref{ine_norm}, we can write
\begin{equation}\label{eq:scal_cond}
\|x\|^2_{D_{k+1}}\leq \eta_{sup}^{k+1}\|x\|^2= \frac{\eta_{sup}^{k+1}}{\eta_{inf}^{k}}\eta_{inf}^{k}\|x\|^2\leq \frac{\eta_{sup}^{k+1}}{\eta_{inf}^{k}}\|x\|_{D_{k}}^2, \quad \forall \ x\in\mathcal{H}.
\end{equation} 
Thanks to \eqref{eq:eta_def}, the factor multiplying $\|x\|_{D_k}^2$ can be rewritten as
\begin{equation*}
\frac{\eta_{sup}^{k+1}}{\eta_{inf}^{k}}= 1+\gamma_{k+1}, \quad \text{where }\quad \gamma_{k+1}=\frac{\nu_{sup}^{k+1}+\nu_{inf}^{k}}{\eta-\nu_{inf}^{k}}.
\end{equation*}
Since $\gamma_{k+1}\leq (\eta-\max_{k}\nu^k_{inf})^{-1}(\nu_{sup}^{k+1}+\nu_{inf}^{k})$ and the sequences $\{\nu^k_{inf}\}_{k\in\mathbb{N}}$, $\{\nu^k_{sup}\}_{k\in\mathbb{N}}$ are summable, we can conclude that condition \eqref{eq:Dk_cond} is satisfied. Furthermore, we also have
\begin{equation*}
\frac{\eta_{sup}^{k+1}}{\eta_{sup}^k}=\frac{\eta_{sup}^{k+1}}{\eta_{inf}^k}\cdot\frac{\eta_{inf}^k}{\eta_{sup}^k}\leq \frac{\eta_{sup}^{k+1}}{\eta_{inf}^k} = 1+\gamma_{k+1},
\end{equation*}
that is condition \eqref{eq:etak_cond} also holds.\\
Note that, if $\mathcal{H}=\mathbb{R}^n$, $\|\cdot\|=\|\cdot\|_2$, and $\{D_k\}_{k\in\mathbb{N}}$ are diagonal matrices, then we can always impose condition \eqref{eq:eta_def} by constraining the diagonal elements of $D_k$ in the interval $[\eta-\nu_{inf}^k,\eta+\nu_{sup}^k]$ (see Section \ref{sec:numerical} for a practical implementation). By doing so, we progressively ``squeeze'' the scaling matrices to a multiple of the identity matrix as the iterations increase.
\end{itemize}
\end{Remark}

We now provide a key descent inequality that holds for the iterates of SAGE-FISTA Algorithm \ref{algo:GFISTA}, in the same spirit of the one employed in \cite[p. 8]{Calatroni-Chambolle-2019}. The result follows by combining Lemma \ref{lem:technical} with the backtracking condition of Algorithm \ref{algo:GFISTA} and Assumption \ref{ass:1}.  

\begin{Lemma}\label{lem:fund}
Suppose that $F=f+g$ is $\mu-$strongly convex with $\mu\geq 0$ and Assumption \ref{ass:1} holds. Then, for all $k \geq 0$, we have
\begin{align*}
& \tau_{k+1}'t_{k+1}^2(F(\xkk)-F(x^*))+\frac{1}{2}\|x^*-\xkk-(t_{k+1}-1)(\xkk-\xk)\|_{D_{k+1}}^2 \nonumber \\
&  \leq \omega_{k+1}(1+\gamma_{k+1})\left(\tau_k't_k^2(F(\xk)-F(x^*)) +\frac{1}{2}\|x^*-\xk-(t_k-1)(\xk-x^{(k-1)})\|_{D_{k}}^2\right)\nonumber\\
&+\tau_{k+1}'t_{k+1}^2\epsilon_{k+1}+t_{k+1}\sqrt{2\epsilon_{k+1}\tau'_{k+1}}\|x^*-\xkk-(t_{k+1}-1)(\xkk-\xk)\|_{D_{k+1}}.
\end{align*}
\end{Lemma}

\begin{proof}
We apply Lemma \ref{lem:technical} with the choices $\tau=\tau_{k+1}$, $\bar{x}=\ykk$, $\tilde{x}=\xkk$, $\epsilon=\epsilon_{k+1}$, $D=D_{k+1}$, $\mu_{f,D}=\mu_{f,k+1}$, $\mu_{g,D}=\mu_{g,k+1}$ and combine it with the backtracking condition \eqref{eq:backtracking} so as to obtain:
\begin{align*}
F(\xkk)& +(1+\tau_{k+1}\mu_{g,k+1})\frac{\|x-\xkk\|_{D_{k+1}}^2}{ 2\tau_{k+1}} \leq F(x)+(1-\tau_{k+1}\mu_{f,k+1})\frac{\|x-\ykk\|_{D_{k+1}}^2}{2\tau_{k+1}}\\
&+\epsilon_{k+1}+\frac{\sqrt{2\epsilon_{k+1}\tau_{k+1}(1+\tau_{k+1}\mu_{g,k+1})}}{\tau_{k+1}}\|x-\xkk\|_{D_{k+1}}. \notag
\end{align*}
Let us now define $x=\frac{(t_{k+1}-1)\xk+x^*}{t_{k+1}}$ with $\xk$ being the $k-$th iterate of the algorithm and $x^*\in\mathcal{H}$ a solution of \eqref{minf}, and set $\bar{y}^{(k+1)}=\xk+\beta_{k+1}(x^{(k)}-x^{(k-1)})$. Since the operator $P_{Y,D_{k+1}}$ is firmly non-expansive \cite[Remark 2.2]{Bonettini2018a} and $x\in Y$, we have $\|x-y^{(k+1)}\|_{D_{k+1}}^2\leq \|x-\bar{y}^{(k+1)}\|_{D_{k+1}}^2$. Then, the previous inequality becomes
\begin{align}\label{eq:ine_fund_fund_bis}
F(\xkk)&+(1+\tau_{k+1}\mu_{g,k+1})\frac{\|(t_{k+1}-1)\xk+x^*-t_{k+1}\xkk\|_{D_{k+1}}^2}{2\tau_{k+1} t_{k+1}^2}\leq F\left(\frac{(t_{k+1}-1)\xk +x^*}{t_{k+1}}\right)\nonumber\\
&+(1-\tau_{k+1}\mu_{f,k+1})\frac{\|(t_{k+1}-1)\xk+x^*-t_{k+1}\bar{y}^{(k+1)}\|_{D_{k+1}}^2}{2\tau_{k+1} t_{k+1}^2}\nonumber\\
&+\epsilon_{k+1}+\frac{\sqrt{2\epsilon_{k+1}\tau_{k+1}(1+\tau_{k+1}\mu_{g,k+1})}}{\tau_{k+1} t_{k+1}}\|(t_{k+1}-1)\xk+x^*-t_{k+1}\xkk\|_{D_{k+1}}.
\end{align}
We now multiply both sides of \eqref{eq:ine_fund_fund_bis} by $t_{k+1}^2$, apply the $\mu_{k+1}-$strong convexity of $F$ with respect to the $D_{k+1}-$norm on the right-hand side (the point $\frac{(t_{k+1}-1)\xk +x^*}{t_{k+1}}$ is a convex combination of $\xk$ and $x^*$) and subtract the term $t_{k+1}^2F(x^*)$ to both sides, thus obtaining
\begin{align}\label{eq:intermediate}
t_{k+1}^2&(F(\xkk)-F(x^*))+\frac{1+\tau_{k+1}\mu_{g,k+1}}{2\tau_{k+1}}\|(t_{k+1}-1)\xk+x^*-t_{k+1}\xkk\|_{D_{k+1}}^2\nonumber\\
&\leq t_{k+1}(t_{k+1}-1)(F(\xk)-F(x^*))-\frac{\mu_{k+1}(t_{k+1}-1)}{2}\|\xk-x^*\|_{D_{k+1}}^2\nonumber\\
&+(1-\tau_{k+1}\mu_{f,k+1})\frac{\|(t_{k+1}-1)\xk+x^*-t_{k+1}\bar{y}^{(k+1)}\|_{D_{k+1}}^2}{2\tau_{k+1}}\nonumber\\
&+t_{k+1}^2\epsilon_{k+1}+t_{k+1}\frac{\sqrt{2\epsilon_{k+1}\tau_{k+1}(1+\tau_{k+1}\mu_{g,k+1})}}{\tau_{k+1}}\|(t_{k+1}-1)\xk+x^*-t_{k+1}\xkk\|_{D_{k+1}}.
\end{align}
By proceeding similarly as in \cite[p. 292]{Chambolle-Pock-2016}, we get:
\begin{align*}
& t_{k+1}^2(F(\xkk)-F(x^*))+\frac{1+\tau_{k+1}\mu_{g,k+1}}{2\tau_{k+1}}\|x^*-\xkk-(t_{k+1}-1)(\xkk-\xk)\|_{D_{k+1}}^2\nonumber\\
&+ \frac{t_{k+1}^2(t_{k+1}-1)\mu_{k+1}(1-\tau_{k+1}\mu_{f,k+1})}{2(1+\tau_{k+1}\mu_{g,k+1}-t_{k+1}\tau_{k+1}\mu_{k+1})}\|\xk-\bar{y}^{(k+1)}\|_{D_{k+1}}^2 \leq t_{k+1}(t_{k+1}-1)(F(\xk)-F(x^*))\nonumber \\
&+\frac{1+\tau_{k+1}\mu_{g,k+1}-t_{k+1}\tau_{k+1}\mu_{k+1}}{2\tau_{k+1}}\left\|x^*-\xk-\frac{t_{k+1}(1-\tau_{k+1}\mu_{f,k+1})}{1+\tau_{k+1}\mu_{g,k+1}-t_{k+1}\tau_{k+1}\mu_{k+1}}(\bar{y}^{(k+1)}-\xk)\right\|_{D_{k+1}}^2\nonumber\\
&+t_{k+1}^2\epsilon_{k+1}+t_{k+1}\frac{\sqrt{2\epsilon_{k+1}\tau_{k+1}(1+\tau_{k+1}\mu_{g,k+1})}}{\tau_{k+1}}\|x^*-\xkk-(t_{k+1}-1)(\xkk-\xk)\|_{D_{k+1}}.
\end{align*}
Recalling the definitions of $\omega_{k+1}$ and $\tau_{k+1}'$ given in \eqref{eq:tau_prime_omega_}, we observe that $0< \omega_{k+1}\leq  1$ since $0\leq q_{k+1}t_{k+1}< 1$ (see Lemma \ref{rem:2}).  Furthermore, the parameters $\beta_{k+1}$ appearing in Algorithm \ref{algo:GFISTA} can be rewritten in terms of $\omega_{k+1}$ as

\begin{equation} \label{eq:beta}
\beta_{k+1}=\omega_{k+1}\frac{t_k-1}{t_{k+1}}\frac{1+\tau_{k+1}\mu_{g,k+1}}{1-\tau_{k+1}\mu_{f,k+1}}.
\end{equation}
 Recalling also the definition of $\bar{y}^{(k+1)}$, plugging  \eqref{eq:tau_prime_omega_} and \eqref{eq:beta} inside the previous inequality and multiplying it by $\tau_{k+1}'$, we deduce:
\begin{align}\label{eq:intermediate_3}
& \tau_{k+1}'t_{k+1}^2(F(\xkk)-F(x^*))+\frac{1}{2}\|x^*-\xkk-(t_{k+1}-1)(\xkk-\xk)\|_{D_{k+1}}^2 \nonumber \\
&  \leq \tau_{k+1}'t_{k+1}(t_{k+1}-1)(F(\xk)-F(x^*)) +\frac{\omega_{k+1}}{2}\|x^*-\xk-(t_k-1)(\xk-x^{(k-1)})\|_{D_{k+1}}^2\nonumber\\
&+\tau_{k+1}'t_{k+1}^2\epsilon_{k+1}+t_{k+1}\sqrt{2\epsilon_{k+1}\tau'_{k+1}}\|x^*-\xkk-(t_{k+1}-1)(\xkk-\xk)\|_{D_{k+1}}.
\end{align}
Combining equation \eqref{eq:choice_tk} with relation $\mu_k/\mu_{k+1}=\eta_{sup}^{k+1}/\eta_{sup}^k$, we deduce from \eqref{eq:intermediate_3} the following:
\begin{align*}
& \tau_{k+1}'t_{k+1}^2(F(\xkk)-F(x^*))+\frac{1}{2}\|x^*-\xkk-(t_{k+1}-1)(\xkk-\xk)\|_{D_{k+1}}^2 \nonumber \\
&  \leq \omega_{k+1}\left(\frac{\eta_{sup}^{k+1}}{\eta_{sup}^k}\tau_k't_k^2(F(\xk)-F(x^*)) +\frac{1}{2}\|x^*-\xk-(t_k-1)(\xk-x^{(k-1)})\|_{D_{k+1}}^2\right)\nonumber\\
&+\tau_{k+1}'t_{k+1}^2\epsilon_{k+1}+t_{k+1}\sqrt{2\epsilon_{k+1}\tau'_{k+1}}\|x^*-\xkk-(t_{k+1}-1)(\xkk-\xk)\|_{D_{k+1}}
\end{align*}
whence the thesis follows by applying Assumption \ref{ass:1}.
\end{proof}

Let us now define the sequence $\{\theta_k\}_{k\in\mathbb{N}}$ by
\begin{equation}\label{eq:theta}
\theta_k:=\frac{\displaystyle\prod_{i=0}^{k}\omega_i}{\tau_k't_k^2}, \quad \forall \ k\geq 0.
\end{equation}
The convergence rate for the function values of the iterates of Algorithm \ref{algo:GFISTA} can be computed by a careful study of the factor $\theta_k$, which decays linearly if $\mu>0$ and quadratically if $\mu=0$. We prove this statement with the following.

\begin{Theorem}\label{thm:rate}
Suppose $F=f+g$ is $\mu-$strongly convex with $\mu=\mu_f+\mu_g\geq 0$, Assumption \ref{ass:1} holds and $1\leq t_0\leq 1/\sqrt{q_0}$. Let $x^*$ be a solution of \eqref{minf}, recall definitions \eqref{eq:tau_prime_omega_} and \eqref{eq:theta}, set $\gamma=\lim_{k\rightarrow\infty}\prod\limits_{i=1}^{k}(1+\gamma_i)<\infty$ and
\begin{equation}\label{eq:OmegaE}
E_1^k=\sum_{i=0}^{k}\sqrt{\theta_{i+1}^{-1}\epsilon_{i+1}}, \qquad E_{2}^k=\sum_{i=0}^{k}\theta_{i+1}^{-1}\epsilon_{i+1}.
\end{equation}
Then, for all $k\geq 0$, the following upper bound for the function values holds:	
\begin{equation}\label{eq:rate}
F(\xkk)-F(x^*)\leq \gamma \theta_{k+1}\left(\sqrt{\frac{\omega_0}{2}}\|x^{(0)}-x^*\|_{D_0}+\sqrt{ \tau_{0}'t_0^2\omega_0\left(F(x^{(0)})-F(x^*)\right)}+2\sqrt{\gamma}E_{1}^k+\sqrt{E_{2}^k}\right)^2.
\end{equation}
Furthermore, the factor $\theta_{k+1}$ can be bounded as follows:
\begin{equation} \label{eq:conv_rate}
\theta_{k+1}\leq 
 \frac{\eta_{sup}}{\eta_{inf}}\cdot \min\left\{\left(1-\sqrt{\frac{\mu\rho}{L_f\eta_{sup}+\mu_g\rho}}\right)^{k+1}\left(\frac{1}{\tau_0}-\mu_{f,0}\right),\frac{4}{(k+2)^2}\sqrt{\frac{L_f\eta_{sup}-\rho\mu_f}{\rho\eta_{sup}}}\right\}. 
\end{equation}
\end{Theorem}

\begin{proof}
The result is obtained by combining the proof in \cite[Theorem 4.6]{Calatroni-Chambolle-2019} with some arguments employed in \cite[Propositions 2-4]{Schmidt2011}. By using the following short-hand notations
\begin{equation*}
u^{(k)}:=x^*-\xk-(t_k-1)(\xk-x^{(k-1)}),\quad
v_k:=F(\xk)-F(x^*),
\end{equation*}
we have that the inequality in Lemma \ref{lem:fund} can be reformulated in a compact form as
\begin{align}\label{eq:rec}
\tau_{k+1}'t_{k+1}^2v_{k+1}+\frac{1}{2}\|u^{(k+1)}\|_{D_{k+1}}^2 &\leq \omega_{k+1}(1+\gamma_{k+1})\left(\tau_k't_k^2 v_k+\frac{1}{2}\|u^{(k)}\|_{D_{k}}^2\right)\nonumber\\
&+\tau_{k+1}'t_{k+1}^2\epsilon_{k+1}+t_{k+1}\sqrt{2\epsilon_{k+1}\tau'_{k+1}}\|u^{(k+1)}\|_{D_{k+1}}.
\end{align}
By recursively applying \eqref{eq:rec}, we obtain
\begin{align}\label{eq:recursive}
\tau_{k+1}'t_{k+1}^2v_{k+1}& +\frac{1}{2}\|u^{(k+1)}\|_{D_{k+1}}^2 \leq \left(\prod_{i=1}^{k+1}\omega_i(1+\gamma_i)\right)\left(\tau'_0t_0^2v_0+\frac{1}{2}\|u^{(0)}\|_{D_0}^2\right)\nonumber\\
&+\sum_{i=0}^k\left(\prod_{j=i+2}^{k+1}\omega_{j}(1+\gamma_j)\right)\left(\tau_{i+1}'t_{i+1}^2\epsilon_{i+1}+t_{i+1}\sqrt{2\epsilon_{i+1}\tau'_{i+1}}\|u^{(i+1)}\|_{D_{i+1}}\right), 
\end{align}
where we used the notation $\prod_{j=i+2}^{k+1}\omega_{j}(1+\gamma_j)=1$ whenever $i>k-1$.\\
Let us define the sequence $\{p_k\}_{k\in \mathbb{N}}$ with elements $p_k=\prod_{i=1}^{k}(1+\gamma_i)$ for all $k\geq 1$. We notice that $p_{k+1}=(1+\gamma_{k+1})p_{k}$. Recalling the summability of the sequence $\{\gamma_k\}_{k\in\mathbb{N}}$ and applying Lemma \ref{key_lemma_2}, it follows that $\{p_k\}_{k\in \mathbb{N}}$ converges. Setting $\gamma =\lim_{k\rightarrow \infty} \prod_{i=1}^{k}(1+\gamma_i)<\infty$, we deduce from equation \eqref{eq:recursive} 
\begin{align*}
\tau_{k+1}'t_{k+1}^2v_{k+1}&+\frac{1}{2}\|u^{(k+1)}\|_{D_{k+1}}^2 \leq \left(\prod_{i=1}^{k+1}\omega_i\right)\left(\gamma\tau'_0t_0^2v_0+\frac{\gamma}{2}\|u^{(0)}\|_{D_0}^2\right)\nonumber\\
&+\sum_{i=0}^k\left(\prod_{j=i+2}^{k+1}\omega_{j}\right)\left(\gamma\tau_{i+1}'t_{i+1}^2\epsilon_{i+1}+\gamma t_{i+1}\sqrt{2\epsilon_{i+1}\tau'_{i+1}}\|u^{(i+1)}\|_{D_{i+1}}\right).
\end{align*}
Let us now define $\Omega_{k+1}:=\prod_{i=0}^{k+1}\omega_i\in (0,1]$. Recalling the definition of $\theta_k$ in \eqref{eq:theta}, we can rewrite the previous inequality as:
\begin{align}\label{eq:recursive_bis}
&\tau_{k+1}'t_{k+1}^2v_{k+1}+\frac{1}{2}\|u^{(k+1)}\|_{D_{k+1}}^2\nonumber\\
&\leq \gamma\Omega_{k+1}\left( \tau'_0t_0^2v_0\omega_0+\frac{\omega_0}{2}\|u^{(0)}\|_{D_0}^2+\sum_{i=0}^k\theta_{i+1}^{-1}\epsilon_{i+1}+ \sqrt{2\theta_{i+1}^{-1}\epsilon_{i+1}}\Omega_{i+1}^{-\frac{1}{2}}\|u^{(i+1)}\|_{D_{i+1}} \right).
\end{align}
We can now proceed as in \cite[Proposition 2]{Schmidt2011}. First, we discard the term $\tau_{k+1}'t_{k+1}^2v_{k+1}$ on the left-hand side, so as to deduce the following inequality
\begin{equation*}
\Omega^{-1}_{k+1}\|u^{(k+1)}\|_{D_{k+1}}^2\leq 2\gamma \left(\frac{\omega_0\|u^{(0)}\|_{D_0}^2}{2}+\tau_0' t_0^2v_0\omega_0+\sum_{i=0}^k\theta_{i+1}^{-1}\epsilon_{i+1}\right)+\sum_{i=0}^k2\gamma\sqrt{2\theta_{i+1}^{-1}\epsilon_{i+1}}\Omega_{i+1}^{-\frac{1}{2}}\|u^{(i+1)}\|_{D_{i+1}}.
\end{equation*}
An application of Lemma \ref{key_lemma} to this relation leads to
\begin{align*}
\Omega^{-\frac{1}{2}}_{k+1}\|u^{(k+1)}\|_{D_{k+1}}&\leq \sum_{i=0}^{k} \gamma \sqrt{2\theta_{i+1}^{-1}\epsilon_{i+1}}\\
&+\left(2\gamma\left(\frac{\omega_0\|u^{(0)}\|_{D_0}^2}{2}+\tau_0't_0^2v_0\omega_0+\sum_{i=0}^k\theta_{i+1}^{-1}\epsilon_{i+1}\right)+\left(\sum_{i=0}^{k}\gamma \sqrt{2\theta_{i+1}^{-1}\epsilon_{i+1}}\right)^2\right)^{\frac{1}{2}}.
\end{align*}
Therefore, by standard inequalities on the square root of the sum of positive terms, we can deduce that for all $j=0,\ldots,k$
\begin{equation}\label{eq:plugin}
\Omega^{-\frac{1}{2}}_{j+1}\|u^{(j+1)}\|_{D_{j+1}}\leq \sqrt{\gamma\omega_0}\|u^{(0)}\|_{D_0}+\sqrt{2\gamma \tau_0' t_0^2v_0\omega_0}+\sqrt{2 \sum_{i=0}^k\gamma\theta_{i+1}^{-1}\epsilon_{i+1}}+2\sum_{i=0}^{k}\gamma \sqrt{2\theta_{i+1}^{-1}\epsilon_{i+1}}.
\end{equation}
We now go back to \eqref{eq:recursive_bis}, discard the term $\frac{1}{2}\|u^{(k+1)}\|_{D_{k+1}}^2$ on the left-hand side, plug the relation \eqref{eq:plugin} and write the following chain of inequalities
\begin{align*}
&\theta_{k+1}^{-1}v_{k+1}\leq \gamma \tau_{0}'t_0^2v_0\omega_0+\frac{\gamma\omega_0}{2}\|u^{(0)}\|_{D_0}^2+\sum_{i=0}^{k}\gamma \theta_{i+1}^{-1}\epsilon_{i+1}\\
&+\left(2\sum_{i=0}^{k}\gamma \sqrt{\theta_{i+1}^{-1}\epsilon_{i+1}}\right)\left( \sqrt{\frac{\gamma\omega_0}{2}}\|u^{(0)}\|_{D_0}+\sqrt{\gamma\tau_0' t_0^2v_0\omega_0}+\sqrt{\sum_{i=0}^k\gamma \theta_{i+1}^{-1}\epsilon_{i+1}}+2\sum_{i=0}^{k}\gamma \sqrt{\theta_{i+1}^{-1}\epsilon_{i+1}}\right)\\
&\leq \left(\sqrt{\frac{\gamma\omega_0}{2}}\|u^{(0)}\|_{D_0}+\sqrt{\gamma\tau_{0}'t_0^2v_0\omega_0}+2\sum_{i=0}^{k}\gamma\sqrt{\theta_{i+1}^{-1}\epsilon_{i+1}}+\sqrt{\sum_{i=0}^k\gamma\theta_{i+1}^{-1}\epsilon_{i+1}}\right)^2.
\end{align*}
We thus deduce
\begin{equation}\label{eq:final}
v_{k+1}\leq \gamma\theta_{k+1}\left(\sqrt{\frac{\omega_0}{2}}\|u^{(0)}\|_{D_0}+\sqrt{\tau_{0}'t_0^2v_0\omega_0}+2\sum_{i=0}^{k}\sqrt{\gamma\theta_{i+1}^{-1}\epsilon_{i+1}}+\sqrt{\sum_{i=0}^k\theta_{i+1}^{-1}\epsilon_{i+1}}\right)^2.
\end{equation}
At this point, proceeding as in \cite[Theorem 4.6]{Calatroni-Chambolle-2019}, we need to provide an upper bound for the factor $\theta_{k}$. To this aim, we introduce the following average quantities:
\begin{equation}   \label{eq:Lbar}
\sqrt{\bar{L}_{k+1}} := \frac{1}{\frac{1}{k+2} \sum\limits_{i=0}^{k+1}  \frac{1}{\sqrt{L_i - \mu_{f,i}}}},\qquad \sqrt{ \bar{q}_{k+1} }:= \frac{1}{k+1}\sum_{i=1}^{k+1} \sqrt{ \frac{\mu_i}{L_i + \mu_{g,i}} },
\end{equation}
where $L_i:=1/\tau_i$.\\
In order to show the quadratic decay of $\theta_k$, we first prove the following inequality:
\begin{equation}\label{eq:induction}
\frac{1}{\sqrt{\eta_{sup}^{k}\theta_{k}}}\geq \frac{1}{2}\sum_{i=0}^{k}\frac{1}{\sqrt{\eta_{sup}^{i}\left(\frac{1}{\tau_{i}}-\mu_{f,i}\right)}}, \quad \forall \ k\geq 0.
\end{equation}
The proof follows by induction: for $k=0$ we have
\begin{align}\label{eq:theta1}
\theta_0=\frac{1-\mu_0\tau_0't_0}{\tau_0't_0^2}=\frac{(1-\mu_0\tau_0't_0)(1+\tau_0\mu_{g,0})}{\tau_0t_0^2} =\frac{1-(t_0-1)\tau_0\mu_{g,0}-\mu_{f,0}\tau_0t_0}{\tau_0 t_0^2}\leq \frac{1}{\tau_0}-\mu_{f,0},
\end{align}
which holds since, by assumption, $t_0\geq 1$. We observe that this is nothing but \eqref{eq:induction} for $k=0$. By induction, suppose now that \eqref{eq:induction} holds for some $k\geq 0$. Then, from \eqref{eq:choice_tk}, we deduce:
\begin{equation}\label{eq:technical_for_rate}
1-\frac{1}{t_{k+1}}=\frac{\mu_k}{\mu_{k+1}}\omega_{k+1}\frac{\tau_k't_k^2}{\tau_{k+1}'t_{k+1}^2}=\frac{\eta_{sup}^{k+1}}{\eta_{sup}^k}\omega_{k+1}\frac{\tau_k't_k^2}{\tau_{k+1}'t_{k+1}^2}=\frac{\eta_{sup}^{k+1}}{\eta_{sup}^k}\frac{\theta_{k+1}}{\theta_k}\leq 1,
\end{equation} 
which implies that the sequence $\{\eta_{sup}^k\theta_k\}_{k\in\mathbb{N}}$ is non-increasing. Therefore, we have
\begin{equation*}
\frac{1}{\sqrt{\eta_{sup}^{k+1}\theta_{k+1}}}-\frac{1}{\sqrt{\eta_{sup}^{k}\theta_{k}}}=\frac{\eta_{sup}^k\theta_k-\eta_{sup}^{k+1}\theta_{k+1}}{\sqrt{\eta_{sup}^k\eta_{sup}^{k+1}\theta_k\theta_{k+1}}\left(\sqrt{\eta_{sup}^k\theta_k}+\sqrt{\eta_{sup}^{k+1}\theta_{k+1}}\right)}\geq \frac{\eta_{sup}^k\theta_k-\eta_{sup}^{k+1}\theta_{k+1}}{2\eta_{sup}^k\theta_k\sqrt{\eta_{sup}^{k+1}\theta_{k+1}}}.
\end{equation*}
By using again \eqref{eq:technical_for_rate}, we further obtain
\begin{equation*}
\frac{1}{\sqrt{\eta_{sup}^{k+1}\theta_{k+1}}}-\frac{1}{\sqrt{\eta_{sup}^{k}\theta_{k}}}\geq \frac{1}{2t_{k+1}\sqrt{\eta_{sup}^{k+1}\theta_{k+1}}}.
\end{equation*}
Now, recalling the definitions of the parameters $\theta_{k+1}$ and $\omega_{k+1}$, we can write the following chain of inequalities:
\begin{align*}
t_{k+1}\sqrt{\eta_{sup}^{k+1}\theta_{k+1}}&=\sqrt{\frac{\eta_{sup}^{k+1}}{\tau_{k+1}'}}\prod_{i=0}^{k+1}\sqrt{\omega_i} \leq\sqrt{\frac{\eta_{sup}^{k+1}\omega_{k+1}}{\tau_{k+1}'}}=\sqrt{\eta_{sup}^{k+1}\left(\frac{1}{\tau_{k+1}'}-\mu_{k+1}t_{k+1}\right)}\\
&\leq\sqrt{\eta_{sup}^{k+1}\left(\frac{1}{\tau_{k+1}'}-\mu_{k+1}\right)}=\sqrt{\eta_{sup}^{k+1}\left(\frac{1}{\tau_{k+1}}-\mu_{f,k+1}\right)}.
\end{align*}
Thus, we get
\begin{equation*}
\frac{1}{\sqrt{\eta_{sup}^{k+1}\theta_{k+1}}}-\frac{1}{\sqrt{\eta_{sup}^{k}\theta_{k}}}\geq \frac{1}{2\sqrt{\eta_{sup}^{k+1}\left(\frac{1}{\tau_{k+1}}-\mu_{f,k+1}\right)}}.
\end{equation*}
By induction, the previous inequality yields \eqref{eq:induction} with $k=k+1$.
Recalling the definition of $\bar{L}_{k+1}$ and observing that $\eta_{inf}\leq \eta_{sup}^i\leq\eta_{sup}$ for $i=0,\ldots,k$, we get
\begin{equation}\label{eq:quad_rate}
\sqrt{\theta_{k+1}}\leq \frac{2}{k+2}\sqrt{\frac{\eta_{sup}}{\eta_{inf}}\bar{L}_{k+1}}, \quad \forall \ k\geq 0.
\end{equation}
To get the linear rate, we use \eqref{eq:technical_for_rate} and \eqref{eq:qt} and deduce the following relations:
\begin{align}\label{eq:lin_rate}
\theta_{k+1}&=\theta_0\prod_{i=1}^{k+1}\frac{\eta_{sup}^{i-1}}{\eta_{sup}^{i}}\left(1-\frac{1}{t_{i}}\right) \leq\frac{\eta_{sup}^0\theta_0}{\eta_{sup}^{k+1}}\prod_{i=1}^{k+1}\left(1-\sqrt{q_{i}}\right)\nonumber\\
&=\frac{\eta_{sup}^0\theta_0}{\eta_{sup}^{k+1}}\prod_{i=1}^{k+1}\left(1-\sqrt{\frac{\mu_{i}}{L_i+\mu_{g,i}}}\right) \leq \frac{\eta_{sup}}{\eta_{inf}}(L_0-\mu_{f,0})\left(1-\sqrt{\bar{q}_{k+1}}\right)^{k+1},
\end{align}
where the last inequality follows from the concavity of the logarithmic function, the definition of $\bar{q}_{k+1}$, \eqref{eq:theta1} and the fact that $\eta_{inf}\leq \eta^{i}_{sup}\leq \eta_{sup} $ for $i=0,\ldots,k+1$.\\
Finally, we note that the averaging term $\bar{L}_{k+1}$ appearing in \eqref{eq:quad_rate} is always smaller than the actual average of the terms $(L_i-\mu_{f,i})$, since:
\begin{equation}   \label{ineq:mean}
 \sqrt{\bar{L}_k}  \leq \frac{1}{k+1} \sum_{i=0}^k \sqrt{L_i - \mu_{f,i}}  \leq \sqrt{ \frac{1}{k+1}\sum_{i=0}^k (L_i - \mu_{f,i})  }.
\end{equation}
Then, since the use of the backtracking strategy at any iteration implies that $\tau_k\geq \rho/L_f$, we can deduce the following uniform bounds for the averaging terms:
$$
\sqrt{\bar{L}_{k+1}}\leq \sqrt{\frac{L_f}{\rho}-\frac{\mu_f}{\eta_{sup}}},\qquad\qquad \sqrt{\bar{q}_{k+1}}\geq \sqrt{\frac{\mu\rho}{L_f\eta_{sup}+\mu_g\rho}},
$$
by which \eqref{eq:conv_rate} follows.
\end{proof}

\begin{Remark}
By \eqref{eq:quad_rate}-\eqref{eq:lin_rate}, we deduce that the factor $\theta_{k+1}$ can be estimated just in terms of the average quantities $\overline{L}_{k+1}$ and $\overline{q}_{k+1}$ defined in \eqref{eq:Lbar}, thus avoiding the dependence on the (possibly unknown) value $L_f$. In this case, note that the storage of all the values $L_i$ computed along the iterations by backtracking strategy is required.
\end{Remark}
The boundedness of the sequences $E_1^k$ and $E_2^k$ can be guaranteed by carefully choosing the error parameters $\epsilon_{k+1}$ controlling the accuracy of the inexact proximal evaluations. If $\mu>0$ the following corollary provides a sufficient condition for this choice.

\begin{corollary}\label{cor:guarantee1}
Suppose that $F=f+g$ is $\mu-$strongly convex with $\mu>0$ and let $x^*$ be the solution of \eqref{minf}. Suppose that Assumption \ref{ass:1} holds and that the sequence of errors $\{\epsilon_k\}_{k\in\mathbb{N}}$ is chosen as
\begin{equation}\label{eq:adaptive_epsilon}
\epsilon_{k+1}=\theta_{k+1}a_{k+1},
\end{equation}
where $\{\theta_k\}_{k\in\mathbb{N}}$ is defined as in \eqref{eq:theta} and $\{a_k\}_{k\in\mathbb{N}}$ is a sequence of nonnegative numbers such that $\sum_{k=0}^{\infty}\sqrt{a_k}<\infty$. Then, for all sufficiently large $k$, we have
\begin{equation*}
F(x^{(k+1)})-F(x^*)=\mathcal{O}\left(\left(1-\sqrt{\frac{\mu\rho}{L_f\eta_{sup}+\mu_g\rho}}\right)^{k+1}\right).
\end{equation*} 
\end{corollary}

Note that the factor $\theta_{k+1}$ depends on $q_{k+1}$ and $t_{k+1}$, which are computed at {\sc STEP 2} and {\sc STEP 3} of Algorithm \ref{algo:GFISTA}. Therefore, in order to ensure the linear rate, the parameter $\epsilon_{k+1}$ needs to be recomputed according to \eqref{eq:adaptive_epsilon} at each trial step of the backtracking procedure. 

\medskip

In the particular case $\mu=\mu_g>0$, we are able to pre-identify a sequence of error parameters that guarantees a linear rate without requiring any inner update, as shown below.

\begin{corollary}\label{cor:guarantee2}
Suppose that $F=f+g$ is $\mu-$strongly convex with $\mu=\mu_g>0$ and let $x^*$ be the solution of \eqref{minf}. Suppose that Assumption \ref{ass:1} holds and that the sequence of errors $\{\epsilon_k\}_{k\in\mathbb{N}}$ is chosen as
\begin{equation}\label{eq:prefixed2}
\epsilon_{k+1}=\mathcal{O}((ab^{k})^{k+1}),
\end{equation}
where $a< \frac{\delta}{2}\min\{1,\eta_{inf}/(\tau_0\mu_g)\}$ and $b< \sqrt{\delta}$. Then, for all sufficiently large $k$, we have
\begin{equation*}
F(x^{(k+1)})-F(x^*)=\mathcal{O}\left(\left(1-\sqrt{\frac{\mu_g\rho}{L_f\eta_{sup}+\mu_g\rho}}\right)^{k+1}\right).
\end{equation*} 
\end{corollary}

\begin{proof}
Let us provide an upper bound for the quantity $\theta_{i+1}^{-1}$ for $i=0,\ldots,k$. Recalling the relation $\tau_{i+1}'=q_{i+1}/\mu_{i+1}$ and the fact that $q_{i+1}t_{i+1}^2\leq 1$ for all $i$ (see Lemma \ref{rem:2}), we have
\begin{align*}
\theta_{i+1}^{-1}= \frac{q_{i+1}t_{i+1}^2}{\mu_{i+1}\prod\limits_{j=1}^{i+1} (1-q_jt_j)}&\leq \frac{1}{\mu_{i+1}}\prod_{j=1}^{i+1}\frac{1}{1-\sqrt{q_j}}=\frac{1}{\mu_{g,i+1}}\prod_{j=1}^{i+1}\sqrt{1+\tau_j\mu_{g,j}}(\sqrt{1+\tau_j\mu_{g,j}}+\sqrt{\tau_j\mu_{g,j}}),
\end{align*}
where the second equality follows from {\sc STEP 2} of Algorithm \ref{algo:GFISTA} and the assumption $\mu=\mu_g$. Note that we can recursively apply the inequality $\tau_j\leq \tau_{j-1}/\delta$ so as to obtain 
\begin{equation}\label{eq:tau}
\tau_j\leq\tau_0\left(\frac{1}{\delta}\right)^{j}, \quad j\geq 0.
\end{equation}
Then, we can continue as follows:
\begin{align*}
\theta_{i+1}^{-1} \leq\frac{2^{i+1}}{\mu_{g,i+1}}\prod_{j=1}^{i+1}1+\tau_j\mu_{g,j}&\leq \frac{\eta_{sup}(2\max\{1,(\tau_0\mu_g)/\eta_{inf}\})^{i+1}}{\mu_g}\prod_{j=1}^{i+1}\left(\frac{1}{\delta}\right)^{j}\\
&=\frac{\eta_{sup}(2\max\{1,(\tau_0\mu_g)/\eta_{inf}\})^{i+1}}{\mu_g}\left(\frac{1}{\sqrt{\delta}}\right)^{(i+2)(i+1)},
\end{align*}
where the second inequality follows from \eqref{eq:tau}, the definition of $\mu_{g,i+1}$ and $\eta_{inf}\leq \eta_{sup}^j\leq \eta_{sup}$. Then, the thesis follows by applying the above inequality to \eqref{eq:rate} in Theorem \ref{thm:rate}.
\end{proof} 
We remark that condition \eqref{eq:prefixed2} is more severe when $\delta<1$, i.e. whenever an adaptive backtracking procedure is considered. 

\medskip

If $\mu=0$, the SAGE-FISTA Algorithm \ref{algo:GFISTA} reduces to an inexact variable metric version of FISTA. Consequently, Theorem \ref{thm:rate} recovers the well-known $\mathcal{O}(1/k^2)$ convergence rate result available for FISTA and for its several inexact variants in the convex case \cite{Beck-Teboulle-2009b,Schmidt2011,Bonettini2018a}. Similarly to the case $\mu=\mu_g$, we can devise some prefixed rule for computing the errors $\epsilon_{k+1}$ in order to get the expected quadratic rate.

\begin{corollary}\label{cor:guarantee3}
Suppose that $F=f+g$ is $\mu-$strongly convex with $\mu=0$ and let $x^*$ be a solution of \eqref{minf}. Suppose Assumption \ref{ass:1} holds and the sequence of errors $\{\epsilon_k\}_{k\in\mathbb{N}}$ is chosen as
\begin{equation}\label{eq:prefixed3}
\epsilon_{k+1}=\begin{cases}
\mathcal{O}(a^{k+1}), \quad &\text{if } \ \delta <1\\
\displaystyle \frac{b_{k+1}}{(k+1+t_0)^2}, \quad &\text{if } \ \delta =1
\end{cases},
\end{equation}
where $a<\delta$ and $\{b_k\}_{k\in\mathbb{N}}$ is a sequence of nonnegative numbers such that $\sum_{k=0}^{\infty}\sqrt{b_k}<\infty$. Then, for all $k\geq 0$, we have
\begin{equation*}
F(x^{(k+1)})-F(x^*)=\mathcal{O}\left(\frac{1}{(k+1)^2}\right).
\end{equation*} 
\end{corollary}

\begin{proof}
If $\mu=0$, we can majorize the quantity $\sqrt{\theta_{i+1}^{-1}}$ as follows
\begin{align}\label{eq:final_plug_2}
\sqrt{\theta_{i+1}^{-1}}= \sqrt{\tau_{i+1}}t_{i+1} \leq \sqrt{\tau_{i+1}} \left( 1+t_i\sqrt{\frac{\tau_i}{\tau_{i+1}}}\right)  = \sqrt{\tau_{i+1}}+ t_i \sqrt{\tau_i}\leq \ldots \leq \sum_{j=1}^{i+1} \sqrt{\tau_j} + t_0\sqrt{\tau_0}, 
\end{align}
which follows by applying recursively the property:
\[
t_{j+1}\leq 1+ t_j \sqrt{\frac{\tau_j}{\tau_{j+1}}},\quad j=0,\ldots, i.
\]
Combining \eqref{eq:final_plug_2} with \eqref{eq:tau} leads to
\begin{align*}
\theta_{i+1}^{-1}\leq \left(\sum_{j=1}^{i+1} \sqrt{\tau_0}\left(\frac{1}{\sqrt{\delta}}\right)^{j} + t_0\sqrt{\tau_0}\right)^2  =\displaystyle\begin{cases}
\tau_0\left(\frac{(1/\sqrt{\delta})^{i+2}-(1/\sqrt{\delta})}{(1/\sqrt{\delta})-1}+t_0\right)^2, \quad &\text{if }\delta<1,\\
\tau_0\left(i+1+t_0\right)^2, \quad& \text{if }\delta=1.
\end{cases}
\end{align*}
Using this upper bound in \eqref{eq:rate}, we get the thesis.
\end{proof}

Note that when $\delta=1$ condition \eqref{eq:prefixed3} is analogous to the one imposed in \cite[Corollary 3.1]{Bonettini2018a} to get the quadratic rate for an inexact, variable metric version of the FISTA algorithm. When $\delta<1$, \eqref{eq:prefixed3} is more restrictive, as it requires the errors to decay at a linear rate.

\medskip

Finally, we consider the case where no backtracking is performed, i.e. the step-size $\tau_0$ is chosen in order to satisfy the descent Lemma \ref{lemmadescent} and $\delta=1$.

\begin{corollary}\label{cor:guarantee4}
Suppose that $F=f+g$ is $\mu-$strongly convex with $\mu>0$ and let $x^*$ be a solution of \eqref{minf}. Suppose Assumption \ref{ass:1} holds, $\delta=1$ and $0<\tau_0< \frac{\eta_{inf}}{L_f}$. Define the parameter $ 0 < q <1$ as $q=(\tau_0\mu)/(\eta_{inf}+\tau_0\mu_g)$, choose $a<1-\sqrt{q}$ and the sequence of errors $\{\epsilon_k\}_{k\in\mathbb{N}}$ as
\begin{equation}\label{eq:prefixed4}
\epsilon_{k+1}=\mathcal{O}(a^{k+1}).
\end{equation}
Then, for all $k\geq 0$, we have $\tau_k= \tau_0$ and 
\begin{equation*}
F(x^{(k+1)})-F(x^*)=\mathcal{O}\left(\left(1-\sqrt{\frac{\tau_0\mu}{\eta_{sup}+\tau_0\mu_g}}\right)^{k+1}\right).
\end{equation*} 
\end{corollary}

\begin{proof}
Since $\tau_0<\eta_{inf}/L_f$ and $\mu_f\leq L_f$, there holds $q<1$. Thanks to $\tau_0<\eta_{inf}/L_f$, $\delta=1$ and Lemma \ref{lemmadescent}, we also have that $\tau_k=\tau_0$ for all $k\geq 0$. Furthermore, $\eta_{sup}^k\geq \eta_{inf}$ implies $q_k\leq q$ for all $k\geq 0$. Then, the following upper bound holds for $\theta_{i+1}^{-1}$:
\begin{equation}
\theta_{i+1}^{-1}\leq \frac{1}{\mu_{i+1}}\prod_{j=1}^{i+1}\frac{1}{1-\sqrt{q_j}}\leq \frac{\eta_{sup}}{\mu}\left(\frac{1}{1-\sqrt{q}}\right)^{i+1}.
\end{equation}
Finally, since $\tau_k\equiv \tau_0$, we can provide the following lower bound for the average quantity $\bar{q}_{k+1}$:
\begin{equation}\label{eq:new_lower}
\sqrt{\bar{q}_{k+1}}\geq \sqrt{\frac{\mu\tau_0}{\eta_{sup}+\tau_0\mu_g}}.
\end{equation}
Then, choosing $\epsilon_{k+1}$ as in \eqref{eq:prefixed4}, and applying \eqref{eq:lin_rate} combined with \eqref{eq:new_lower} in \eqref{eq:rate}, the thesis follows.
\end{proof}

Corollary \ref{cor:guarantee4} guarantees that the function values of the iterates of the SAGE-FISTA  Algorithm \ref{algo:GFISTA} converge to the optimal value at a $R-$linear rate, provided that the errors sequence $\{\epsilon_k\}_{k\in\N}$ converges linearly to $0$ at a sufficiently small rate. We remark that the result obtained in Corollary \ref{cor:guarantee4} is similar to the one provided in \cite[Proposition 4]{Schmidt2011} for a basic variant of the inexact FISTA algorithm, which is specifically designed for strongly convex problems. In contrast with our work, the authors in \cite{Schmidt2011} assume strong convexity only on the differentiable part and, unlike SAGE-FISTA, their algorithm makes use of a constant inertial parameter defined by $\beta_k=(1-\sqrt{\mu_f/L_f})/(1+\sqrt{\mu_f/L_f})$.

\section{Numerical experiments} \label{sec:numerical}

In this section, we report some numerical experiments where  SAGE-FISTA (Algorithm \ref{algo:GFISTA}) is applied to some image restoration problems such as edge-preserving image denoising and image deblurring. As observed in several related works (see, e.g., \cite{Bonettini2019} for a review), the scaling framework discussed in Section \ref{sec:2} turns out to be particularly effective whenever a modeling encoding signal-dependent Poisson noise is considered, as it is common in microscopy and astronomy imaging applications \cite{Bertero2018}. 

Given an observed image $z\in \mathbb{R}_{\geq 0}^n$ with nonnegative pixels, we thus consider the following ill-posed image restoration problem:
\begin{equation}  \label{eq:Poisson_recon}
    z = \mathcal{P}\left(Hx + b \right),
\end{equation}
where $x\in\mathbb{R}_{\geq 0}^n$ is the desired solution, $H\in\mathbb{R}^{n\times n}$ represents a blur (convolution) operator computed for a given Point Spread Function (PSF), the term $b\in\mathbb{R}^n_{> 0}$ stands for a positive background and $\mathcal{P}(w)$ denotes a realization of a Poisson-distributed $n$-dimensional random vector with parameter $w\in\mathbb{R}^n_{>0}$. Following the standard Bayesian Maximum A Posteriori (MAP) paradigm, one finds that the data fidelity term corresponding to the negative log-likelihood of the Poisson p.d.f. is the generalized Kullback-Leibler (KL) divergence functional defined by:
\begin{equation}  \label{eq:KL_def}
    KL(Hx+b;z) := \sum_{i=1}^n \left( z_i \log \frac{z_i}{(Hx)_i + b_i} + (Hx)_i + b_i - z_i \right),
\end{equation}
where the convention $0 \log 0 = 0$ is adopted.  If the operator $H$ has nonnegative entries and if it has at least one strictly positive entry for each row and column (i.e. $He>0$ and $H^Te>0$ for $e\in\mathbb{R}^n$ being the vector of all ones), the KL divergence is a nonnegative, convex and coercive function on the non-negative orthant $Y:= \left\{ x\geq 0 \right\}$, see, e.g., \cite{Harmany12}. The generalized KL divergence can be then coupled with convex and non-smooth image regularization terms, such as the isotropic Total Variation (TV) regularization:
\begin{equation}   \label{eq:TV}
TV(x) := \| \nabla x \|_{2,1} = \sum_{i=1}^n \| \nabla_i x \|_2,
\end{equation}
where, for all $i=1,\ldots,n$, the term $\nabla_i x \in \mathbb{R}^{2}$ stands for the standard forward-difference discretization of the image gradient along the two horizontal and vertical directions at pixel $i$. 
 For simplicity, we will consider reflexive boundary conditions for the computation of such discretization. 
Due to the non-smoothness of the TV term \eqref{eq:TV}, first-order optimization schemes based on forward-backward splitting are often used for solving the composite optimization problem:
\begin{equation}  \label{pb:TV_KL}
    \min_{x\in\mathbb{R}^n}~ KL(Hx+b;z) + \lambda TV(x) + \iota_Y(x)
\end{equation}
where the parameter $\lambda>0$ balances the effect of the regularization against data fidelity and where $\iota_Y(\cdot)$ denotes the indicator function of  $Y$, thus enforcing a non-negative constraint on $Y\subset \mathbb{R}^n$ for the desired solution $x$.

In the following sections, we consider two models analogous to \eqref{pb:TV_KL}, which we adapt to fit with the strongly-convex scenarios suitable for SAGE-FISTA. For both models, we simulate Poisson noisy acquisitions by means of the MATLAB \texttt{imnoise} routine.  We report  in Figure \ref{fig:test_images} the blurred and noisy images used in our numerical experiments, while in Table \ref{table:details_figs} we specify the parameters involved in both the simulation and the reconstruction for each image considered.

\begin{figure}[t!]
\begin{tabular}{cccc}
    \includegraphics[scale = 0.24]{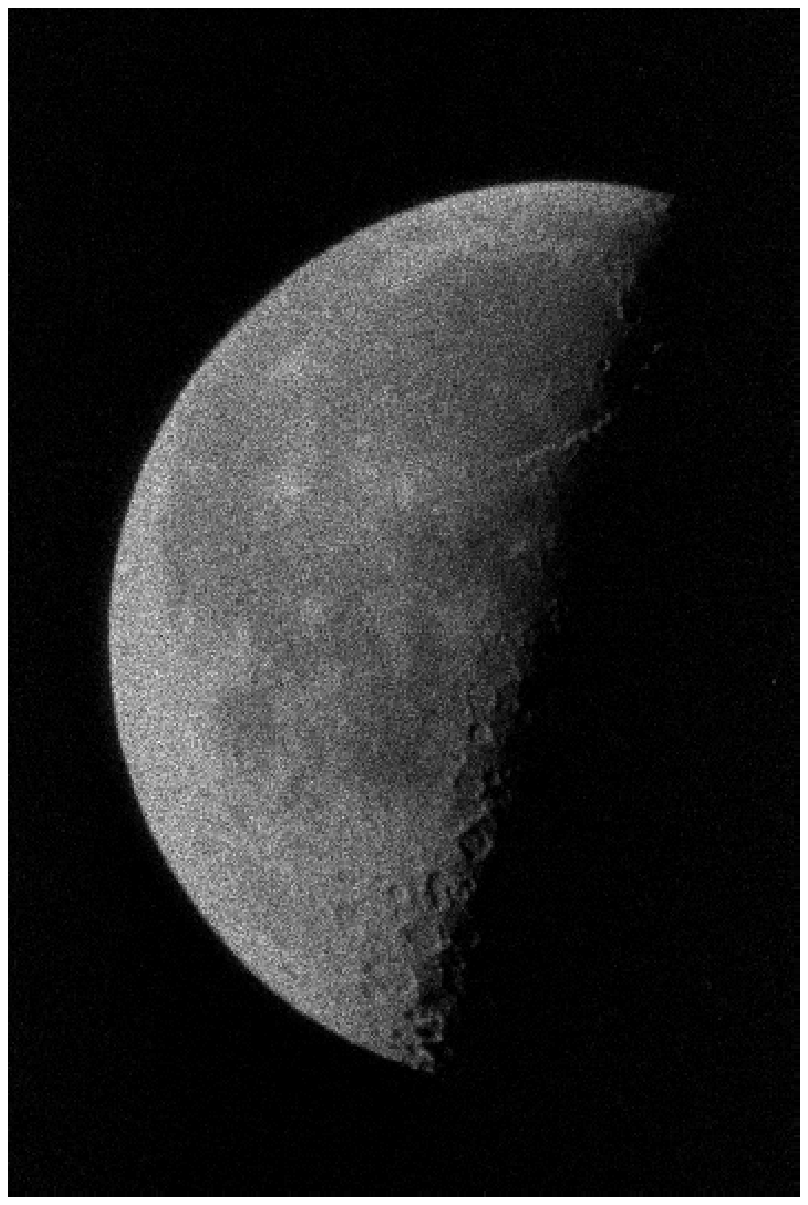} & 
    \includegraphics[scale = 0.24]{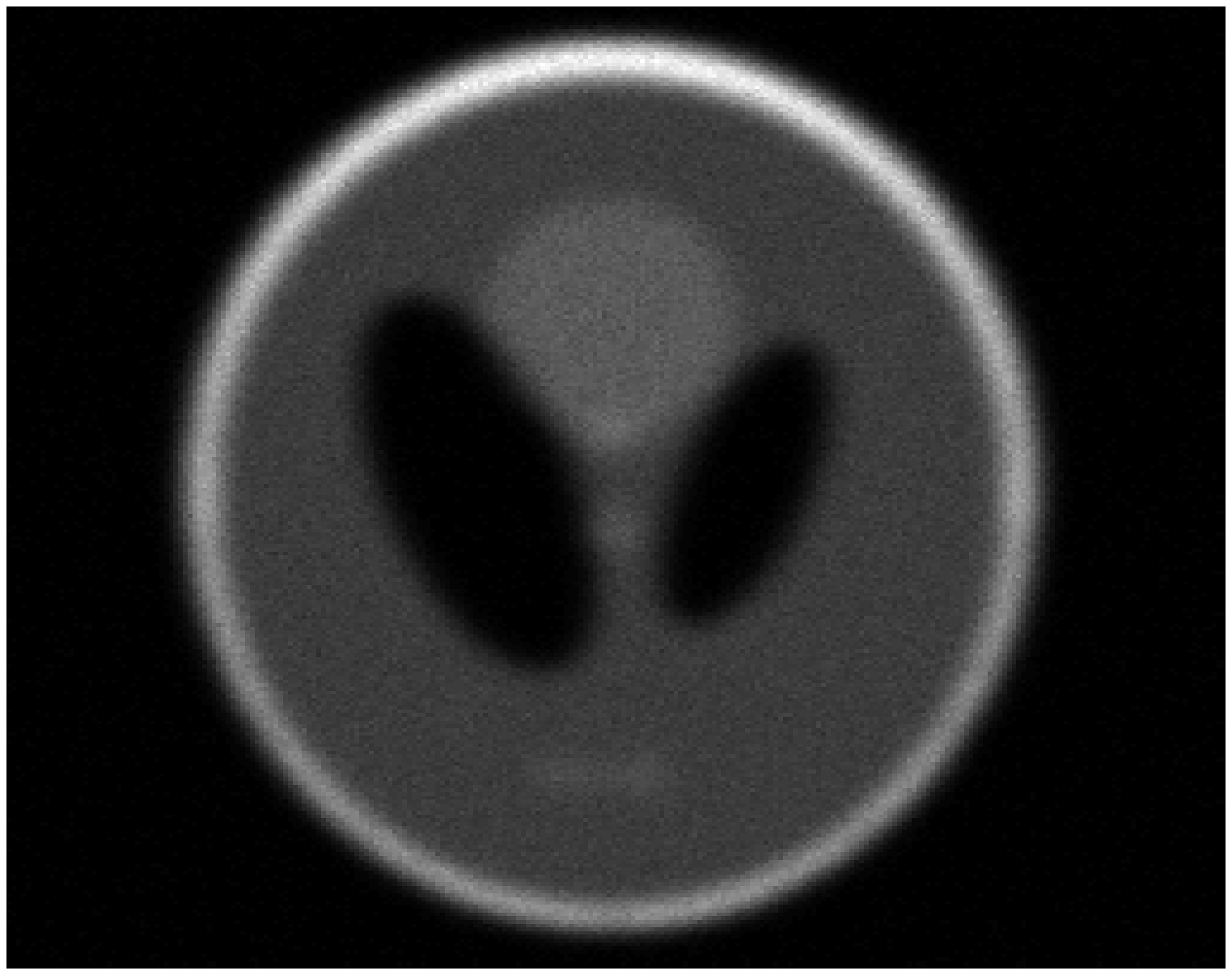} &
    \includegraphics[scale = 0.24]{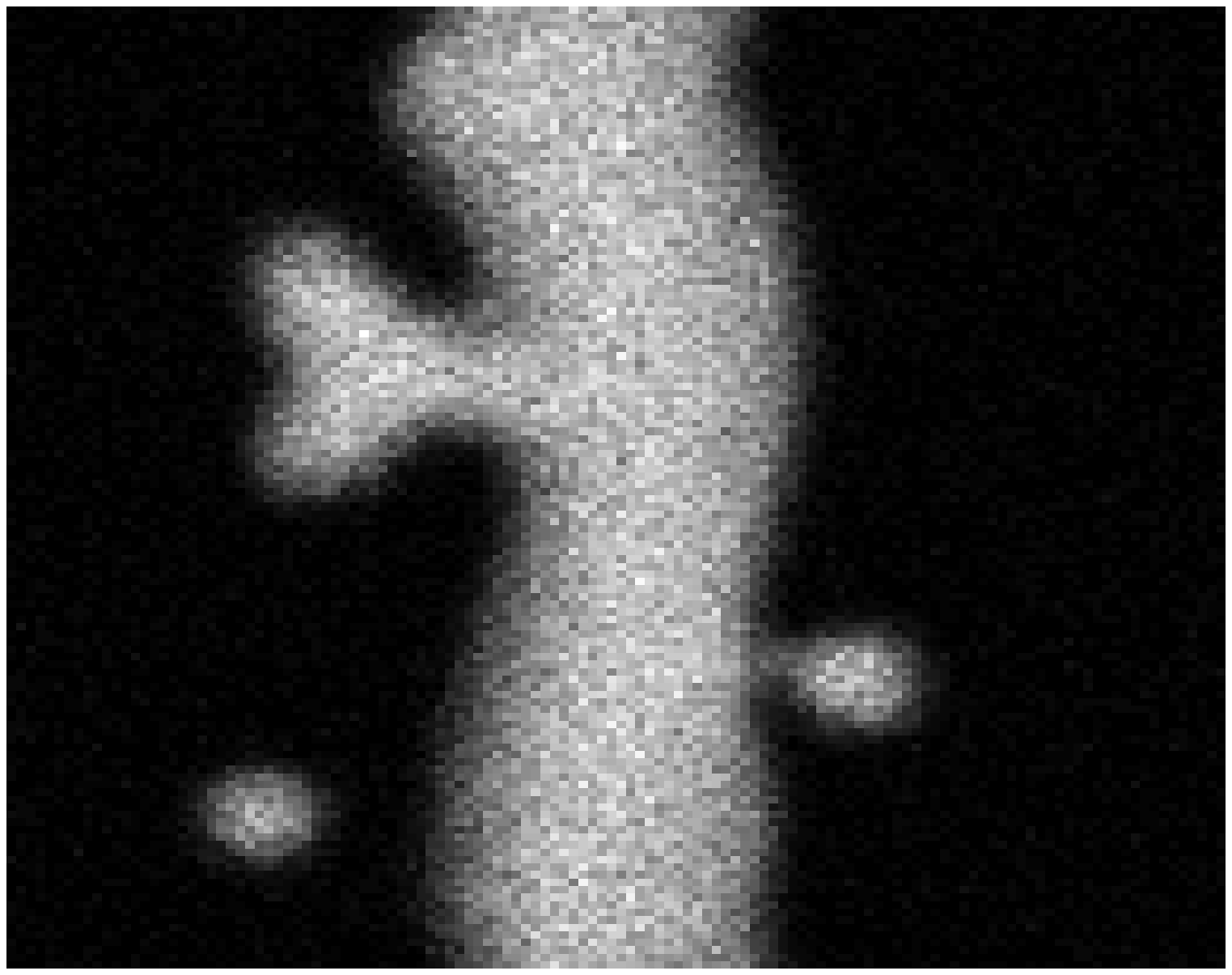} &
    \includegraphics[scale = 0.24]{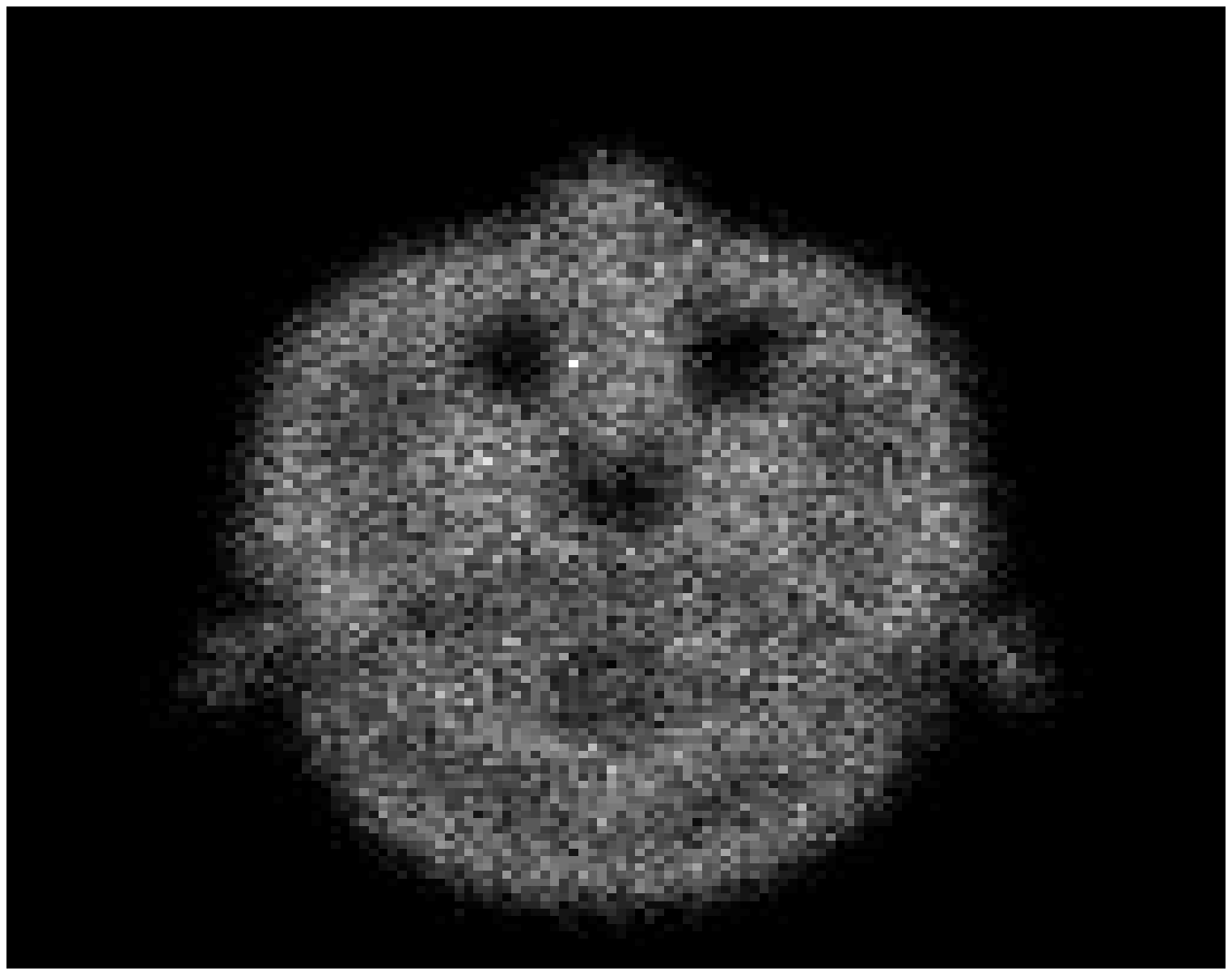} \\
    (a) \texttt{moon} &
    (b) \texttt{phantom} & 
    (c) \texttt{micro} &
    (d) \texttt{mri}
\end{tabular}
    \caption{Blurred and noisy test images used for the numerical tests.}
    \label{fig:test_images}
\end{figure}

For all tests, an approximation $x^*$ of the desired solution is pre-computed by running standard FISTA for 5000 iterations.
Our results are validated by computing the relative objective error: 
\begin{equation}
    eF_k = \frac{F(x^{(k)}) - F(x^*)}{F(x^*)},\qquad k\geq 1
\end{equation}
both along the iterations and with respect to computational times. In the latter case, we show (at most) the first 30 seconds of run, which for our examples is a time frame long enough to reach a significant level of accuracy. The other algorithmic parameters (backtracking factors and initializations) are specified in each test.


\begin{table}[t!]
\centering
\begin{tabular}{|c|c|c|c|c|c|c|c|}
\hline
\textbf{Image}   & \textbf{Size}   & \textbf{Range} & \textbf{$\sigma_{PSF}$} & \textbf{$b$} & \textbf{$\lambda$} & \textbf{$\mu$} & \textbf{$L_f$} \\ \hline\hline
\texttt{moon}    & $537\times 358$ & $[0,400]$      & $0$                     & $0.01$    & $0.15$             & $\mu_f\in\left\{0,0.0025\right\}$ & $100$ \\ \hline
\texttt{phantom} & $256\times 256$ & $[0,1]$        & $1.4$                   & $0.01$    & $0.004$            & $\mu_g=0.0001$ & $7701$       \\ \hline
\texttt{micro}   & $128\times 128$ & $[0,92]$       & $3.2$                   & $0.5$        & $0.092$            & $\mu_g=0.0001$ & $368$      \\ \hline
\texttt{mri}     & $128\times 128$ & $[0,170]$      & $3.2$                   & $0.5$        & $0.001$            & $\mu_g=0.0005$  & $680$     \\ \hline
\end{tabular}%
\caption{Simulation and model parameters  for the test images in Figure \ref{fig:test_images}.}
\label{table:details_figs}
\end{table}

\paragraph{Inexact computation of proximal points with $\epsilon_{k}$-accuracy.} Following \cite[Section 4.2]{Bonettini2018a}, we now sketch a general strategy ensuring that the inexactness condition required by Definition \ref{def:eps_approx} is fulfilled in practice. For that, we require that the non-smooth function $g$ in \eqref{minf} can be expressed in the form:
\begin{eqnarray}  \label{eq:g_form}
g(x) = \sum_{i=1}^p \phi_i (M_ix) + \psi(x),
\end{eqnarray}
where $M_i:\mathcal{H}\to \mathcal{Z}_i$ are linear bounded operators between Hilbert spaces and $\phi_i:\mathcal{Z}_i\to \mathbb{R}\cup\left\{\infty\right\}$, $\psi:\mathcal{H}\to \mathbb{R}\cup\left\{\infty\right\}$ are proper, l.s.c. and convex functions. For image restoration problems, the couple $(\phi_i, M_i)$ is typically associated to the use of edge-preserving regularization terms (e.g. \eqref{eq:TV} by setting $\phi_i:\mathbb{R}^2\to\mathbb{R}, \phi_i(v) = \lambda\|v\|$ and $M_i = \nabla_i$), while the function $\psi$ may encode further requirements on the desired signal such as a positivity constraint or a convex perturbation term.

At each iteration $k\geq 1$ of SAGE-FISTA, the inner (primal) subproblem that needs to be solved to compute the proximal point $x^{(k)}$ has the form:
\begin{equation}  \label{eq:primal_pb}
   \min_{x\in \mathcal{H}}~ \left\{ \mathcal{P}_{\tau_{k},D_{k}}(x) := \sum_{i=1}^p \phi_i (M_ix) + \psi(x) + \frac{1}{2\tau_{k}} \|x- \bar{y}^{(k)} \|^2_{D_{k}} \right\},
\end{equation}
where $\bar{y}^{(k)}:=y^{(k)}-\tau_{k} D_{k}^{-1}\nabla f (y^{(k)})$.
Denoting by $\mathcal{Z}= \mathcal{Z}_1\times\ldots\times \mathcal{Z}_p$ the Hilbert product space equipped with the inner product $\langle w, v\rangle=\sum_{i=1}^p \langle w_i, z_i\rangle$ and by $M:\mathcal{H}\to\mathcal{Z}$ the linear bounded operator  $Mx = (M_1x, \ldots, M_p x)\in\mathcal{Z}$, the dual formulation of \eqref{eq:primal_pb} reads as
\begin{equation}  \label{eq:dual_pb}
    \max_{w\in\mathcal{Z}}~\left\{ \mathcal{Q}_{\tau_{k},D_{k}}(w) := -\sum_{i=1}^p \phi_i^*(w_i) + \Phi_{\tau_{k},D_{k},\bar{y}^{(k)}}(w)\right\}
\end{equation}
where $\phi_i^*$ is the Fenchel conjugate of $\phi_i$ and $\Phi_{\tau_{k},D_{k},\bar{y}^{(k)}}(w)$ is defined as:
\begin{align*}
    \Phi_{\tau_{k},D_{k},\bar{y}^{(k)}}(w)  &:=  \psi\left( \prox_{\tau_{k}\psi}^{D_{k}} (\bar{y}^{(k)}-\tau_{k}D^{-1}_{k} M^* w) \right)   - \frac{1}{2\tau_{k}} \|\bar{y}^{(k)}-\tau_{k}D^{-1}_{k} M^* w  \|^2_{D_{k}}  + \frac{1}{2\tau_{k}} \|\bar{y}^{(k)}\|^2_{D_{k}}\\
    & +\frac{1}{2\tau_{k}} \| \prox_{\tau_{k}\psi}^{D_{k}} (\bar{y}^{(k)}-\tau_{k}D^{-1}_{k} M^* w) - (\bar{y}^{(k)}-\tau_{k}D^{-1}_{k} M^* w ) \|^2_{D_{k}}.
\end{align*}
Upon suitable assumptions (see, e.g. \cite[Section 4.1]{Bonettini2018a}) problems \eqref{eq:primal_pb} and \eqref{eq:dual_pb} are equivalent. Recalling definition \eqref{eq:function_h}, one can further deduce the following inequality:
\begin{align*}
h_{\tau_k,D_k}(x;\bar{y}^{(k)})-h_{\tau_k,D_k}(\prox_{\tau_{k}\psi}^{D_{k}} (\bar{y}^{(k)});\bar{y}^{(k)})&=\mathcal{P}_{\tau_{k},D_{k}}(x) - \mathcal{P}_{\tau_{k},D_{k}} \left( \prox_{\tau_{k}\psi}^{D_{k}} (\bar{y}^{(k)})\right) \\
&\leq   \mathcal{P}_{\tau_{k},D_{k}}(x) - \mathcal{Q}_{\tau_{k},D_{k}}(w), \quad \forall \ w\in\mathcal{Z}.
\end{align*}
Hence, a sufficient condition for a point $x$ to be an $\epsilon_{k}$-approximation as in Definition \ref{def:eps_approx} is the existence of a dual point $w$ such that $\mathcal{P}_{\tau_{k},D_{k}}(x) - \mathcal{Q}_{\tau_{k},D_{k}}(w)\leq \epsilon_{k}$. Assuming that $g$ is continuous on its domain and $\text{dom}(\psi)=\text{dom}(g)$, such pair $(x,w)$ can be computed by generating a dual sequence $\left\{ w ^{(k,l)}\right\}_{l\in\mathbb{N}}\subset \mathcal{Z}$ converging to the solution of \eqref{eq:dual_pb} with
$$
\lim_{l\to\infty} ~\mathcal{Q}_{\tau_{k},D_k}\left( w ^{(k,l)}\right) = \max_{z\in\mathcal{Z}}~\mathcal{Q}_{\tau_{k},D_k}(z),
$$
a corresponding primal sequence $\left\{ x ^{(k,l)}\right\}_{l\in\mathbb{N}}\subset \mathcal{H}$ computed as
\begin{equation}
        x^{(k,l)} := \prox_{\tau_{k}\psi}^{D_k}\left(\bar{y}^{(k)} - \tau_{k}D_k^{-1} M^* w^{(k,l)} \right),\qquad \forall l\in\mathbb{N},\label{eq:primal_seq}
\end{equation}
and then stopping the iterates at the first nonnegative integer $l$ such that
\begin{equation}  \label{eq:criterion_eps}
\mathcal{P}_{\tau_{k},D_k}(x^{(k,l)}) - \mathcal{Q}_{\tau_{k},D_k}(w^{(k,l)})\leq \epsilon_k.
\end{equation}
Such a procedure is well-defined, as condition \eqref{eq:criterion_eps} holds for all sufficiently large $l$ \cite[Proposition 4.2]{Bonettini2018a}. Note that the dual sequence $\left\{w^{(k,l)} \right\}_{l\in\mathbb{N}}$ can be generated by means of a forward--backward scheme solving \eqref{eq:dual_pb}, e.g. using an efficient inner FISTA routine, provided that the extrapolation parameters are chosen in a way that weak convergence of the iterates is guaranteed (see, e.g.~ \cite{Chambolle-Dossal-2014}).

\paragraph{Variable metric selection}

As far as the choice of the scaling matrices $\left\{D_k\right\}_{k\in\mathbb{N}}$ is concerned, we exploit the split-gradient technique studied in \cite{Lanteri-etal-2001} and later used in several works (see, e.g., \cite{Bonettini-Loris-Porta-Prato-2015,Bonettini2018a}) where a decomposition of $\nabla f$ into a non-negative part and a positive one is used. The idea beyond this strategy consists in decomposing the gradient of the smooth component $f$ of problem \eqref{minf} as $-\nabla f(x) = U(x) - V(x)$, where $U(x)\geq 0$ and $V(x)>0$, and then defining for each $k\geq 1$ the scaling matrix $D_k=\text{diag}\left( y^{(k)}/V(y^{(k)}) \right)^{-1}$, so that the forward iteration can be written as:
\begin{equation} \label{eq:classical_form}
    \bar{y}^{(k)} = y^{(k)}- \tau_{k}D_k^{-1}\nabla f(y^{(k)}) =  y^{(k)}+ \tau_{k} y^{(k)} \cdot \left( \frac{U(x^{(k)})}{V(x^{(k)})} -1 \right),
\end{equation}
where the fraction and product symbols denote component-wise division and product, respectively. Note that several classical reconstruction algorithms in imaging can be cast in the form \eqref{eq:classical_form}, such as the Richardson-Lucy (RL) algorithm for Poisson image deconvolution (see, e.g., \cite{Bonettini2019,Lanteri-etal-2001}). In order to ensure the conditions required by Assumption \ref{ass:1}, we introduce an appropriate thresholding parameter in the definition of $D_k$, thus obtaining
\begin{equation}  \label{eq:scaling_general}
    D_k = \text{diag}\left(\max\left(\frac{1}{\gamma_k},\min\left(\gamma_k,  \frac{y^{(k)}}{V(y^{(k)})} \right)\right)\right)^{-1},
\end{equation}
for  parameters $\gamma_k$ defined as in \cite{Bonettini-Loris-Porta-Prato-2015,Bonettini2018a} by:
\begin{equation} \label{eq:gamma_k_pb1}
    \gamma_k = \sqrt{1+\frac{s_1}{(k+1)^{s_2}}},\quad \text{where }s_1 >0, \ s_2>1.
\end{equation}
Clearly, the choice \eqref{eq:scaling_general} is problem-dependent, due to presence of the the function $V(\cdot)$. Note that in \eqref{eq:gamma_k_pb1}, when $s_1=0$, we have $D_k\equiv\mathcal{I}$, hence the standard Euclidean metric is recovered. In general, it is good practice to choose a large value of $s_1$ in order to benefit from the use of the variable metric in the early iterations of the algorithm, while leaving the asymptotic convergence behavior to be driven by $s_2$.

\subsection{Weighted-$\ell^2$-TV image denoising} \label{sec:Pden}

As a first test, we  set $H=\mathcal{I}\in\mathbb{R}^{n\times n}$  in \eqref{eq:Poisson_recon}  and consider a Poisson image denoising problem where TV regularization \eqref{eq:TV} is combined with a data fidelity term corresponding to the second order Taylor approximation of the KL divergence \eqref{eq:KL_def} around the noisy data $z$, see, e.g.~ \cite{Sawatzky2011,Burger_2014,WCEL0}. For a given Poisson noisy image $z\in\mathbb{R}^n_{\geq 0}$ and for a fixed regularization parameter $\lambda>0$, we consider the problem:
\begin{equation}  \label{eq:wl2_tv}
    x^* = \argmin_{x\in \mathbb{R}^n}~\left\{ F(x) : = \frac{1}{2}\sum_{i=1}^n \frac{(x_i-z_i+b)^2}{z_i+b} + \lambda TV(x) + \iota_{Y}(x)\right\},
\end{equation}
where $b>0$ is a constant background term. By weighting the least-square residuals by the inverse values of the noisy image, the use of the data fidelity term in \eqref{eq:wl2_tv} promotes high fidelity in low-intensity pixels (low noise) and large regularization in high-intensity pixels (high noise). Applying the splitting \eqref{minf} to problem \eqref{eq:wl2_tv}, we can thus set:
\begin{equation}  \label{eq:pb1_splitting}
    f(x):=\frac{1}{2}\sum_{i=1}^n \frac{(x_i-z_i+b)^2}{z_i+b} ,\qquad g(x):=\lambda TV(x)+ \iota_{Y}(x)
\end{equation}
and observe that $f$ is $\sigma_f$-strongly convex with $L_f$-Lipschitz continuous gradient with:
\begin{equation}  \label{eq:parametersTEST1}
    \sigma_f = \frac{1}{\max\limits_{i=1,\ldots,n} z_i+b},\qquad \nabla f(x) = \frac{x-z+b}{z+b},\qquad L_f= \frac{1}{\min\limits_{i=1,\ldots,n} z_i+b}.
\end{equation}

We now apply SAGE-FISTA Algorithm \ref{algo:GFISTA} with $ \mu_g=0$ and either $\mu_f=\sigma_f$ or $\mu_f=0$. In particular, we report the results obtained when strong convexity is taken into account (SAGE-FISTA with $\mu_f=\sigma_f$) and when it is not (S-FISTA with $\mu_f=0$). For this first test, we consider the \texttt{moon} noisy image in Figure \ref{fig:test_images}.  From \eqref{eq:parametersTEST1} and Table \ref{table:details_figs}, we have that  $L_f=100$ in this case. 

\paragraph{Scaled VS. non-scaled SAGE-FISTA with backtracking.}
First, we compare the performance of the non-scaled version of SAGE-FISTA with its scaled counterparts for different choices of the scaling matrices $\left\{D_k\right\}_{k\in\mathbb{N}}$. We apply Algorithm \ref{algo:GFISTA} with  $\lambda=0.15$, maximum number of outer iterations \texttt{maxiter}=$500$, maximum number of inner backtracking iterations \texttt{max\_bt}=$10$, $\rho=0.8$ (decreasing backtracking factor), $L_0=30$ and $t_0=1.01$. In order to compare the adaptive backtracking strategy with the non-adaptive one, we choose $\delta\in\left\{0.99,1\right\}$.
We recall that for $\delta=1$, a standard Armijo backtracking is performed, while $\delta<1$ allows for local increasing of the algorithmic step-size. Finally, SAGE-FISTA is initialized by setting $x^{(0)}=z$.

In order to compute a suitable $\epsilon_k$-approximation $x^{(k)}$ of the proximal-gradient point, we follow the strategy described above with the choice $\phi_i:\mathbb{R}^2\to \mathbb{R},~ \phi_i(v)=
\lambda \| v\|_2$, $M_i = \nabla_i, i=1,\ldots,n$ and $\psi: \mathbb{R}^n \to \mathbb{R}\cup\left\{\infty\right\}, \psi(x) = \iota_Y(x)$.
Then, for each inner iteration $l\in\mathbb{N}$ of the primal-dual subroutine, the primal iterate \eqref{eq:primal_seq} is given by
\begin{equation}
x^{(k,l)}   = P_{Y,D_k} \left(y^{(k)}-\tau_k D_{k}^{-1}\left(\nabla f(y^{(k)})   + M^* w^{(k,l)} \right) \right).  \label{eq:projection_psi}
\end{equation}
Concerning the sequence $\left\{\epsilon_k\right\}_{k\in\mathbb{N}}$, we apply the result provided by Corollary \ref{cor:guarantee1}, choosing in particular $a_{k} = 1/k^{2.1}$ for all $k\geq 1$ and updating $\theta_{k}$ in the inner backtracking routine by using the current estimates of $t_{k}$ and $q_{k}$ appearing in the definition \eqref{eq:tau_prime_omega_} of the elements $\omega_{k}$. 

The sequence of scaling matrices $\left\{D_k\right\}_{k\in\mathbb{N}}$ is chosen by applying the diagonal split-gradient strategy described above, which, for  $k\geq 1$, corresponds to the choice:
\begin{equation}  \label{eq:metric_wl2}
    D_k = \text{diag} \left( \max \left( \frac{1}{\gamma_k}, \min \left( \gamma_k, z + b\right) \right) \right)^{-1},
\end{equation}
where $\gamma_k$ are the thresholding parameters in \eqref{eq:gamma_k_pb1}.
We notice that since $H=\mathcal{I}$ in this example, the expression of the terms $D_k$ in \eqref{eq:metric_wl2} is very handy as it depends only on the noisy image $z$ and on the background term $b$.

We now apply SAGE-FISTA with and without scaling for solving problem \eqref{eq:wl2_tv} whenever different choices of parameters $s_1$ and $s_2$ in \eqref{eq:gamma_k_pb1} are considered. In Figure \ref{fig:TEST1_GFISTA_armijo}, we compare the results obtained with standard Armijo backtracking ($\delta=1$) and adaptive backtracking ($\delta=0.99$), respectively. This choice allows for local non-monotone adjustments of the estimates $L_k$ of $L_f$, see the plots for the Lipschitz estimates in Figure \ref{fig:TEST1_GFISTA_armijo}c. We observe that the scaling procedure significantly improves convergence rates for any selected combination of the thresholding parameters $(s_1,s_2)$, while remaining affordable in terms of computational times. This is a well-known behavior for scaled non strongly-convex variants of FISTA \cite{Bonettini2018a}; our experiments confirm that the same holds in strongly-convex regimes.

\begin{figure}[t!]

\hspace{-5mm}
\begin{tabular}{c@{}c@{}c}
 \includegraphics[height=3.9cm]{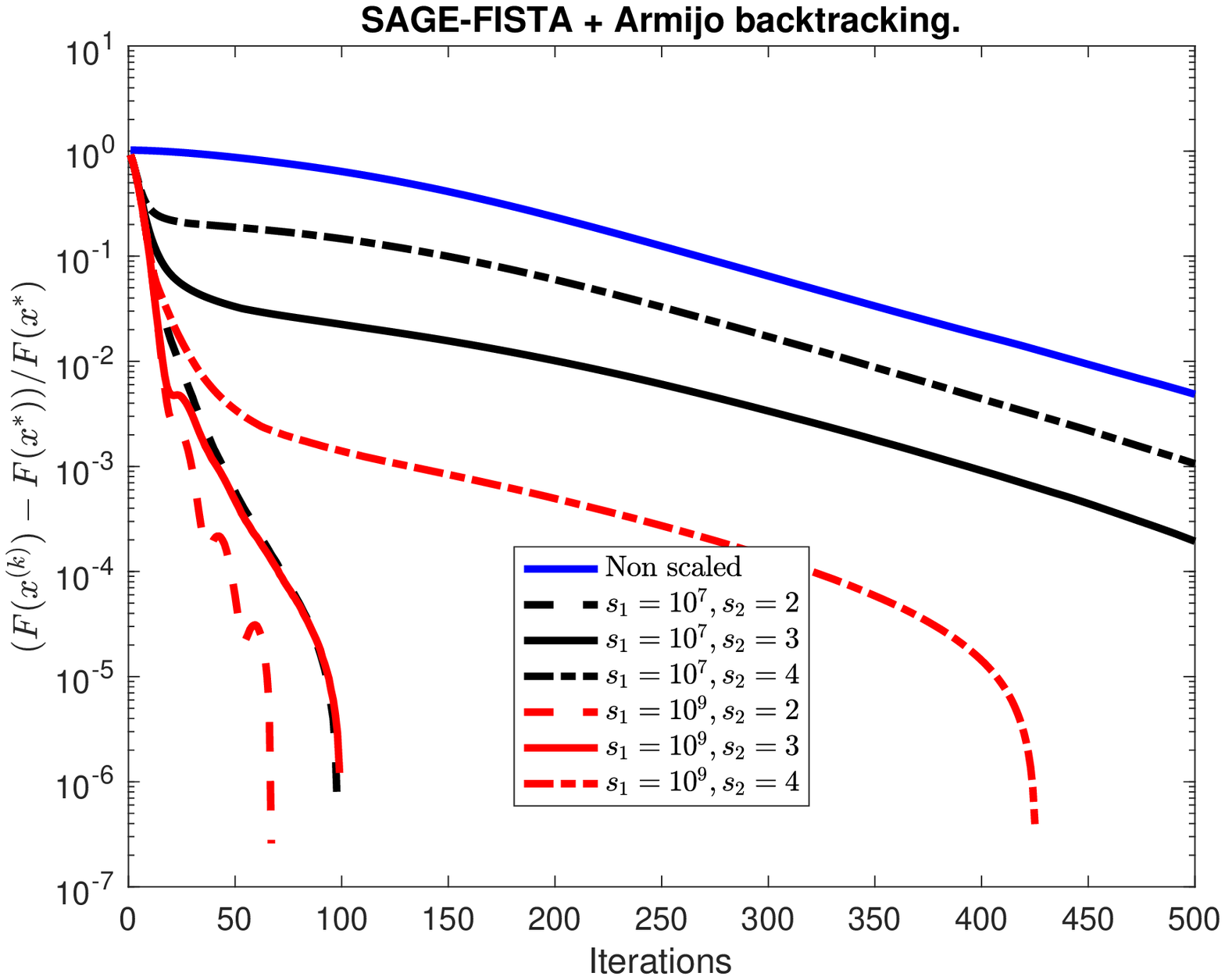} & \includegraphics[height=3.9cm]{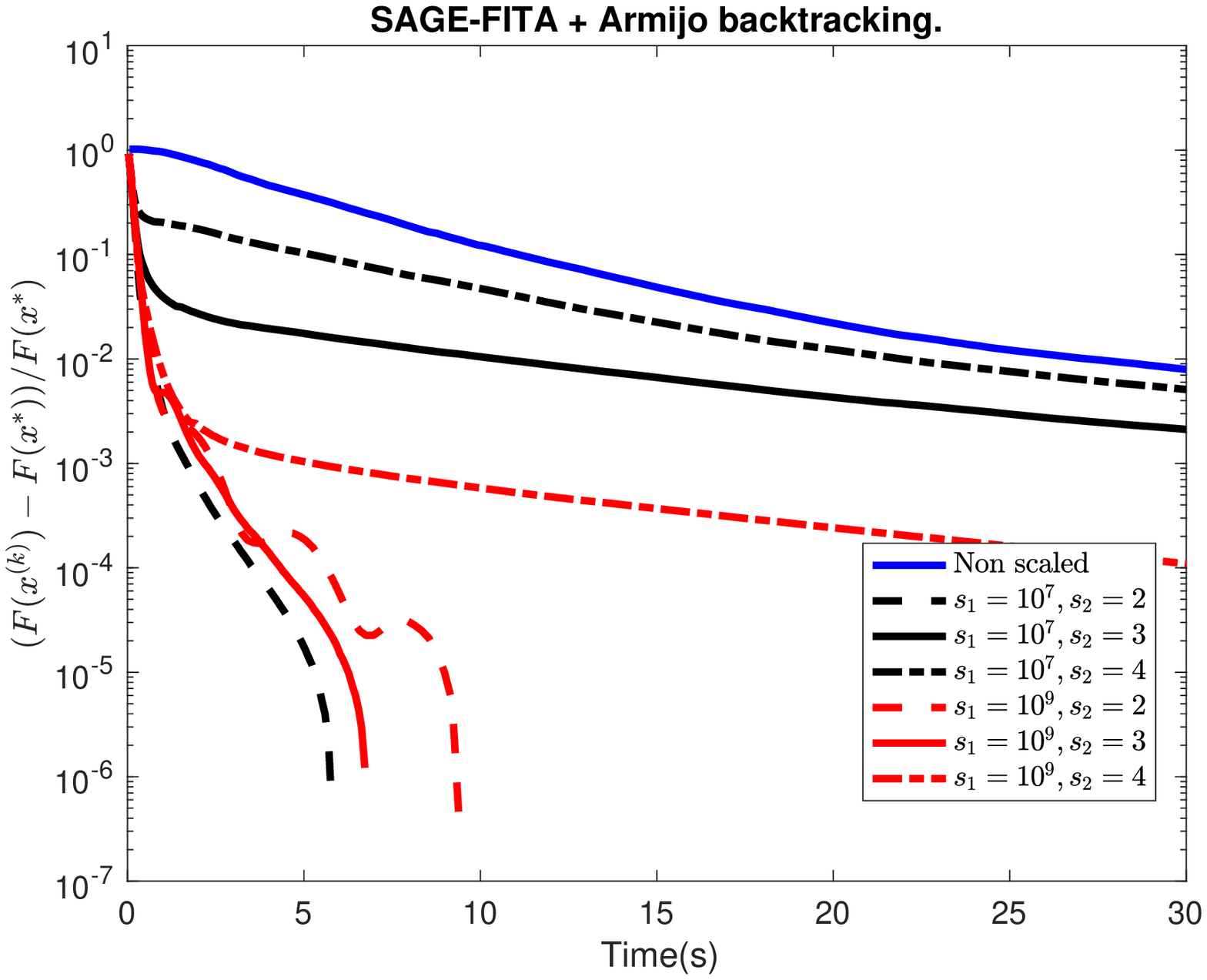} & \includegraphics[height=3.9cm]{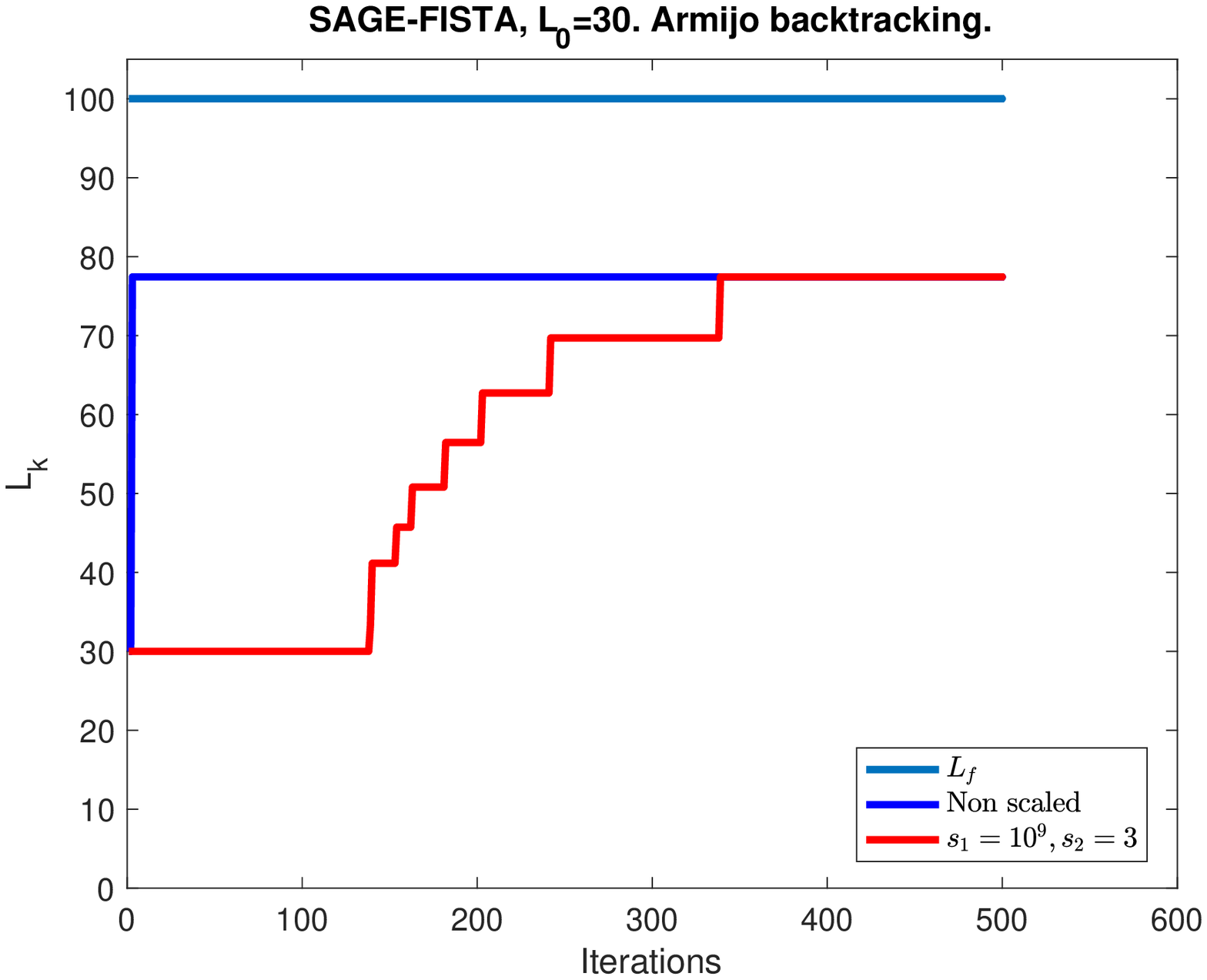} \\
  \includegraphics[height=3.9cm]{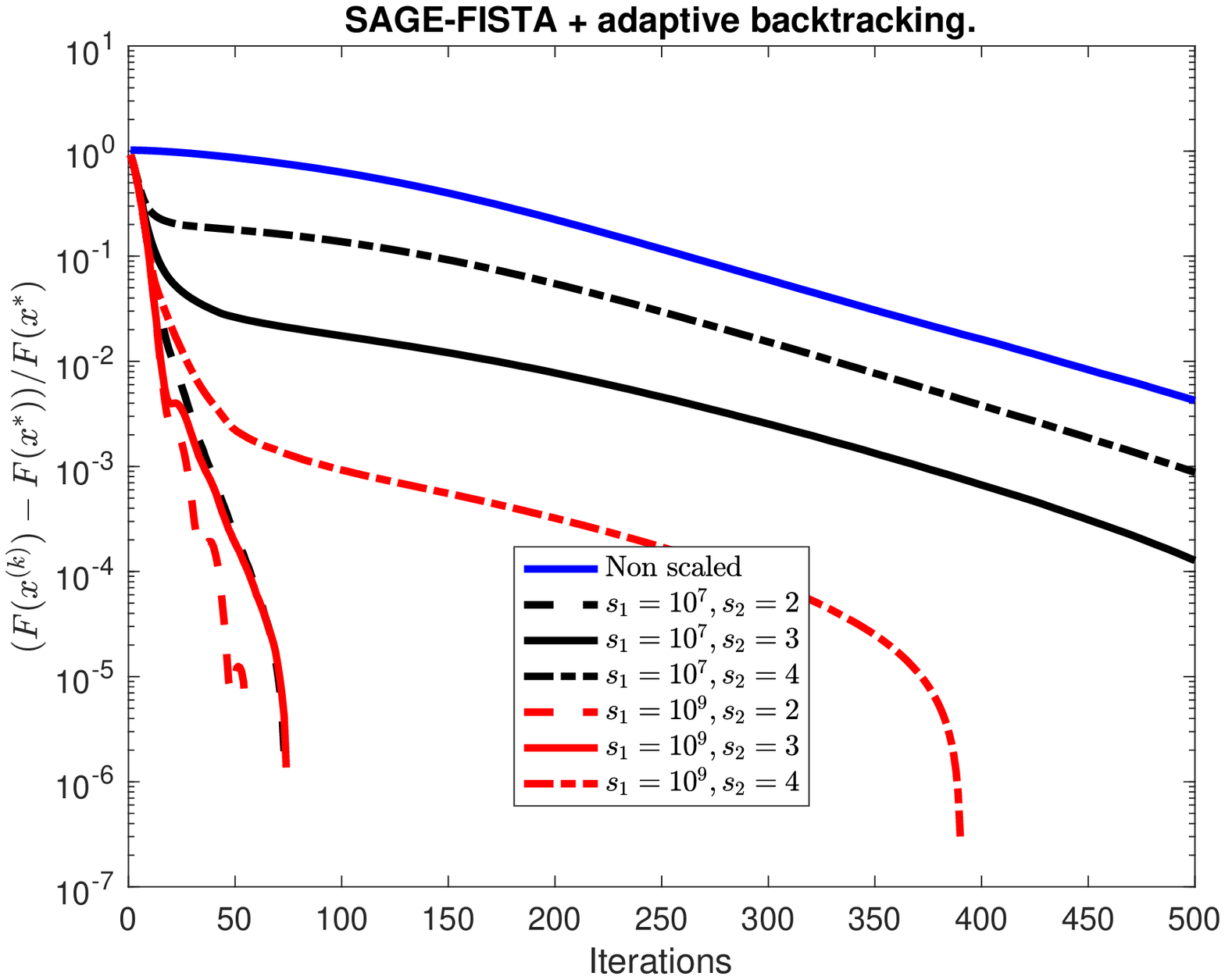} & \includegraphics[height=3.9cm]{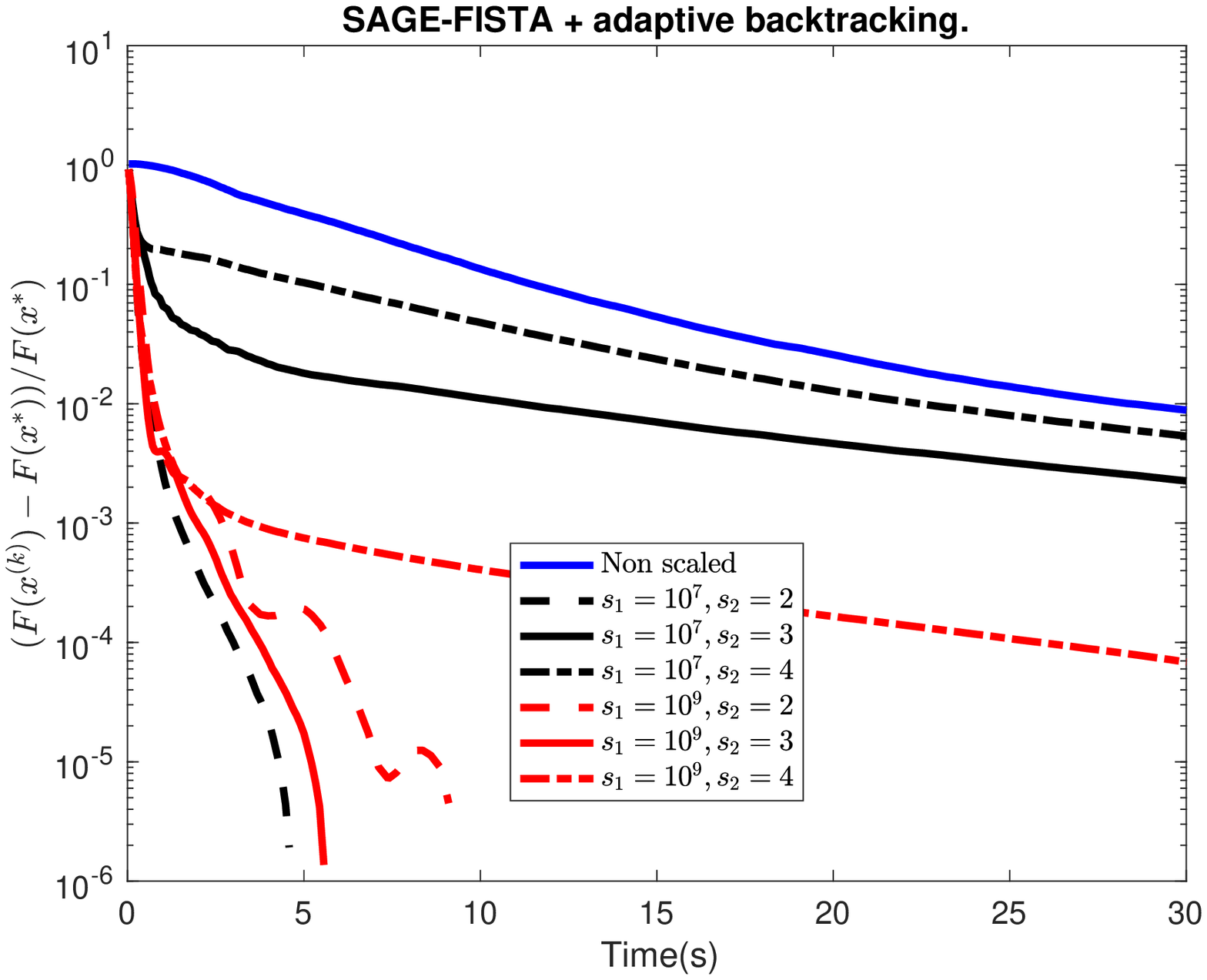} & \includegraphics[height=3.9cm]{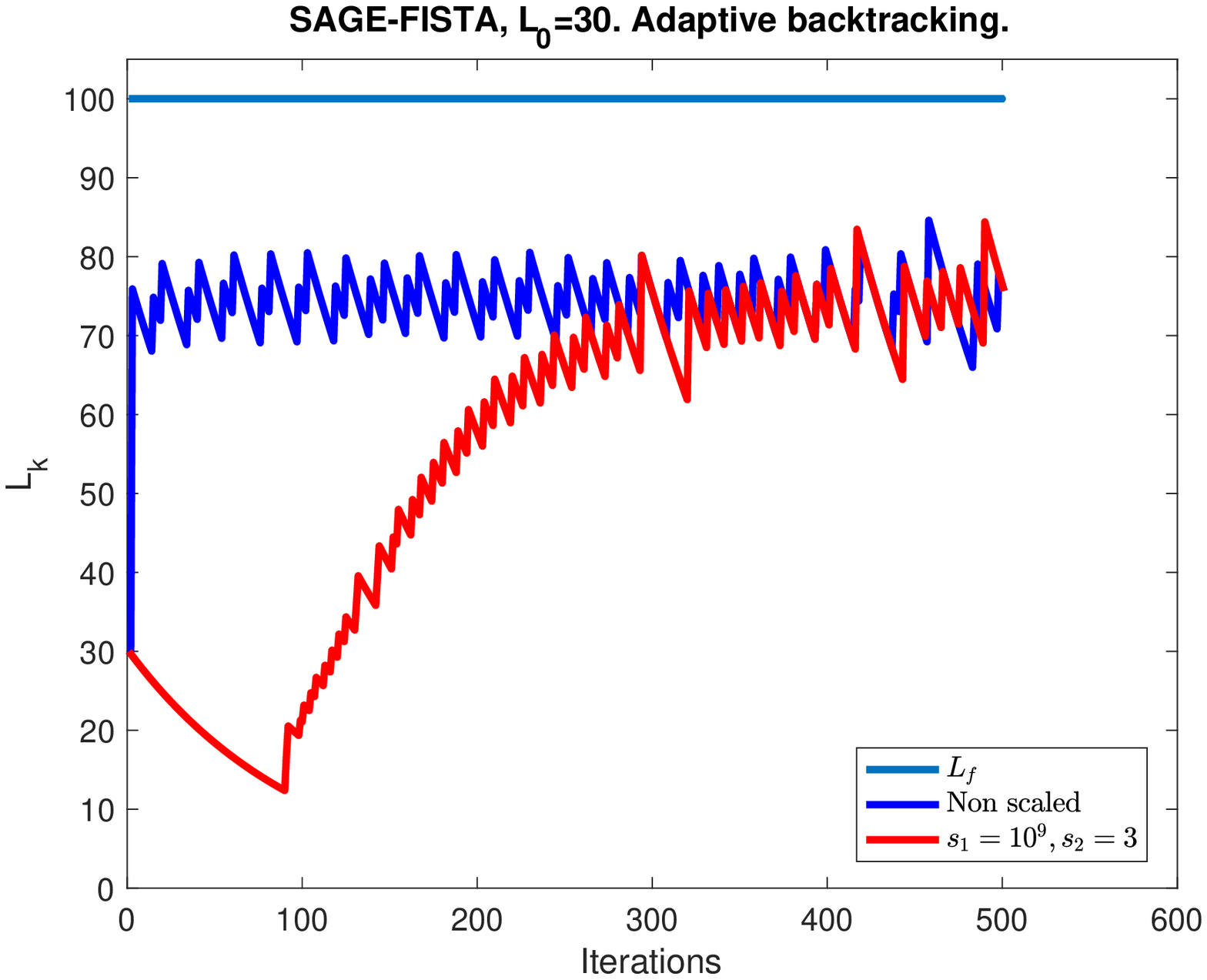} \\
   (a) Relative rates VS it. & (b) Relative rates VS first 30 secs. & (c) $L_k$ estimates.
\end{tabular}
    \caption{Performance of non-scaled VS. scaled SAGE-FISTA with Armijo backtracking (top row) and adaptive backtracking (bottom row) with $L_0=30$ for different choices of $s_1$ and $s_2$ for problem \eqref{eq:wl2_tv} on \texttt{moon} image.}
    \label{fig:TEST1_GFISTA_armijo}
\end{figure}

\paragraph{SAGE-FISTA VS. S-FISTA with backtracking.} In order to show the effectiveness of SAGE-FISTA also when strong convexity is not directly taken into account, we now run the same experiments as above by setting $\mu_f=0$ while keeping all the other parameters fixed. Note that this corresponds to apply the inexact and scaled FISTA (denoted here by S-FISTA) algorithm considered in \cite{Bonettini2018a} in analogous contexts, the novelty being here the use of the adaptive backtracking strategy in comparison to the standard Armijo one.
For S-FISTA, we compute the scaling matrices $\{D_k\}_{k\in\mathbb{N}}$ either according to the variable strategy \eqref{eq:metric_wl2} or setting $D_k\equiv D$ with $D = \text{diag}\left(z+b\right)^{-1}$. This constant choice is motivated by the fact that, for problem \eqref{eq:wl2_tv}, the update of the matrices $D_k$ in \eqref{eq:metric_wl2} depends only on the presence of the parameters $\gamma_k$, which can be possibly kept fixed (for instance by setting $s_1 \gg 1$ and $s_2 =0$), thus imposing a constant Newton-type scaling matrix along the iterations. By Remark \ref{eq:rem3}, this choice still satisfies Assumption \ref{ass:1}, thus the convergence result of Theorem \ref{thm:rate} still applies. Our results are reported in Figure \ref{fig:TEST1_FISTA_adaptive_compScaling}.  As expected, this choice improves even more the convergence rate with respect to the number of iterations, at the price of significantly higher computational costs. Indeed, since the matrices $D_k$ do not converge any longer to the identity matrix, the Lipschitz constant of the inner dual problem \eqref{eq:dual_pb} tends to be bigger, being proportional to $\|D_k^{-1}\|$ (see \cite[Section 4.2]{Bonettini2018a}), thus affecting negatively the inner FISTA subroutine for the computation of the $\epsilon_k-$approximation when the parameter $\epsilon_k$ is small. In support of our remark, we report the total CPU time to execute $500$ iterations: S-FISTA with Armijo backtracking takes 106.5 s (variable metric) and 3591 s (constant metric), whereas S-FISTA with adaptive backtracking takes 238.4 s (variable metric) and 2672.6 s (constant metric).
We conclude that, for suitable choices of $s_1$ and $s_2$, the variable metric strategy \eqref{eq:metric_wl2} represents a good balance between the good convergence properties of Newton methods and the higher computational costs.

\begin{figure}[t!]
         \centering
         \begin{subfigure}[b]{0.45\textwidth}
             \centering
             \includegraphics[height=5cm]{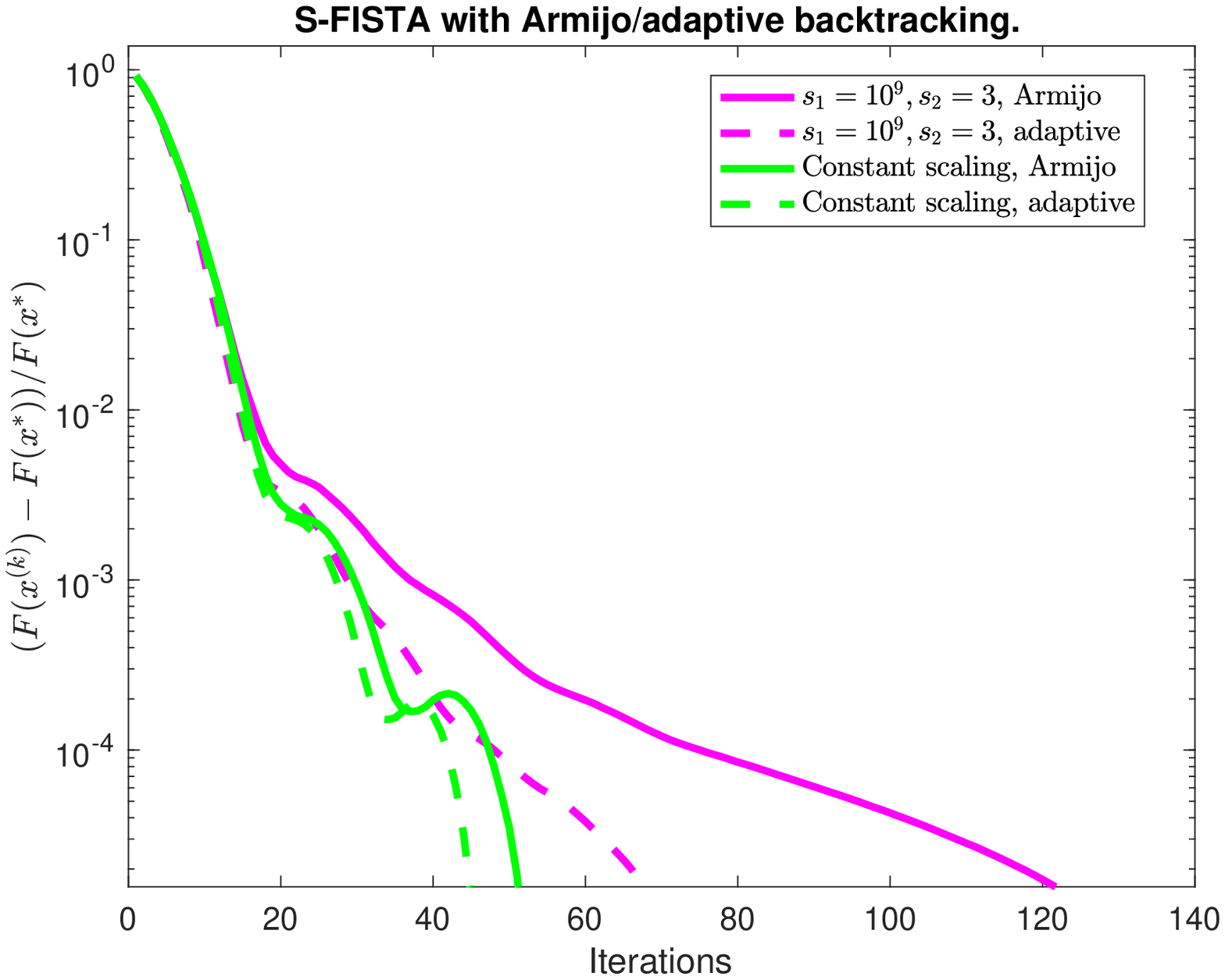}
             \caption{Variable VS. constant metric.}
             \label{fig:variable_constant}
         \end{subfigure}
         \begin{subfigure}[b]{0.45\textwidth}
             \centering
             \includegraphics[height=5cm]{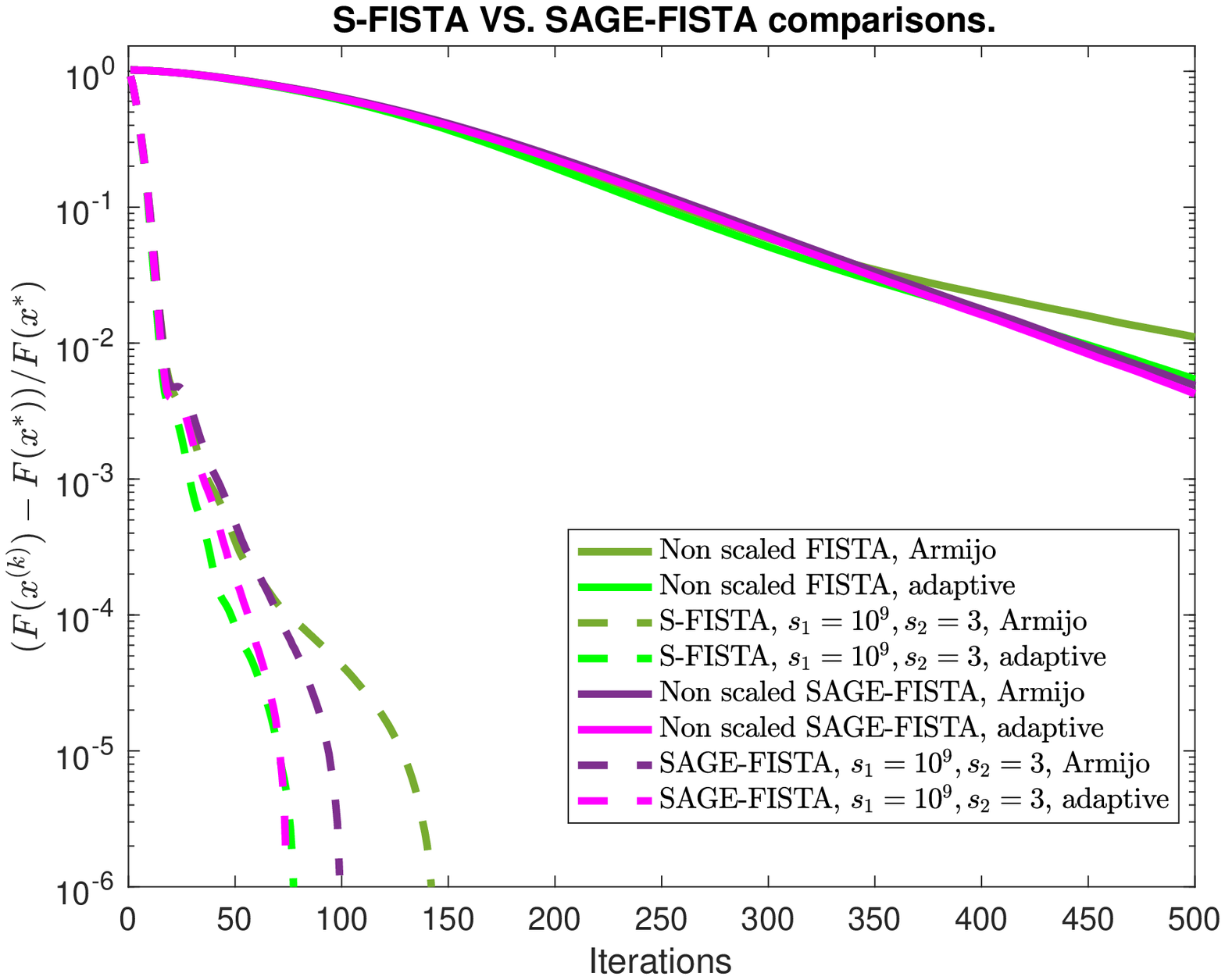}
             \caption{Relative convergence rates comparisons.}
             \label{fig:general_comparisonSFISTA}
        \end{subfigure}
    \caption{Figure \ref{fig:variable_constant}: S-FISTA  with Armijo/adaptive backtracking for variable and constant metric. Figure \ref{fig:general_comparisonSFISTA}: comparison between non-scaled and scaled convex (S-FISTA) and strongly-convex (SAGE-FISTA) algorithms with Armijo/adaptive backtracking. Problem \eqref{eq:wl2_tv} on \texttt{moon} image. }
   \label{fig:TEST1_FISTA_adaptive_compScaling}
\end{figure}

As a final comparison, we report in Figure \ref{fig:general_comparisonSFISTA} the relative convergence rates for the various scaled and non-scaled versions of S-FISTA and SAGE-FISTA, combined with either an Armijo or an adaptive backtracking strategy. 
As it is expected, encoding explicitly the prior knowledge of strong convexity (SAGE-FISTA) improves convergence rates in comparison with its convex variant (S-FISTA). As far as the two backtracking strategies are concerned, the use of an adaptive strategy favoring non-monotone adjustments of the values $L_k$ further improves convergence speed with respect to standard monotone strategies, as already observed in \cite{Calatroni-Chambolle-2019} for GFISTA.

\subsection{Poisson image deblurring with strongly convex TV regularization} \label{sec:Pdeb}

As a second example, we consider a TV image deblurring problem for images corrupted by Poisson noise. For this test, we consider the KL divergence functional \eqref{eq:KL_def} and combine it with a TV-type regularization term, where a quadratic perturbation weighted by a parameter $\varepsilon\ll 1$ is added to make the composite problem strongly-convex. By further encoding also a non-negativity constraint, the final problem reads
\begin{equation}  \label{eq:pb_KL_TVsc}
    x^* = \argmin_{x\in\mathbb{R}^n} \left\{ F(x): = KL(Hx+b;z) + \lambda TV(x) + \frac{\varepsilon}{2} \| x\|_2^2 + \iota_{Y}(x)\right\}.
\end{equation}
Recalling splitting \eqref{minf}, we set here:
\begin{equation}  \label{eq:splitting_PoisDeb}
    f(x):=KL(Hx+b;z),\qquad g(x):=\lambda TV(x) + \frac{\varepsilon}{2} \| x\|_2^2  + \iota_{Y}(x).
\end{equation}
It is well-known that $\nabla f$ is Lipschitz continuous on $Y$ and that an over-estimation of its Lipschitz constant $L_f$ can be provided \cite{Harmany12}. Furthermore, the function $g$ is $\varepsilon$-strongly convex. We thus have:
\begin{equation}  \label{eq:parameters_PoisDeb}
    \nabla f(x) = H^T e - H^T \left( \frac{z}{Hx + b}\right),\qquad L_f=\frac{\max z_i}{b^2} \max (H^T e) \max(H e), \qquad \mu_g=\varepsilon,
\end{equation}
where $e\in\mathbb{R}^n$ is a vector whose components are all equal to one. Due to \eqref{eq:parameters_PoisDeb}, we observe that the variability in the estimation of $L_f$ (see Table \ref{table:details_figs}) depends on the range on the data $z$, which in our examples is intentionally allowed to vary.
By assuming reflexive boundary conditions, the matrix-vector products defined in terms of $H$ and $H^T$ can be performed efficiently via discrete cosine transform. Under these choices, we apply the SAGE-FISTA Algorithm \ref{algo:GFISTA} with $\mu_f=0$ and $\mu=\mu_g$ on the \texttt{phantom}, \texttt{micro} and \texttt{mri} images reported in Figure \ref{fig:test_images}. For each image, the choice of model parameters $\lambda$ and $\varepsilon$ is reported in Table \ref{table:details_figs}. As far as the algorithmic parameters are concerned, we set for all the following tests \texttt{maxiter}$=300$, \texttt{max\_bt}$=10$, $\rho=0.85$ and $\delta\in\left\{1,0.98\right\}$ depending on whether Armijo/adaptive backtracking is considered, respectively. As for the previous test, we set $t_0=1.01$ and $x^{(0)}=z$. SAGE-FISTA is run with an initial estimate of $L_f$ chosen as $L_0=0.1$ for \texttt{phantom}, $L_0=1$ for \texttt{micro} and $L_0=100$ for \texttt{mri}.

Regarding the sequence of scaling matrices $\left\{D_k\right\}_{k\in\mathbb{N}}$, we consider the diagonal split-gradient strategy \eqref{eq:scaling_general} setting $V(y^{(k)})=H^T e$. In this case, unlike in \eqref{eq:metric_wl2}, the scaling matrix $D_k$ explicitly depends on the extrapolated point $y^{(k)}$.

Concerning the sequence $\left\{\epsilon_k\right\}_{k\in\mathbb{N}}$, we apply again the result provided by Corollary 3.1, with the same choice of $\{\theta_k\}_{k\in\mathbb{N}}$ and $\{a_k\}_{k\in\mathbb{N}}$ adopted in the previous test. As far as the inexact computation of the proximal operator of $g$ is concerned, we start noticing that the function $g$ in \eqref{eq:splitting_PoisDeb} can be cast in the form \eqref{eq:g_form} with the choices $\phi_i(v) = \lambda \|v\|_2, M_i =  \nabla_i, i=1,\ldots,n$ and $\psi(x) = \frac{\varepsilon}{2} \| x\|_2^2  + \iota_{Y}(x) $. Due to the presence of the quadratic perturbation, the computation of the sequence $x^{(k,l)}$ in \eqref{eq:primal_seq} can be performed similarly as in \eqref{eq:projection_psi}, but after a suitable rescaling defined in terms of the operators $D_k$. In particular, denoting by $\left\{w^{(k,l)}\right\}_{l\in\mathbb{N}}$ the dual sequence, we have that the inner primal iterates can be computed as follows:
\begin{align}
x^{(k,l)}  & =  \text{prox}^{D_k}_{\tau_k \psi}\left(y^{(k)}-\tau_k D_{k}^{-1}\left(\nabla f(y^{(k)})   + M^* w^{(k,l)}\right) \right)  \label{eq:primal_sequence_PDEB1} \\
& = \text{prox}^{D_k+\tau_k\varepsilon\mathcal{I}}_{\tau_k \iota_Y}\left(  \left(\frac{D_k}{D_k + \tau_k\varepsilon\mathcal{I}}\right)\left(y^{(k)}-\tau_k D_{k}^{-1}\left(\nabla f(y^{(k)})   + M^* w^{(k,l)} \right) \right) \right) \label{eq:primal_sequence_PDEB2} \\
&  = P_{Y,D_k}  \left( \left(\frac{D_k}{D_k + \tau_k\varepsilon\mathcal{I}}\right)\left(y^{(k)}-\tau_k D_{k}^{-1}\left(\nabla f(y^{(k)})   + M^* w^{(k,l)} \right) \right) \right),\qquad \forall l\in\mathbb{N}, \notag
\end{align}
where the fraction $D_k/(D_k + \tau_k\varepsilon\mathcal{I})$ has to be intended element-wise. We report the proof of the result that guarantees the equality between \eqref{eq:primal_sequence_PDEB1}  and  \eqref{eq:primal_sequence_PDEB2}  in Lemma \ref{lemma:rescaling} in Appendix \ref{appendix:1}.

We now report the results obtained by the use of SAGE-FISTA  with Armijo and adaptive backtracking in correspondence of the images \texttt{phantom} (Figure \ref{fig:phantom_GFISTA_armijo}) \texttt{micro} (Figure \ref{fig:micro_GFISTA_armijo}) and \texttt{mri} (Figure \ref{fig:mri_GFISTA_armijo}). For all the experiments, we observe a clear advantage in the use of the variable scaling algorithm in comparison to its non-scaled variant. This is particularly clear from plots (a) in the Figures, where relative error is shown to decay much faster along the iterations. Regarding the computational times (Plots (b)), we notice that, depending on the particular data at hand, the performance may vary, especially depending on the the choice of the parameters $s_1$ and $s_2$. Overall, however, we observe that the use of a scaling procedure allows to compute highly accurate solutions in a limited amount of CPU time. This is particularly useful in applications such as medical imaging  where the need of fast image analysis strongly benefits from the ability of computing an accurate result in a pre-determined amount of time (usually 5/10 secs). As far as the choice of the backtracking strategy is concerned, we observe from the tests performed on the \texttt{micro} and \texttt{mri} images that whenever the initial value $L_0=1/\tau_0$ is too large, the use of a monotone Armijo strategy does not allow for its adjustment. In this case, even in the presence of an appropriate scaling strategy, convergence speed may suffer. The use of an adaptive backtracking corrects this inconvenience by adjusting the values  $L_k$, which corresponds to considering larger step-sizes $\tau_k$ favoring faster convergence. This is particularly evident for problems like  \eqref{eq:pb_KL_TVsc}, where the estimation $L_f$ in \eqref{eq:parameters_PoisDeb} significantly overestimates the admissible upper value for which convergence is guaranteed.

\begin{figure}[t!]
\hspace{-5mm}
\begin{tabular}{c@{}c@{}c}
     \includegraphics[height=3.9cm]{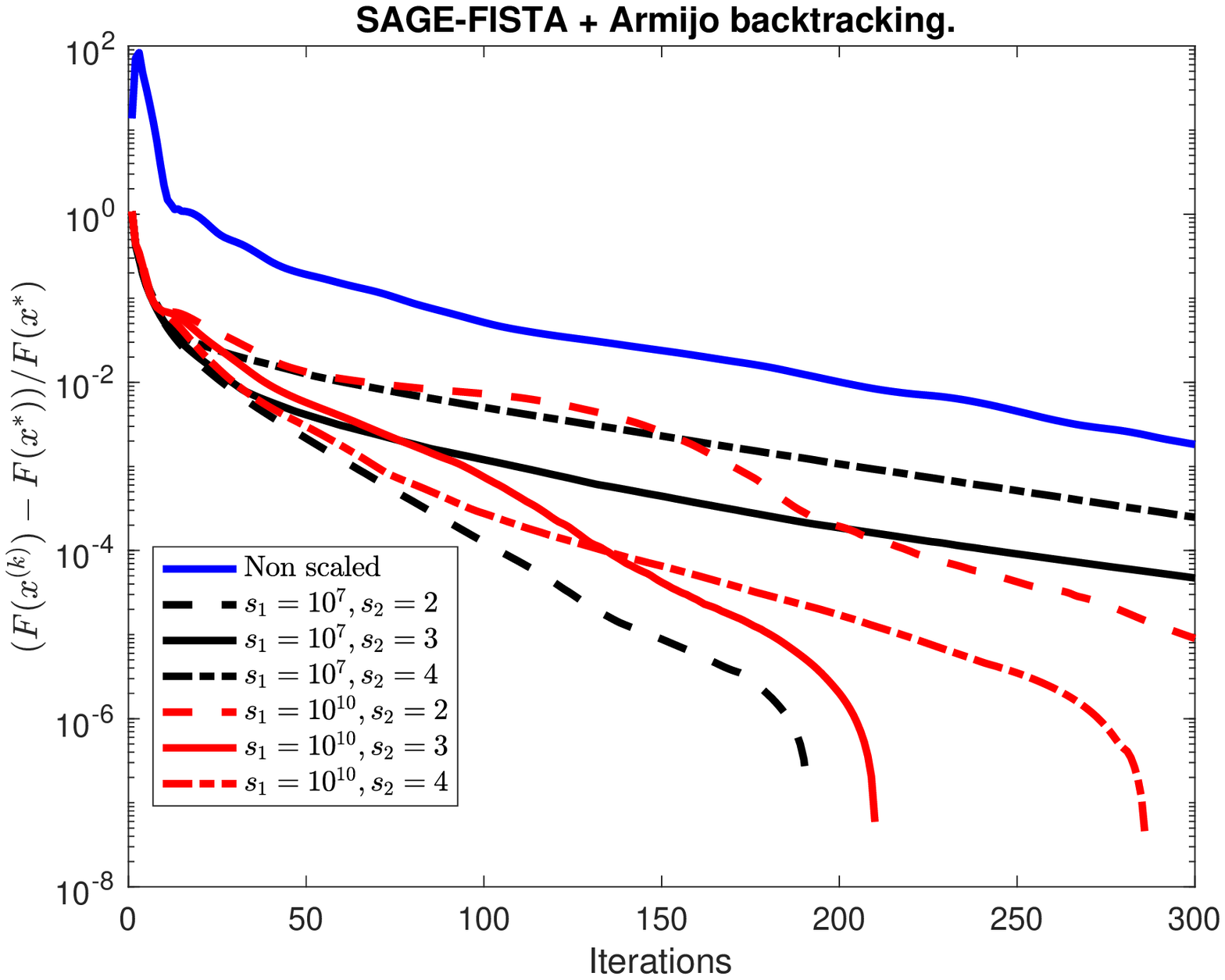} &  \includegraphics[height=3.9cm]{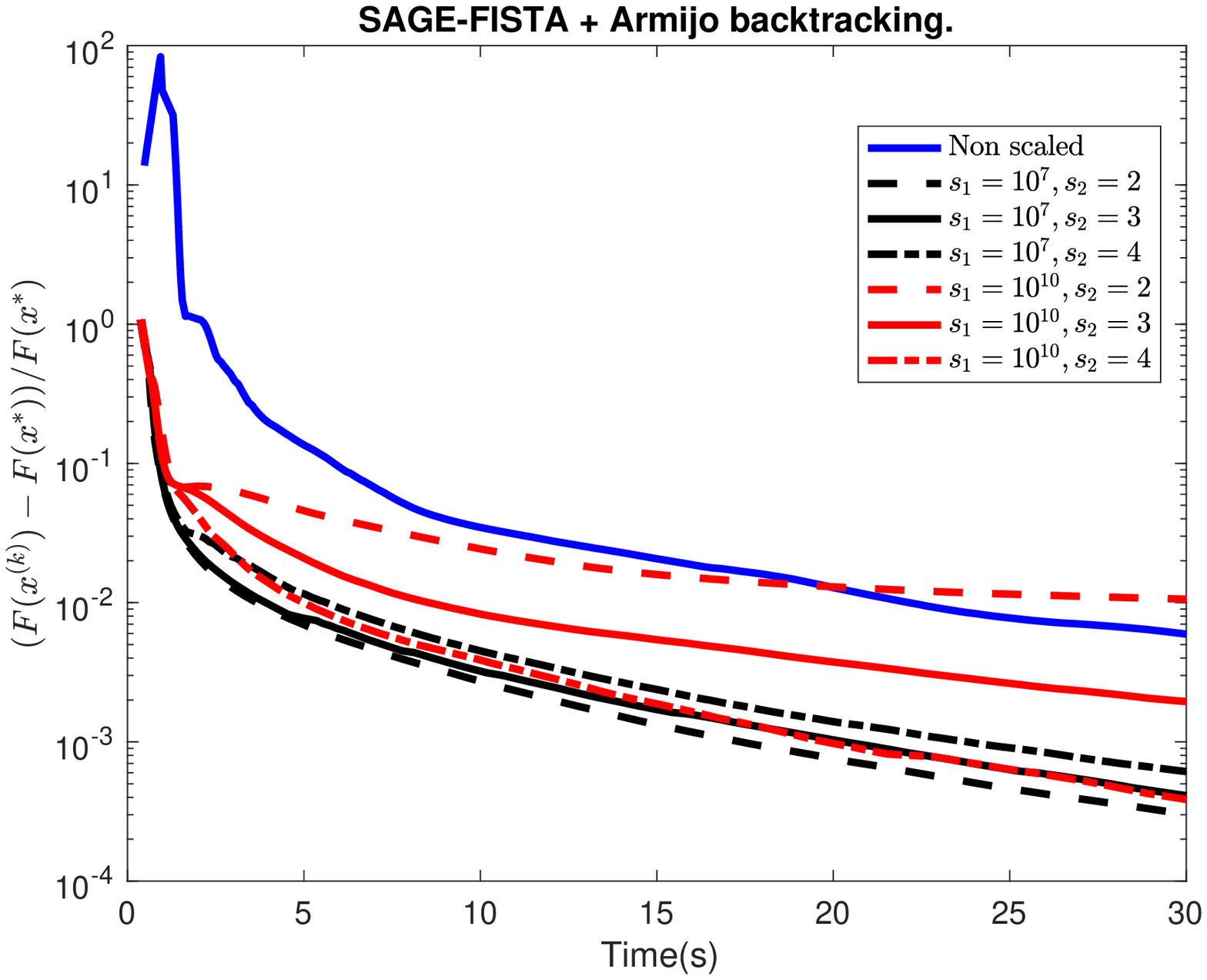} & \includegraphics[height=3.9cm]{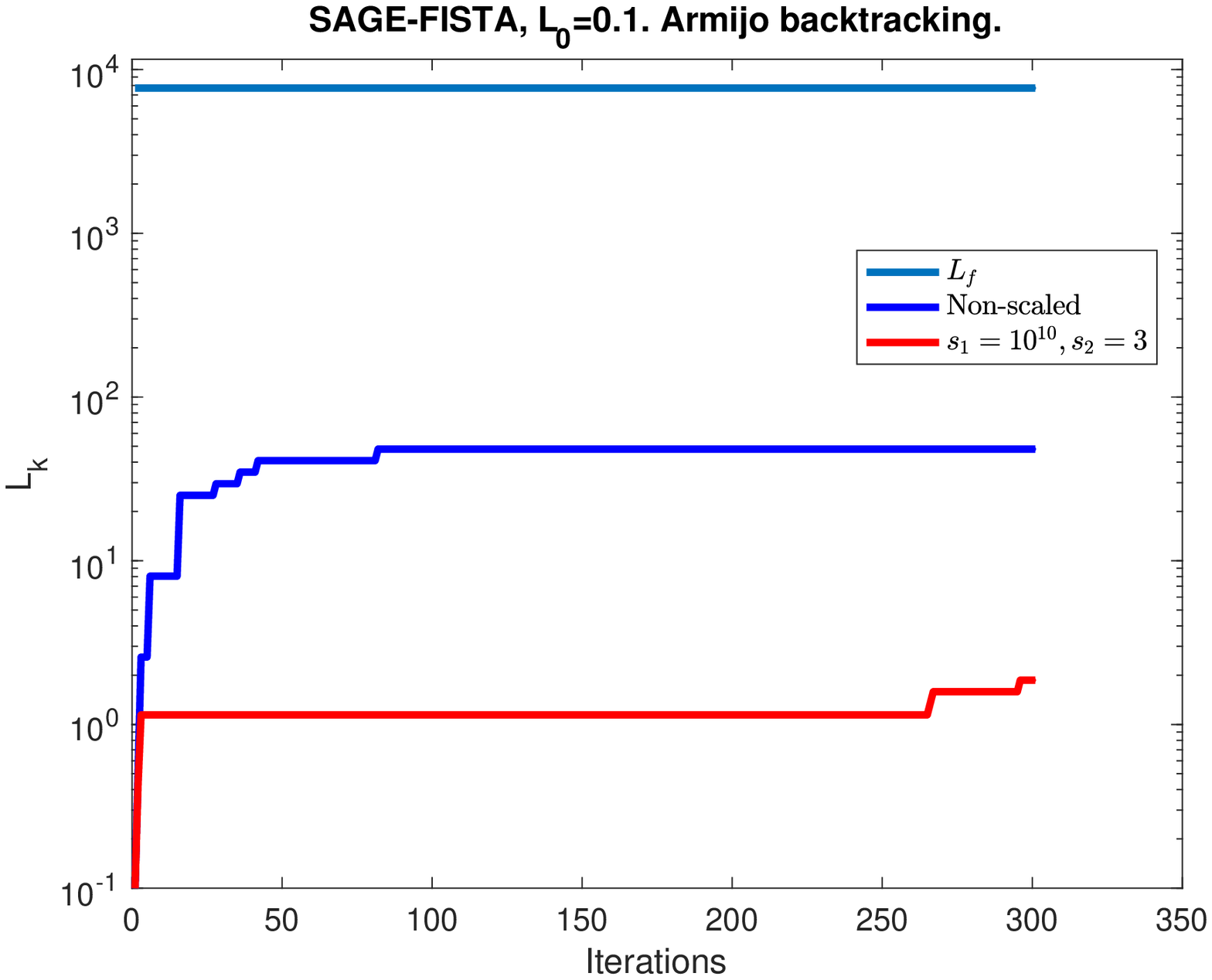} \\
     \includegraphics[height=3.9cm]{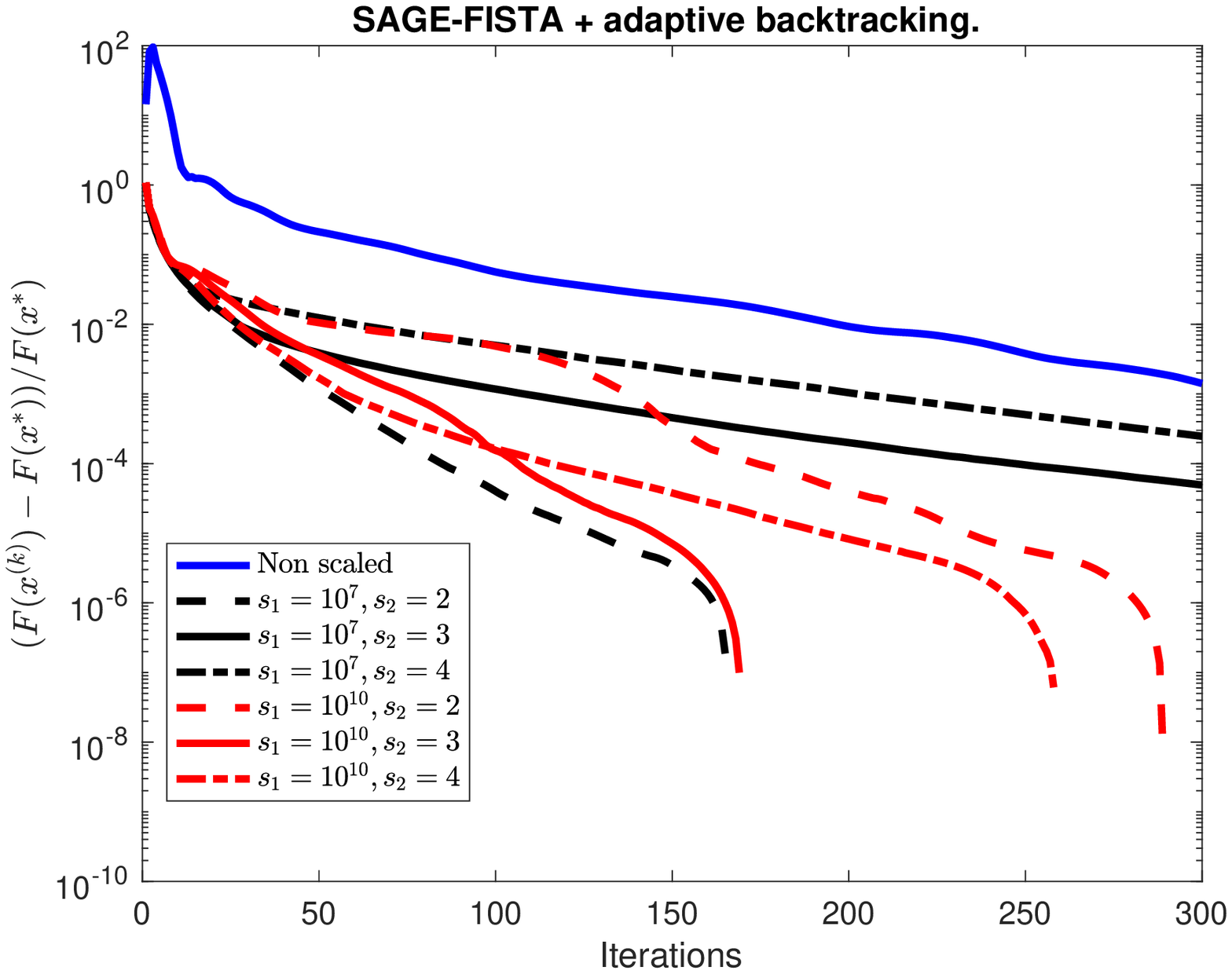} & \includegraphics[height=3.9cm]{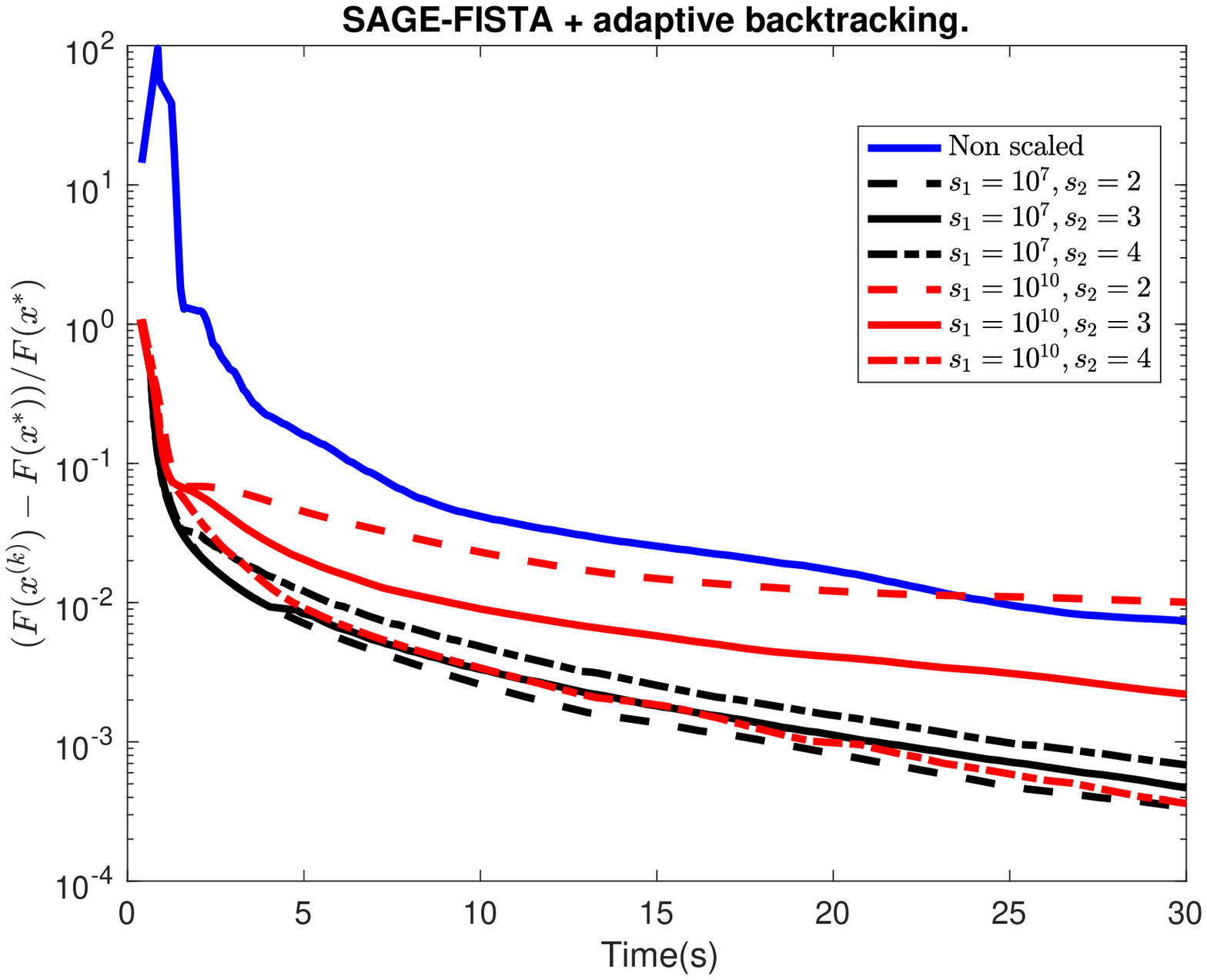} & \includegraphics[height=3.9cm]{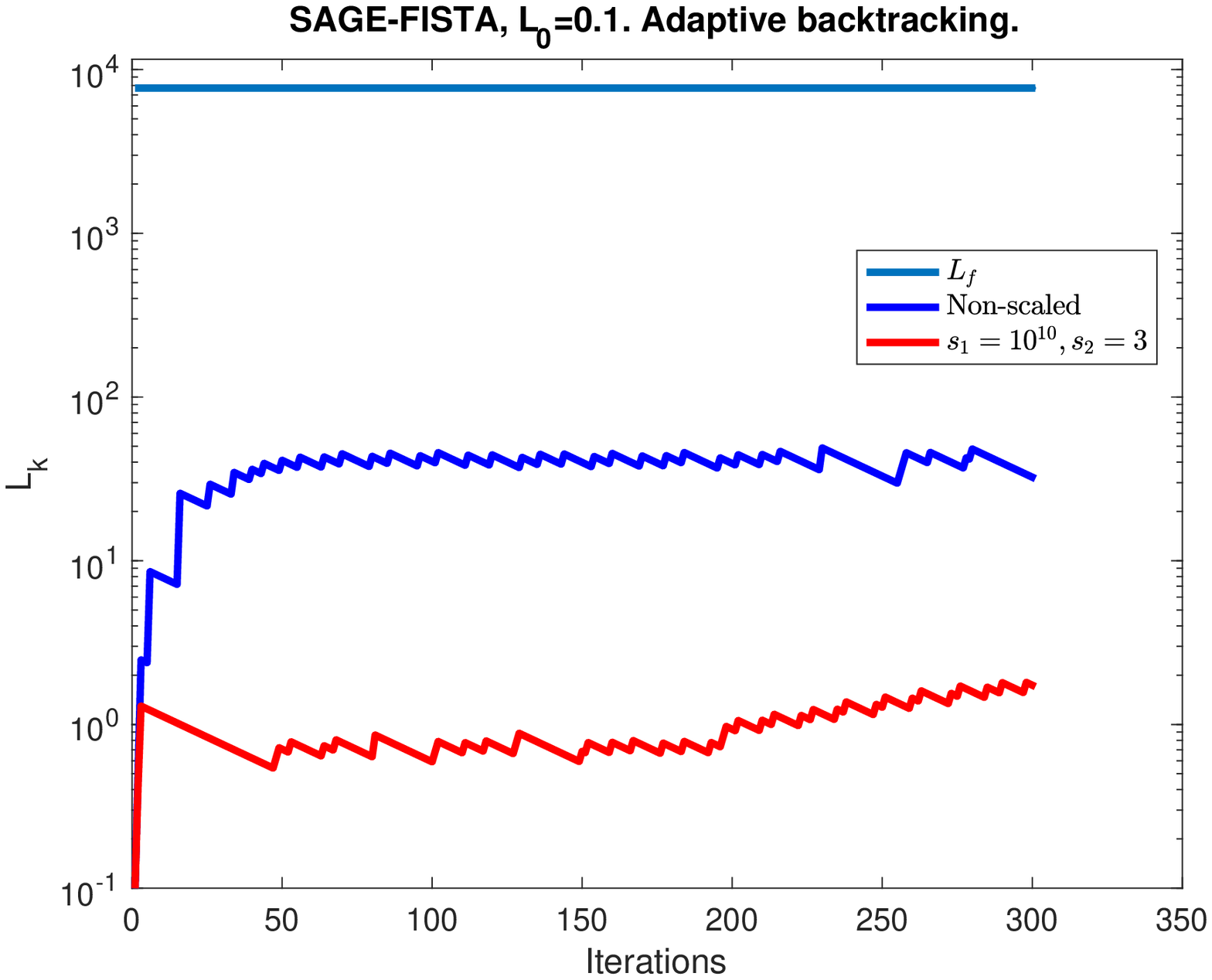}\\
     (a) Relative rates VS it. & (b) Relative rates VS first 30 secs. & (c) $L_k$ estimates.  
\end{tabular}
\caption{Performance of SAGE-FISTA with Armijo backtracking (top row) and adaptive-backtracking (bottom row) for different choices of $s_1$ and $s_2$ for problem \eqref{eq:pb_KL_TVsc} on \texttt{phantom} image.}
    \label{fig:phantom_GFISTA_armijo}
\end{figure}

\begin{figure}[t!]
\hspace{-5mm}
\begin{tabular}{c@{}c@{}c}
     \includegraphics[height=3.9cm]{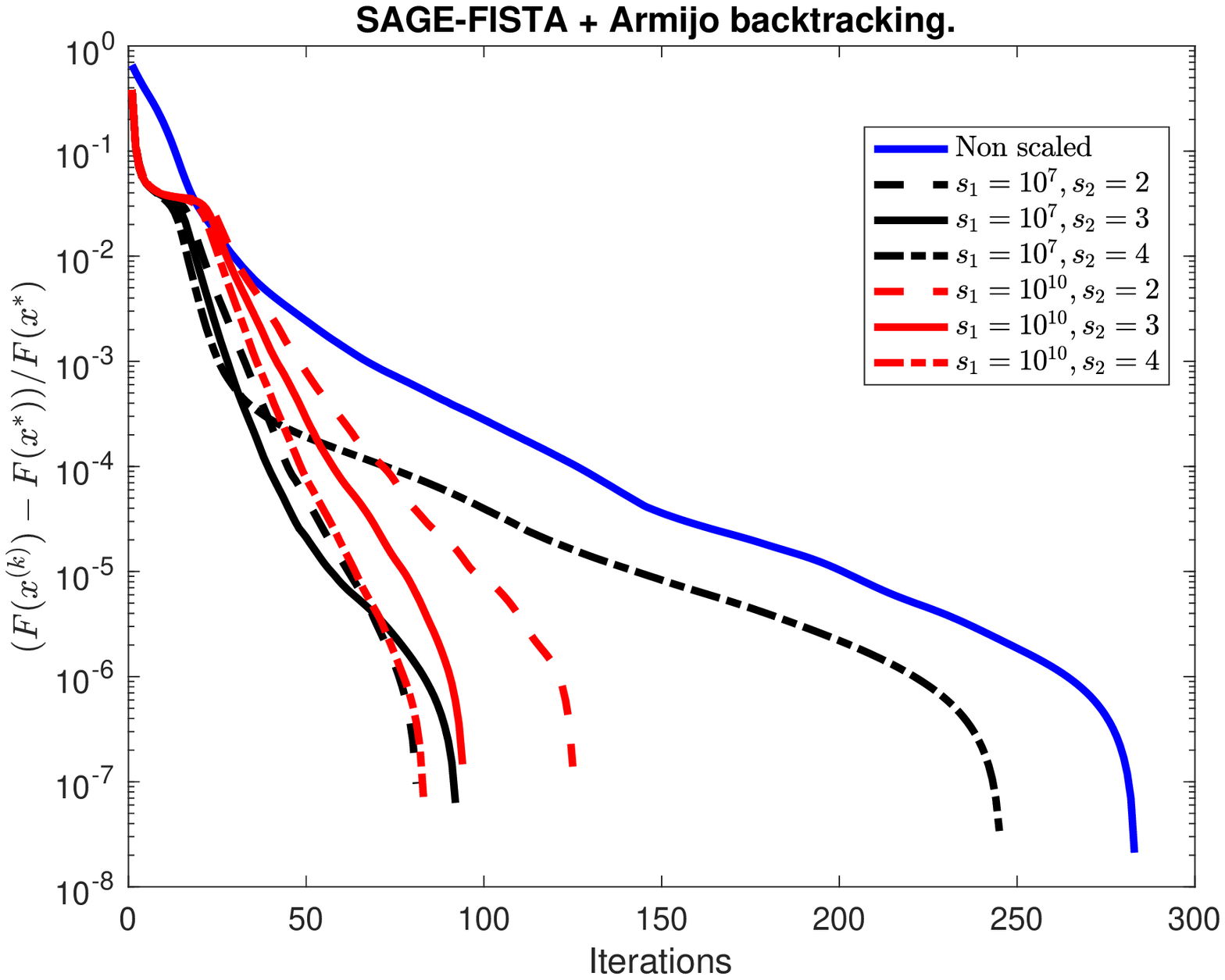} &  \includegraphics[height=3.9cm]{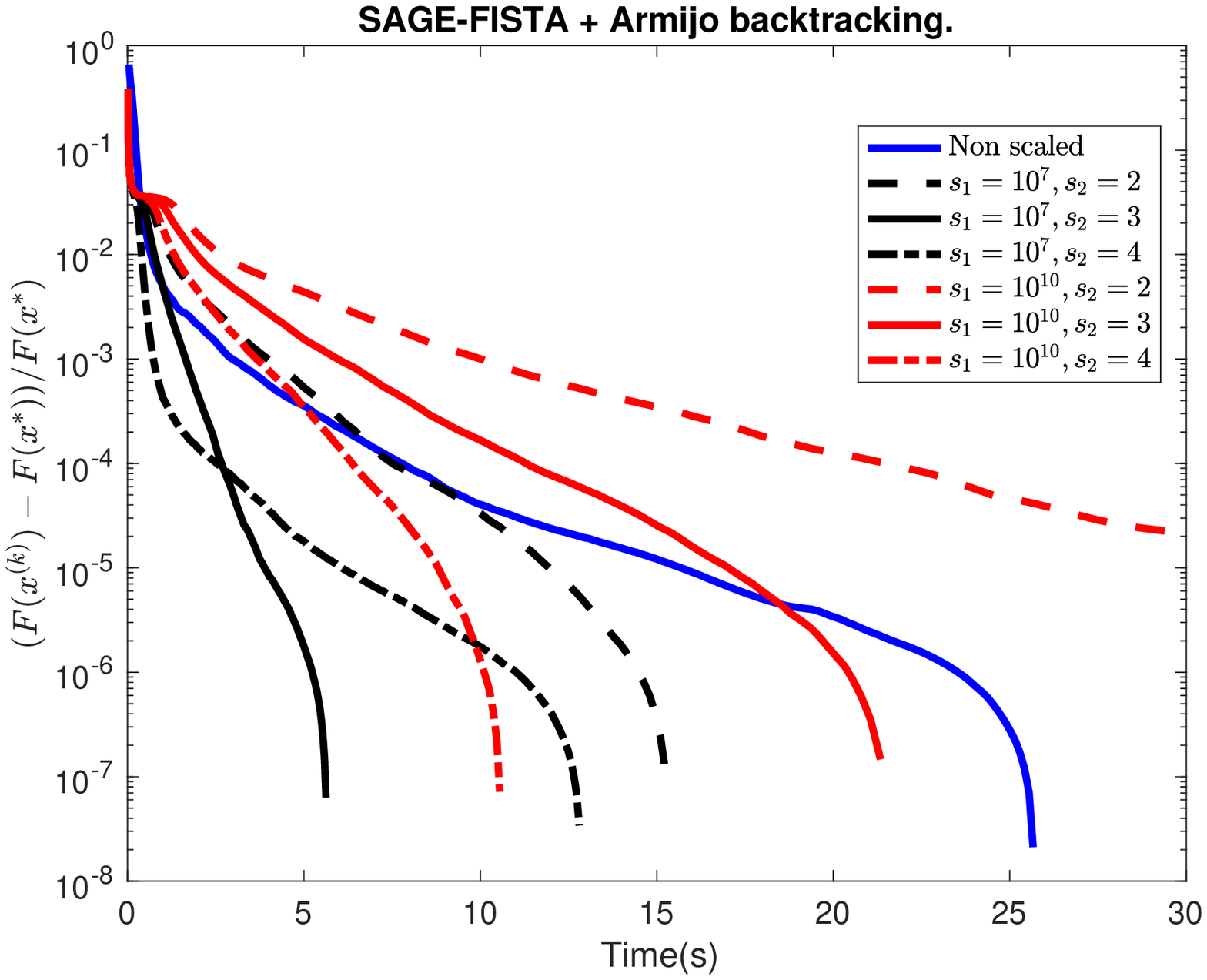} & \includegraphics[height=3.9cm]{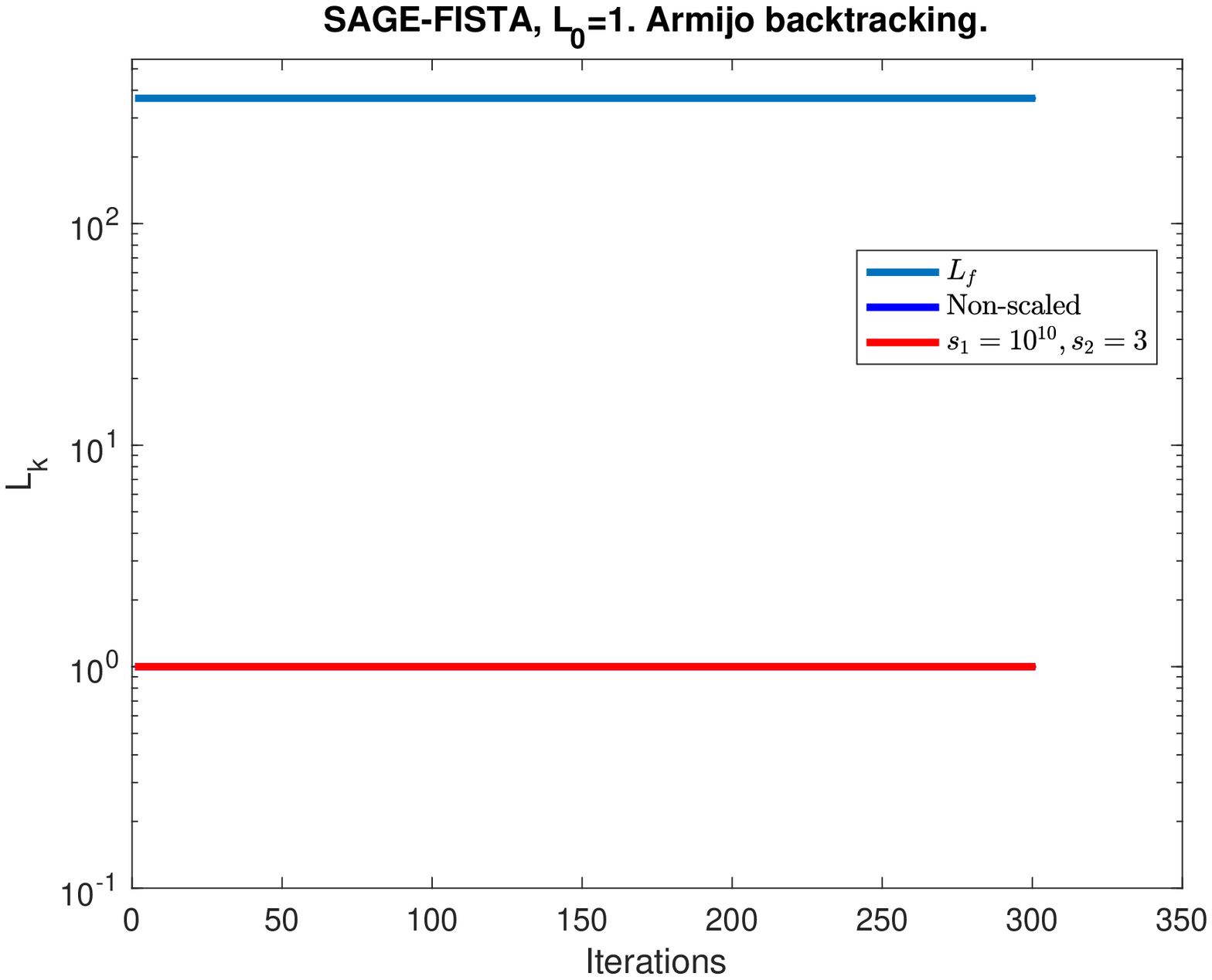} \\
     \includegraphics[height=3.9cm]{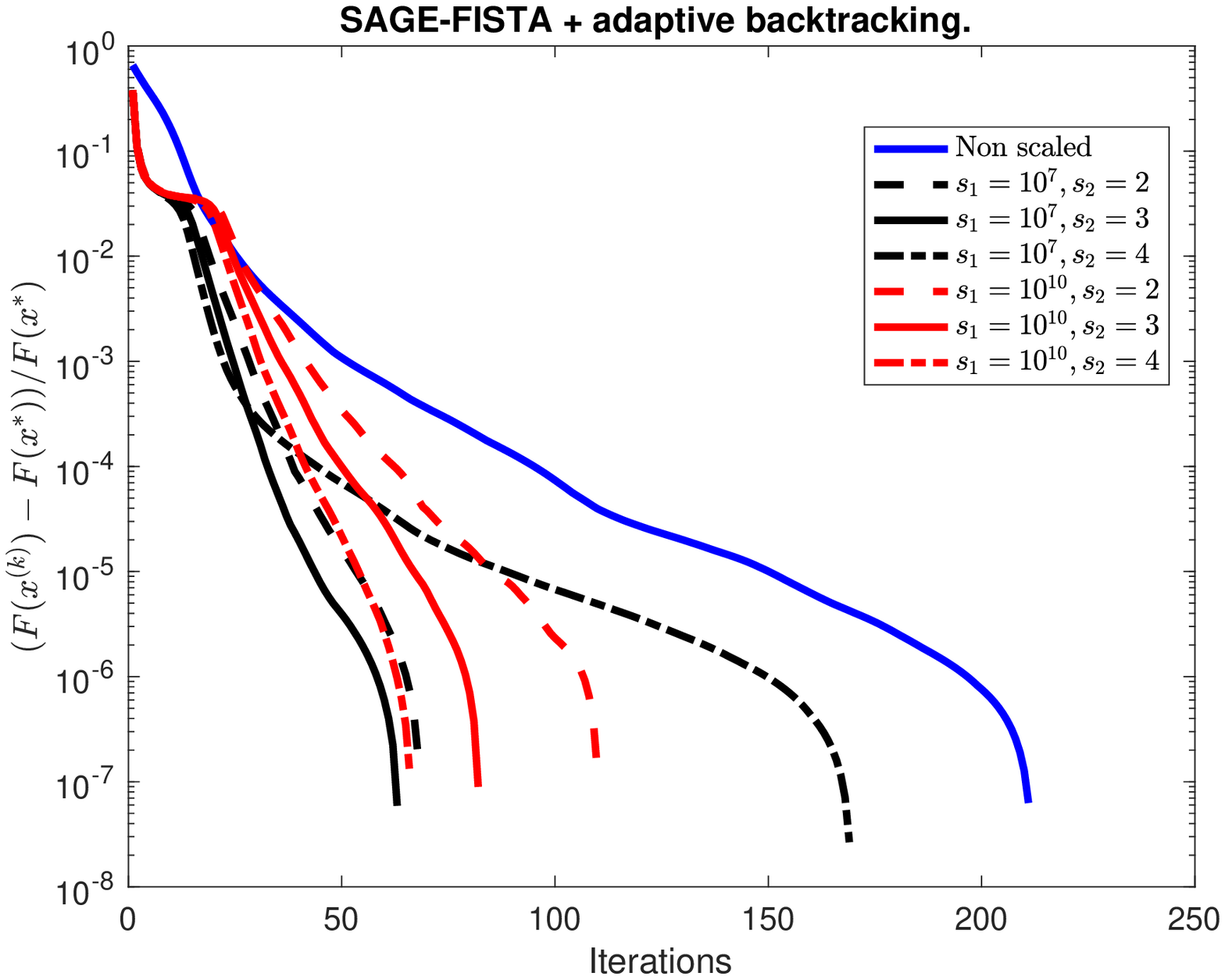} & \includegraphics[height=3.9cm]{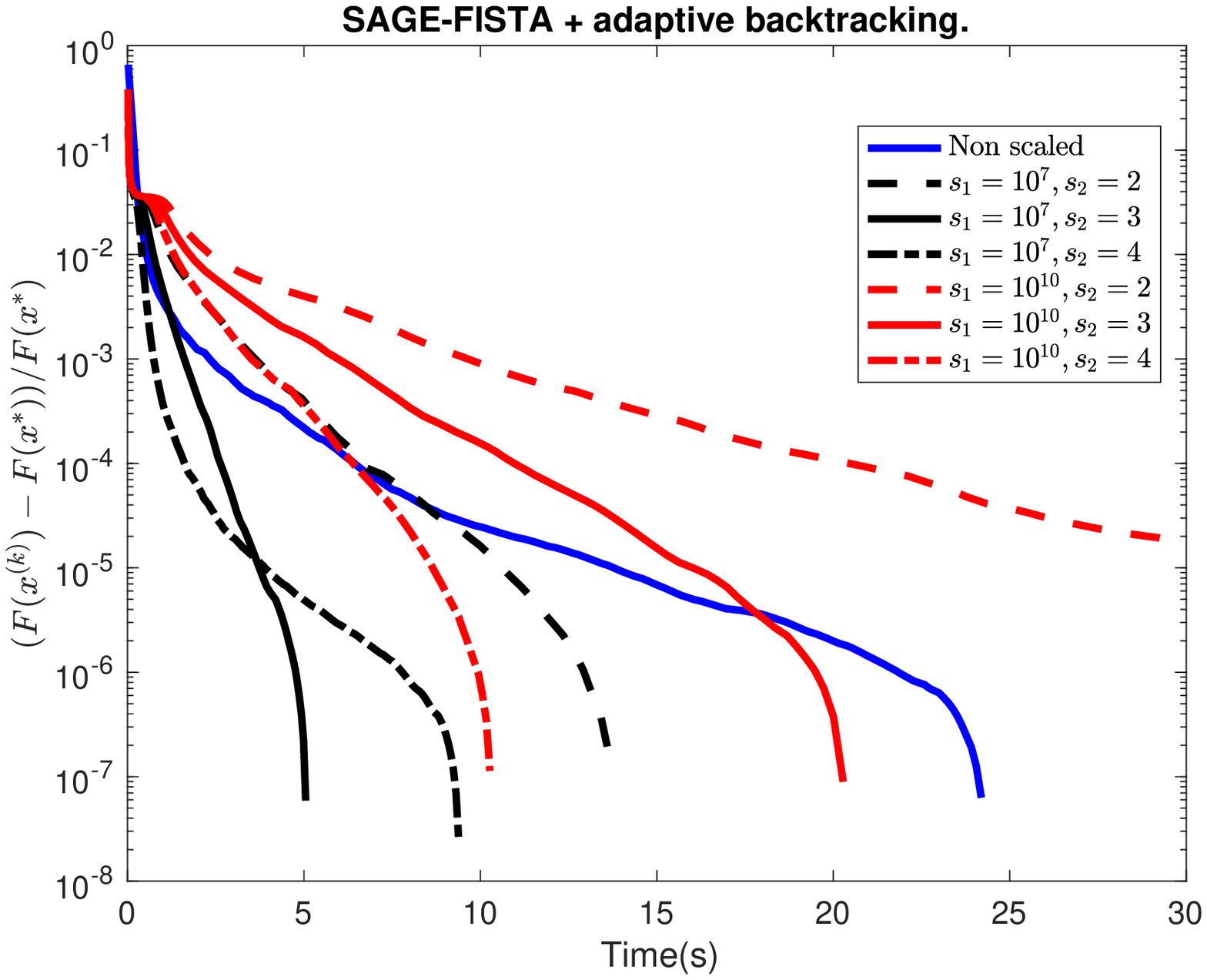} & \includegraphics[height=3.9cm]{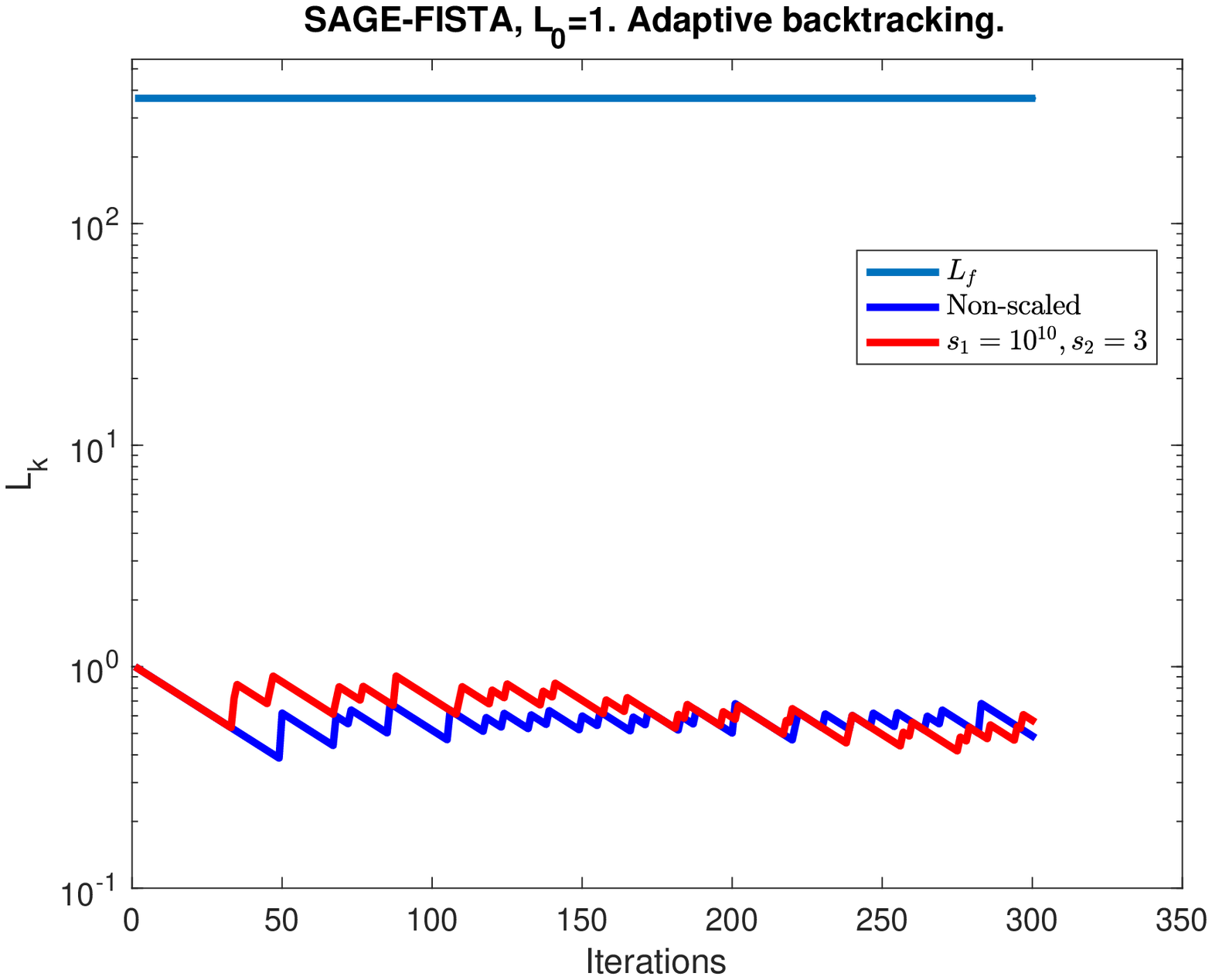}\\
     (a) Relative rates VS it. & (b) Relative rates VS first 30 secs. & (c) $L_k$ estimates.  
\end{tabular}
\caption{Performance of SAGE-FISTA with Armijo backtracking (top row) and adaptive-backtracking (bottom row) for different choices of $s_1$ and $s_2$ for problem \eqref{eq:pb_KL_TVsc} on \texttt{micro} image.}
    \label{fig:micro_GFISTA_armijo}
\end{figure}

\begin{figure}[t!]
\hspace{-5mm}
\begin{tabular}{c@{}c@{}c}
     \includegraphics[height=3.9cm]{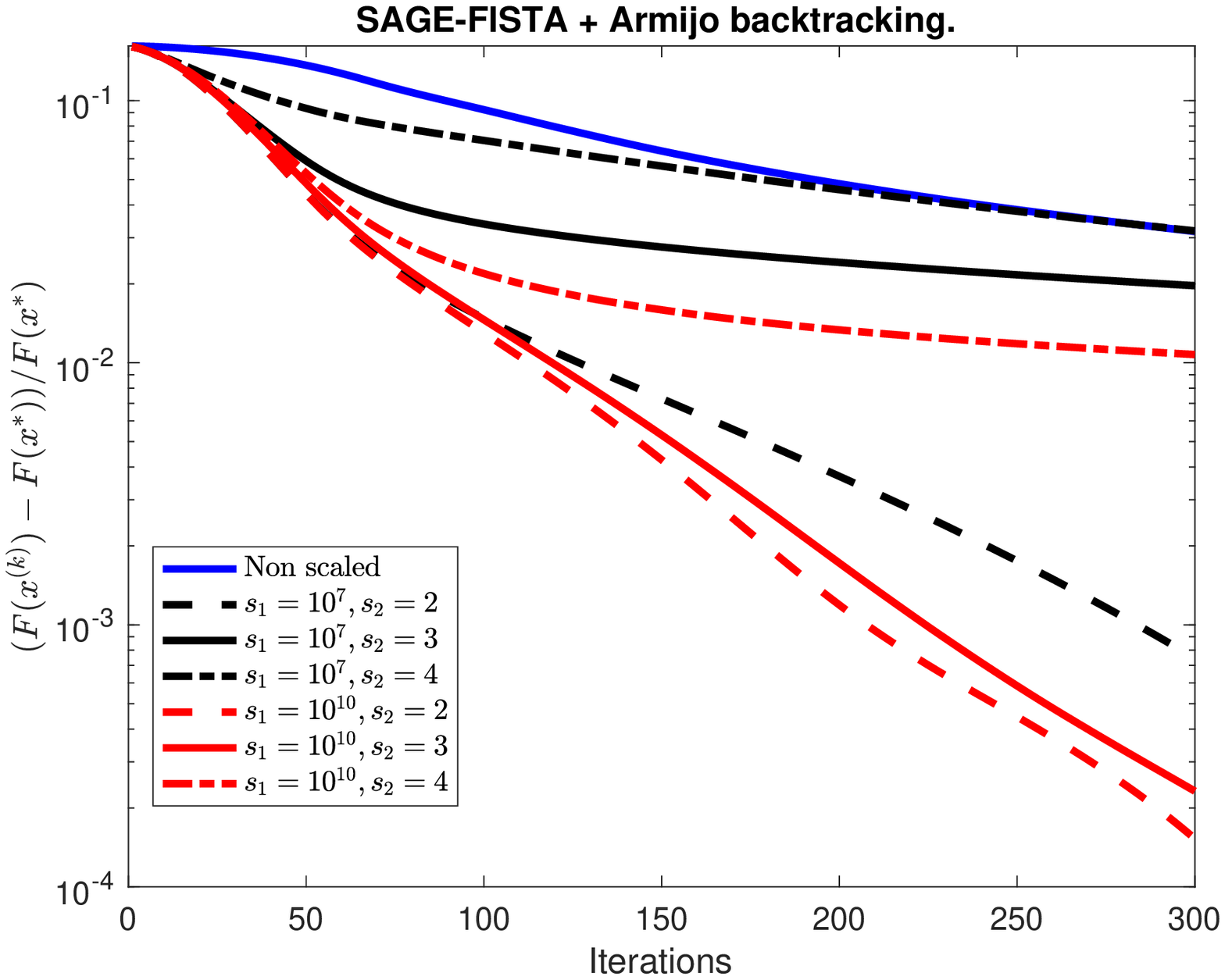} &  \includegraphics[height=3.9cm]{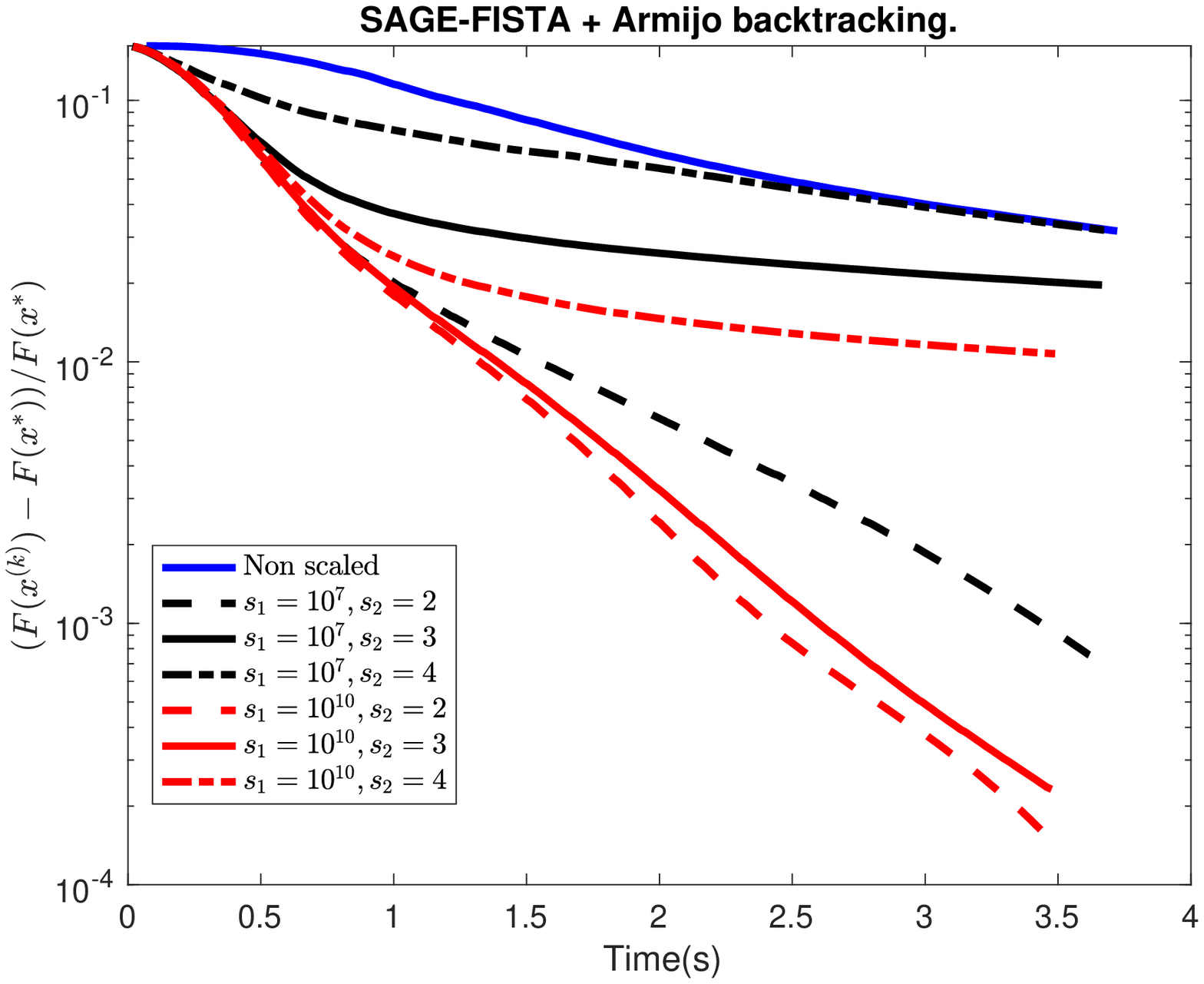} & \includegraphics[height=3.9cm]{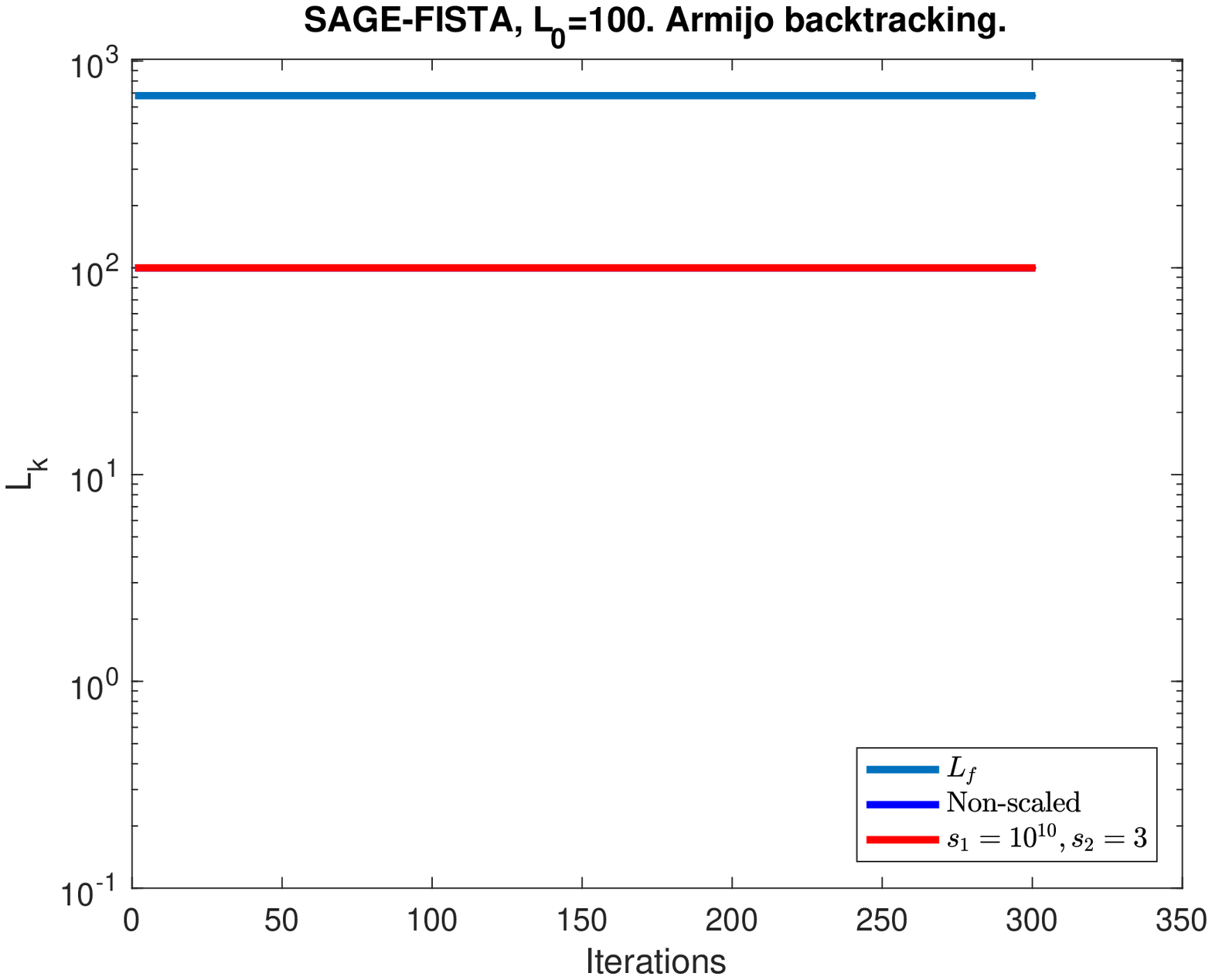} \\
     \includegraphics[height=3.9cm]{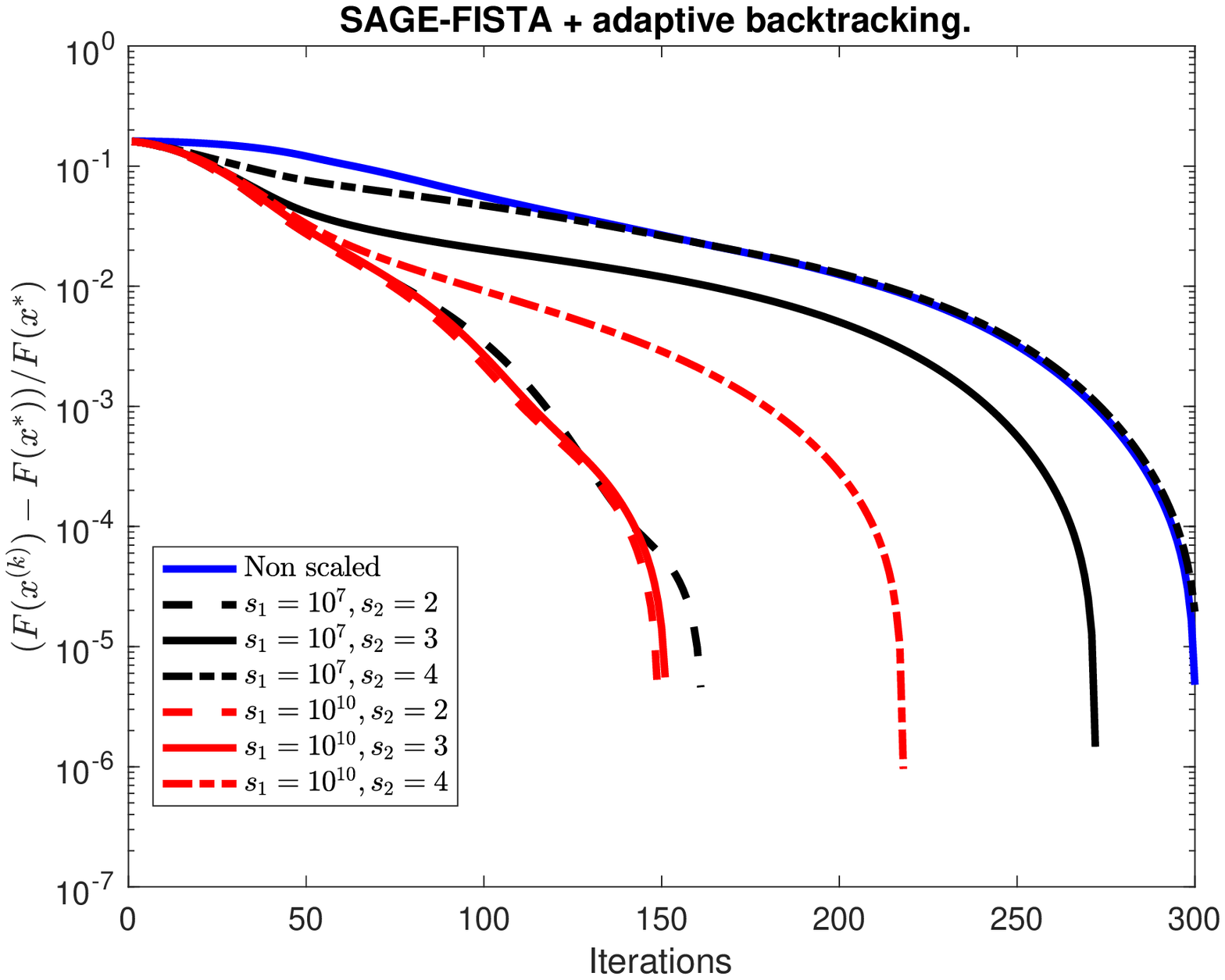} & \includegraphics[height=3.9cm]{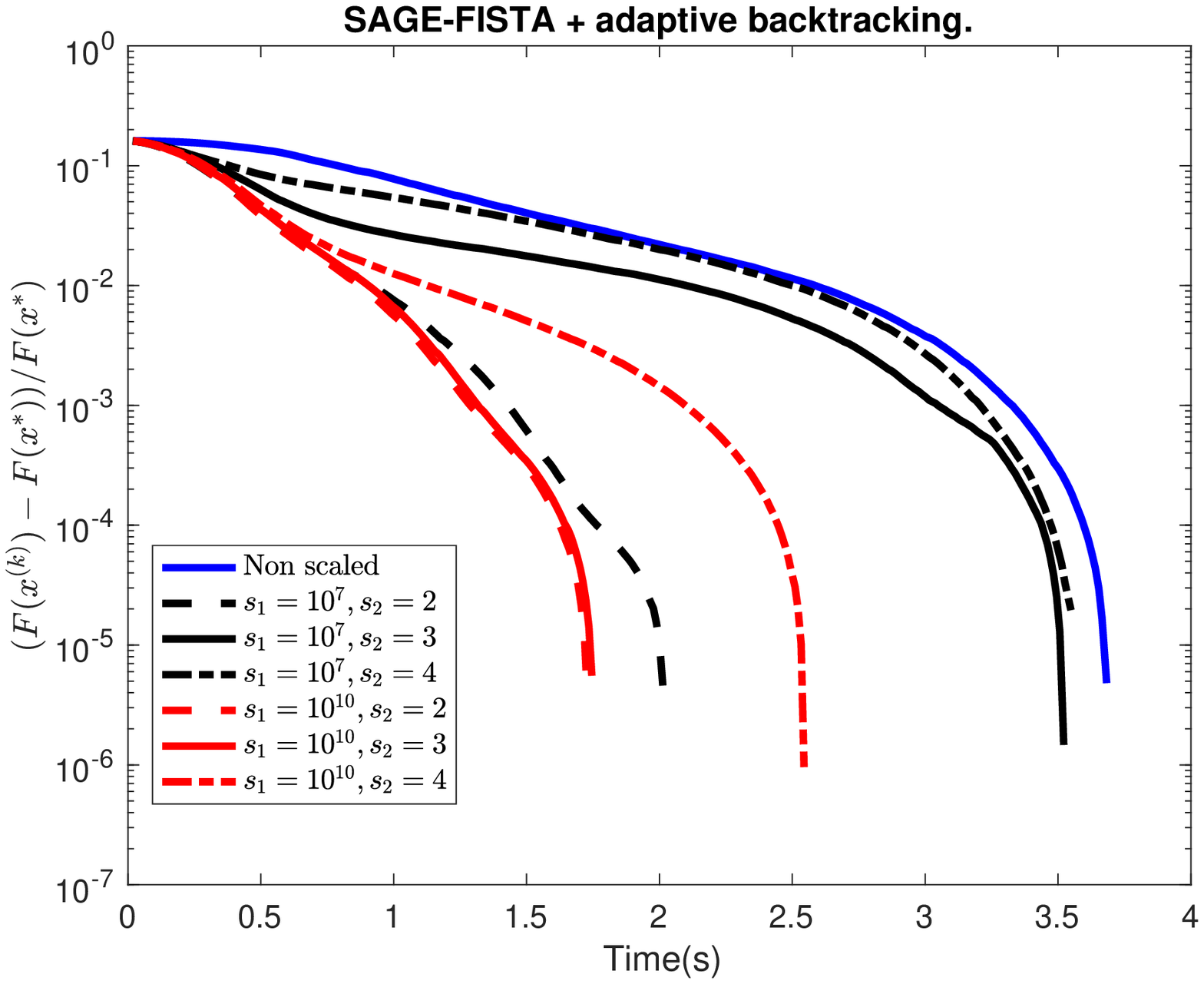} & \includegraphics[height=3.9cm]{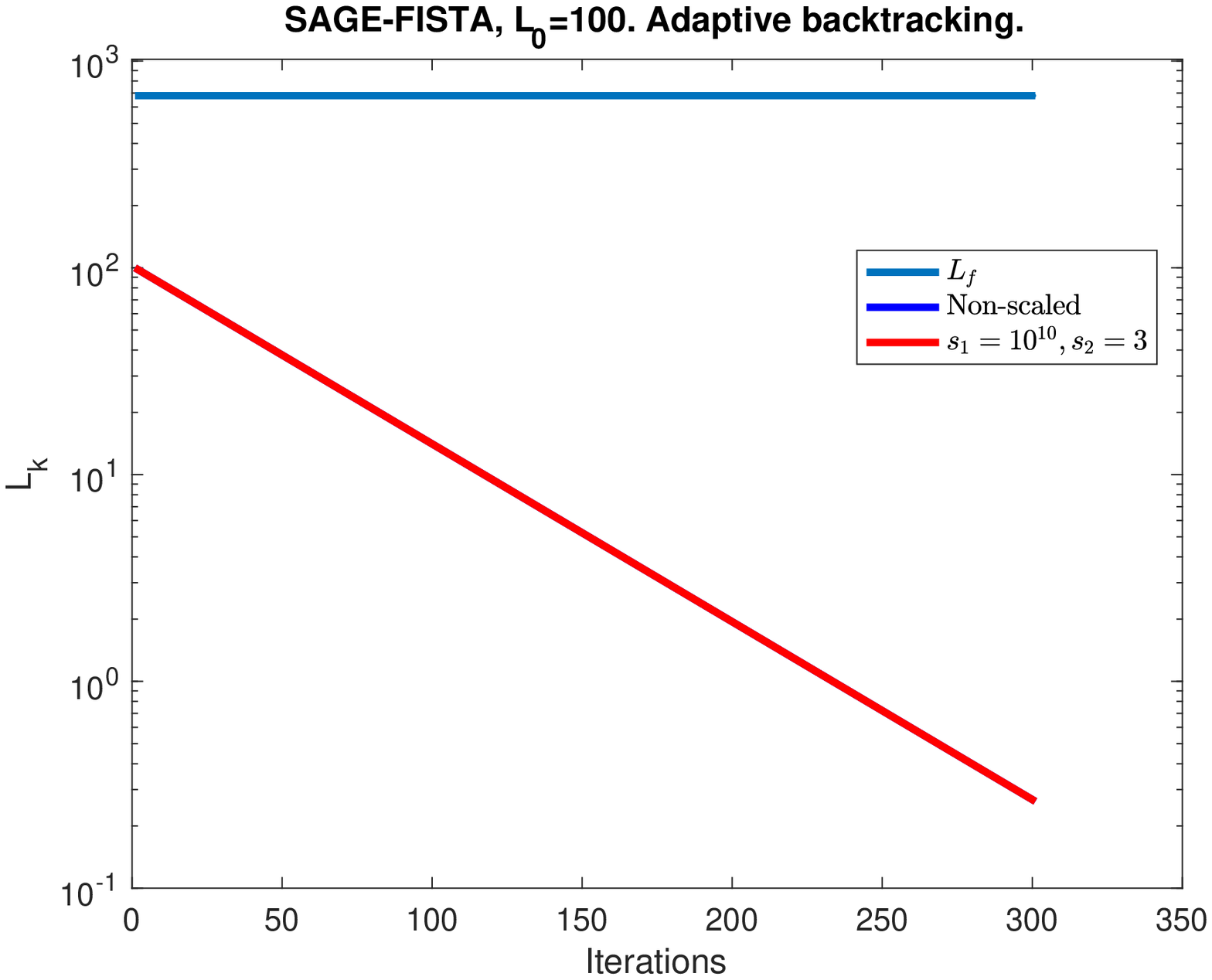}\\
     (a) Relative rates VS it. & (b) Relative rates VS first 30 secs. & (c) $L_k$ estimates.  
\end{tabular}
\caption{Performance of SAGE-FISTA with Armijo backtracking (top row) and adaptive-backtracking (bottom row) for different choices of $s_1$ and $s_2$ for problem \eqref{eq:pb_KL_TVsc} on \texttt{mri} image.}
    \label{fig:mri_GFISTA_armijo}
\end{figure}

\section{Conclusions and outlook}

We studied a Scaled, inexact and Adaptive GEneralized FISTA (named SAGE-FISTA) algorithm for solving (strongly) convex composite optimization problems via inertial forward-backward splitting. In order to take into account possible errors in the computation of the proximal-gradient points, an appropriate notion of inexactness is introduced and combined with the use of suitable variable-metric operators. To avoid the limitations due to the use of monotone Armijo-type backtracking strategies, we considered an adaptive backtracking strategy that allows for local increasing/decreasing of the algorithmic step-size. A linear convergence result in function values is given, and numerical experiments on some exemplar Poisson image restoration problems are reported. The combination of the scaling procedure with the adaptive backtracking strategy allows to compute much more accurate solutions in shorter computational times, which is of particular interest in real-world scenarios where limited computational resources are available. 

Possible generalizations of this work shall consider the use of Bregman distances in the computation of the proximity operator, in the same spirit of  \cite{Bauschke-Bolte-Teboulle-2016}. For the estimation of the strong convexity moduli,  the use of restarting strategies similar to the ones proposed in \cite{ODonoghue2015,FercoqQu2020} should be also explored. Finally, in order to relax the hypothesis on strong convexity, the study of SAGE-FISTA under the proximal Polyak-\L{}ojasiewicz condition \cite{Karimi2016} is envisaged.

\appendix

\section{Computation of perturbed scaled projections} \label{appendix:1}

We report a useful Lemma employed in Section \ref{sec:Pdeb} to define appropriately the primal sequence $\left\{x^{(k,l)}\right\}_{l\in \mathbb{N}}$ in \eqref{eq:primal_sequence_PDEB1} for problem \eqref{eq:pb_KL_TVsc}. We prove the result under the assumption that, at each iteration $k\geq 1 $, the scaling matrix $D_k$ is diagonal and positive definite. Whenever $D_k = \mathcal{I}$, the identity matrix, the result is a standard property of proximal operators, see, e.g., \cite{Combettes-Pesquet-2009}.

\begin{Lemma}  \label{lemma:rescaling}
Let $h:\mathbb{R}^n\to\mathbb{R}\cup \left\{\infty\right\}$ be a proper, convex and lower semicontinuous function, $D\in\mathbb{R}^{n\times n}$ a diagonal positive definite matrix and let $ \varepsilon>0$. Then, defining the function $g: \mathbb{R}^n\to\mathbb{R}\cup \left\{\infty\right\}$ as $g(x):= h(x) + \frac{\varepsilon}{2}\|x\|_2^2$, there holds:
\begin{equation}
    \prox^D_{\tau g}(z) = \prox_{\tau h}^{D+\tau\varepsilon\mathcal{I}} \left( \left(\frac{D}{D+\tau\varepsilon\mathcal{I}}\right) z\right),\qquad \forall z\in\mathbb{R}^n,\quad \tau>0,
\end{equation}
where $\mathcal{I}\in\mathbb{R}^{n\times n}$ is the identity matrix and the fraction symbol is intended element-wise.
\end{Lemma}

\begin{proof}
We have the following chain of equalities holding for all $z\in\mathbb{R}^n$:
\begin{align*}
    \prox^D_{\tau g}(z) &  =\argmin_{y\in\mathbb{R}^n}~ h(y) + \frac{\varepsilon}{2}\|y\|_2^2 + \frac{1}{2\tau}\|y-z\|_D^2 \\
    & =\argmin_{y\in\mathbb{R}^n}~ h(y) + \sum_{i=1}^n \left(\frac{D_{i,i}+\tau\varepsilon}{2\tau}y_i^2 + \frac{D_{i,i}}{2\tau} z_i^2- \frac{D_{i,i}}{\tau}y_i z_i\right) \\
    & =\argmin_{y\in\mathbb{R}^n}~ h(y) + \sum_{i=1}^n \left(   \frac{D_{i,i}+\tau\varepsilon}{2\tau}y_i^2 + \left(\frac{D_{i,i}^2}{2\tau(D_{i,i}+\tau\varepsilon)} + \frac{\varepsilon D_{i,i}}{2(D_{i,i}+\tau\varepsilon)} \right)z_i^2 - \frac{D_{i,i}}{\tau}y_i z_i \right).
\end{align*}
By neglecting the constant term $(\varepsilon D_{i,i}z_i^2)/(2(D_{i,i}+\tau\varepsilon))$, we obtain:
\begin{align*}
\prox^D_{\tau g}(z) & = \argmin_{y\in\mathbb{R}^n}~ h(y) + \sum_{i=1}^n  \frac{D_{i,i}+\tau\varepsilon}{2\tau} \left(y_i^2+\left(\frac{D_{i,i}z_i}{D_{i,i}+\tau\varepsilon}\right)^2 -2 y_i \left(\frac{D_{i,i}z_i}{D_{i,i}+\tau\varepsilon}\right)\right) \\
    & = \argmin_{y\in\mathbb{R}^n}~ h(y) + \frac{1}{2\tau}\left\| y - \left(\frac{D}{D+\tau\varepsilon\mathcal{I}}\right)z\right\|_{D+\tau\varepsilon\mathcal{I}}^2 = \prox_{\tau h}^{D+\tau\varepsilon\mathcal{I}} \left( \left(\frac{D}{D+\tau\varepsilon\mathcal{I}}\right) z\right).
\end{align*}
\end{proof}

\bibliography{biblio_Silvia}{}
\bibliographystyle{plain}
\end{document}